\documentclass[12pt,oneside,a4paper,reqno]{amsart}
\usepackage[utf8]{inputenc}
\usepackage[hmargin=2.5cm,vmargin=2.5cm]{geometry}
\usepackage{microtype}

\usepackage{mathtools}
\usepackage{amssymb}
\usepackage{amsthm}
\usepackage{bbm}
\usepackage[foot]{amsaddr}
\usepackage{aligned-overset}
\usepackage{enumitem}
\usepackage{xcolor}
\usepackage{verbatim}
\usepackage{svg}
\usepackage{tikz}
\usepackage[sortcites,url=false,eprint=false,isbn=false,giveninits=true,maxnames=99,minnames=99,abbreviate=true,date=year]{biblatex}
\usepackage[colorlinks,allcolors=blue]{hyperref}
\hypersetup{breaklinks=true,colorlinks,allcolors=blue,bookmarksopen,
            pdfauthor={Emile Bukieda, Bj\"orn de Rijk}}
\numberwithin{equation}{section}

\newtheorem{theorem}{Theorem}[section]
\newtheorem{lemma}[theorem]{Lemma}

\newtheorem{proposition}[theorem]{Proposition}
\newtheorem{corollary}[theorem]{Corollary}
\theoremstyle{definition}
\newtheorem{definition}[theorem]{Definition}
\theoremstyle{remark}
\newtheorem{remark}[theorem]{Remark}

\DeclareFontFamily{U}{mathx}{\hyphenchar\font45}
\DeclareFontShape{U}{mathx}{m}{n}{
      <5> <6> <7> <8> <9> <10>
      <10.95> <12> <14.4> <17.28> <20.74> <24.88>
      mathx10
      }{}
\DeclareSymbolFont{mathx}{U}{mathx}{m}{n}
\DeclareFontSubstitution{U}{mathx}{m}{n}
\DeclareMathAccent{\widecheck}{0}{mathx}{"71}
\DeclareMathAccent{\wideparen}{0}{mathx}{"75}

\newcommand\restr[2]{{
  \left.\kern-\nulldelimiterspace 
  #1 
  \vphantom{\big|} 
  \right|_{#2} 
  }}

\definecolor{rred}{HTML}{db816f}
\definecolor{bblue}{HTML}{9bacde}
\definecolor{figblue}{HTML}{abc6e1}
\definecolor{figred}{HTML}{db816f}
\definecolor{figorange}{HTML}{FDDA9C}

\definecolor{darksalmon}{RGB}{233, 100, 89}

\newcommand{\C}{\mathbb{C}} 
\newcommand{\R}{\mathbb{R}} 
\newcommand{\Z}{\mathbb{Z}} 
\newcommand{\N}{\mathbb{N}} 
\newcommand{\iu}{\mathrm{i}} 
\newcommand{\eu}{\mathrm{e}} 

\newcommand{\ran}{\textnormal{ran}}
\newcommand{\ds}{\, \textnormal{d}s}
\newcommand{\dz}{\, \textnormal{d}z}
\newcommand{\dr}{\, \textnormal{d}r}

\newcommand{\eps}{\varepsilon}
\newcommand{\Eps}{\mathcal{E}}
\newcommand{\dxi}{\, \textnormal{d}\xi}
\newcommand{\dx}{\, \textnormal{d}x}
\newcommand{\dy}{\, \textnormal{d}y}

\newcommand{\per}{\textnormal{per}}
\newcommand{\El}{\mathcal{L}}
\newcommand{\A}{\mathcal{A}}
\newcommand{\vt}{\tilde{v}}

\title{Nonlinear stability of periodic waves in the Korteweg--de Vries equation under localized perturbations}
\author{Emile Bukieda}
\email{emile.bukieda@kit.edu}
\author{Bj\"orn de Rijk}
\address{Department of Mathematics, Karlsruhe Institute of Technology, Englerstra\ss e 2, 76131 Karlsruhe, Germany}
\email{bjoern.de-rijk@kit.edu}
\thanks{This work is funded by the Deutsche Forschungsgemeinschaft (DFG, German Research Foundation) - Project-ID 258734477 - SFB 1173.}
\begin{document}

\pagenumbering{arabic}
\setcounter{page}{1} 

\begin{abstract}
We investigate the stability and asymptotic behavior of spatially periodic \emph{cnoidal} waves in the Korteweg--de Vries equation subject to localized perturbations. Standard stability arguments in Hamiltonian systems break down in this setting, since localized perturbations preclude a characterization of stable periodic waves as strict minimizers of a suitable energy functional subject to finitely many constraints. As a result, the nonlinear stability of periodic waves under localized perturbations has remained a long-standing open problem in Hamiltonian systems, with previous results only addressing plane waves that can be reduced to constant states by passing to polar coordinates. In this paper, we develop a novel method that resolves this obstruction by combining variational arguments, Floquet--Bloch theory, and Duhamel-based estimates with spatiotemporal modulation. Our framework applies to general periodic waves in Hamiltonian systems with symmetry and reduces the nonlinear stability problem to verifying \emph{diffusive} spectral stability conditions for the second variation of a suitable conserved energy. Applying our approach to cnoidal waves in the Korteweg--de Vries equation, we obtain the first nonlinear stability result for periodic waves in Hamiltonian systems under localized perturbations that cannot be reduced to constant states.
\end{abstract}



\maketitle

\section{Introduction} \label{introduction}

In Hamiltonian systems, periodic waves naturally emerge as coherent structures that persist over time. Notable examples include water and plasma waves, soliton trains in nonlinear optics, and waves in nonlinear oscillator chains; see~\cite{Jeffrey1972weak, miura1968korteweg,schneider2017nonlinear,geyer2025stability} and references therein. Given the physical relevance and structural persistence of periodic waves, understanding their dynamical (or nonlinear) stability is a problem of fundamental importance.

To date, nonlinear stability results for periodic standing or traveling waves in Hamiltonian systems have primarily addressed co-periodic and subharmonic perturbations, whose periods are integer multiples of that of the underlying wave; see~\cite{geyer2025stability,benzonigavage2016coperiodic,Benzoni2013Stability} for an overview of results and abstract stability criteria. However, a longstanding open question concerns their stability under localized perturbations: a natural scenario in many applications. For instance, in hydrodynamics~\cite{clamond2003cnoidal,osborne1994shallow}, periodic waves naturally arise on spatially extended domains, making their interaction with localized perturbations of primary interest~\cite{osborne1994shallow,hoefer2023kdv}. Notably, the question of nonlinear stability of periodic waves in the Korteweg--de Vries equation under non-periodic perturbations was already posed by Benjamin in his 1974 lecture series on nonlinear wave motion~\cite{benjamin1974lectures}.

To the best of the authors' knowledge, there are only two nonlinear stability results for periodic waves in Hamiltonian systems under localized perturbations in the current literature. These results concern plane waves in the complex Klein--Gordon and nonlinear Schr\"odinger equations~\cite{bukieda2025orbital,zhidkov2001korteweg}, which are monochromatic periodic waves with a single nonzero Fourier mode. The analyses in~\cite{bukieda2025orbital,zhidkov2001korteweg} rely heavily on the system's gauge symmetry, which permits a reduction of the plane waves to constant states by passing to polar coordinates. 

In this paper, we develop a novel stability theory that applies to general periodic waves in Hamiltonian systems, including ones that cannot be reduced to constant states. We apply our method to the Korteweg--de Vries (KdV) equation
\begin{align} \label{KdV}
    u_t + u u_x + u_{xxx} = 0, \qquad x,t \in \R, \, u(x,t) \in \R,
\end{align}
which is a classical model~\cite{Korteweg1895,Boussinesq1877} for unidirectional wave propagation in various physical settings, such as hydrodynamics, elastic rods, and lattice waves. As such, the KdV equation has been the subject of ongoing interest in both the mathematical and physical literature; see, for instance, the surveys~\cite{benjamin1974lectures,Jeffrey1972weak,miura1968korteweg,schneider2017nonlinear} and references therein. Our focus here is on the KdV equation as a paradigmatic example of a Hamiltonian system that admits a family of periodic waves with infinitely many nonzero Fourier modes, for which no reduction to constant states is available. Specifically, the periodic traveling-wave solutions of~\eqref{KdV} form a four-parameter family given by
\begin{align} \label{e:cnoidal_form}
u(x,t) = h + 12 \kappa^2 \Eps \, \mathrm{cn}^2(\kappa(x-ct-x_0),\Eps),
\end{align}
where $\mathrm{cn}(\cdot,\Eps)$ denotes the Jacobi elliptic cosine function with modulus $\Eps \in (0,1)$, $h \in \R$ represents the height of the wave, $\kappa \in (0,\infty)$ corresponds to the rescaled wave number, $x_0$ is the initial phase offset, and 
\begin{align} \label{e:wavespeed}
c = 4\kappa^2(2\Eps - 1) + h 
\end{align}
is the associated wave speed; see Appendix~\ref{appendixPhPA} for more details. In analogy to the shape of their profiles, these periodic waves are commonly referred to as \emph{cnoidal waves}. In the long-wavelength limit $\Eps \to 1$, they converge to the well-known soliton solution $u(x,t) = h + 12 \kappa^2 \, \mathrm{sech}^2(\kappa(x-ct-x_0))$ of~\eqref{KdV} propagating with wave speed $c = 4\kappa^2 + h$.

Cnoidal waves play a key role as fundamental coherent structures in the KdV equation and beyond. In particular, \emph{cnoidal wave theory} provides a unifying framework for describing nonlinear wave phenomena, including water and plasma waves~\cite{osborne1994shallow, clamond2003cnoidal,nezlin1993physics}. Moreover, dispersive shock waves in the KdV equation and more sophisticated fluid and geophysical models can be understood as modulations of cnoidal waves~\cite{el2016dispersive}.

Orbital stability of periodic traveling wave solutions to~\eqref{KdV} has been established with respect to co-periodic perturbations~\cite{mckean1977stability,angulo2006stability,johnson_nonlinear_2010} and subharmonic perturbations~\cite{deconinck2010orbital}. This entails that, for any $n \in \N$, solutions to~\eqref{KdV} in $H^1_\per(0,n\ell)$ starting sufficiently close to an $\ell$-periodic cnoidal wave remain close to its symmetry orbit consisting of spatial translates. In addition, considering the linearization of~\eqref{KdV} about a cnoidal wave on $H^s(\R)$, it has been shown for all $s \in \N_0$ that the spectrum is confined to the imaginary axis~\cite{bottman2009kdv} and that the associated $C_0$-group is bounded~\cite{rodrigues2018linear}. These spectral and linear stability results strongly point toward the nonlinear stability of cnoidal waves under localized perturbations. The general stability theory that we develop in this paper answers this question in the affirmative. It reduces the orbital stability problem to a \emph{diffusive spectral stability} problem for the second variation of a conserved energy, without relying on integrability of the underlying equation; see Remark~\ref{rem:diff_energies_in_other_systems}. Before stating our main result, we illustrate this principle in the setting of a general Hamiltonian system with symmetry. 

\subsection{Principle of diffusive spectral stability} \label{sec:intro_principle} To demonstrate the broad applicability of our method, we consider a general Hamiltonian evolution equation on the line endowed with a one-parameter symmetry group. We formulate diffusive spectral stability conditions on the second variation of a conserved quantity under which our method yields a \emph{modulational} notion of nonlinear stability. This, in turn, leads to orbital stability of periodic waves with respect to localized perturbations, locally uniform in space.

To treat both localized perturbations and $\ell$-periodic waves within a unified framework, we consider Hamiltonian systems that can be posed on both $L^2(\R)$ and $L^2_{\mathrm{per}}(0,\ell)$. Thus, let $X$ denote either of these Hilbert spaces, and let
\begin{align}\label{e:Hamilton}
u_t = \mathcal{J} \mathcal{H}'(u)
\end{align}
be an abstract Hamiltonian evolution equation on $X$, where $\mathcal{J}$ is a skew-symmetric (pseudo-)differential operator and the Hamiltonian $\mathcal{H} \colon Y \to \R$ is a smooth nonlinear functional defined on a dense subspace $Y$ of $X$. Suppose that~\eqref{e:Hamilton} is invariant under the action of a strongly continuous one-parameter symmetry group $T(\phi)$, $\phi \in \R$, acting on $X$ such that $T(\phi)$ leaves $Y$ invariant. In particular, we require that $\mathcal{H}(T(\phi) u) = \mathcal{H}(u)$ for all $u \in Y$ and $\phi \in \R$. For a foundational treatment of Hamiltonian evolution equations and further background, we refer to~\cite{Kuksin_2000_analysis,Olver_1980_Hamiltonian}.

Let $w$ be an stationary solution to~\eqref{e:Hamilton} in $L^2_\per(0,\ell)$. Take a smooth nonlinear energy functional $E \colon Z \to \R$, defined on a dense subspace $Z$ of $X$, that is conserved by the flow of~\eqref{e:Hamilton}. Suppose that the $\ell$-periodic wave $w$ is a critical point of $E$, i.e.~$E'(w) = 0$. Since $w$ is a stationary solution to~\eqref{e:Hamilton}, a typical choice for this conserved quantity is $E = \mathcal{H}$, although, as Hamiltonian systems can admit multiple conserved quantities, alternative choices may also be possible.

In order to formulate spectral stability conditions on the second variation $E''(w)$, we write it as a bilinear form $\smash{E''(w)[v_1,v_2] = \langle \A v_1,v_2\rangle_{L^2(\R)}}$ and characterize the spectrum of the associated operator $\A$ on $L^2(\R)$. Note that $\A$ is typically a self-adjoint differential operator with $\ell$-periodic coefficients and so, by Floquet--Bloch theory~\cite{Reed1978Methods,Scarpellini_1999_Stability,Kuchment_1993_operator,Lewin_2024_spectral}, its $L^2(\R)$-spectrum is determined by the family of Bloch operators
\begin{align} \label{e:def_Bloch}
\A(\xi) u = \eu^{-\iu \xi \cdot} \A\big[\eu^{\iu \xi \cdot } u\big], \qquad \xi \in \C,
\end{align}
acting on $L^2_\per(0,\ell)$. For each fixed $\xi \in \C$, the operator $\A(\xi)$ has compact resolvent, and hence its spectrum consists of isolated eigenvalues with finite algebraic multiplicities only. In contrast, the spectrum of $\A$ on $L^2(\R)$ is purely essential and arises from the union of the spectra of $\A(\xi)$ for real values of $\xi$. That is, we have the spectral relation
\begin{align} \label{e:spectral_rel}
\sigma_{L^2(\R)}(\A) = \bigcup_{\xi \in [-\frac{\pi}{\ell},\frac{\pi}{\ell})} \sigma_{L^2_{\per}(0,\ell)}(\A(\xi)).
\end{align}
Since~\eqref{e:Hamilton} is invariant under the action of the symmetry group $T(\phi)$, $0$ is an eigenvalue of $\A(0)$ with eigenfunction $A_0 w$, where $A_0$ is the generator of the group $T(\phi)$ acting on $L^2_\per(0,\ell)$. Consequently, the $L^2(\R)$-spectrum of $\A$ must touch the imaginary axis at the origin, implying that $\A$ is positive \emph{semi}-definite at best, and hence $-\A$ is only marginally spectrally stable. The notion of \emph{diffusive spectral stability}, see Figure~\ref{fig1}, captures precisely this property, while additionally requiring that the touching at the origin is nondegenerate.

\begin{definition} \label{def:diffusive_spec}
The operator $-\A$ is \emph{diffusively spectrally stable}, if there exists a constant $\theta_0 > 0$ such that:
\begin{itemize}
    \item[(i)] For all $\xi \in [-\frac\pi\ell,\frac\pi\ell)$ we have
\begin{align}    
\inf \sigma_{L^2_\per(0,\ell)}(\A(\xi)) \geq \theta_0 \xi^2. \label{e:coercivity_diffusive}\end{align}
    \item[(ii)] $0$ is a simple eigenvalue of $\A(0)$ with eigenfunction $A_0w$.
\end{itemize}
\end{definition}

\begin{figure}[ht!]
\begin{minipage}{0.45\textwidth}
\centering
\includegraphics[width=\textwidth]{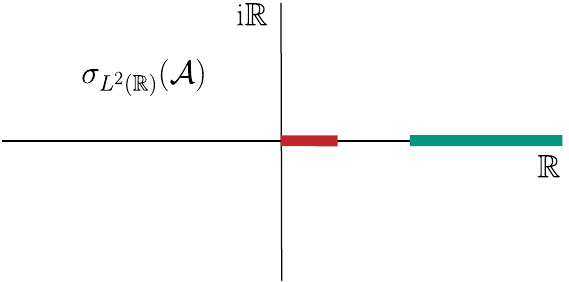}
\end{minipage} \hfill
\begin{minipage}{0.45\textwidth}
\centering
\includegraphics[width=\textwidth]{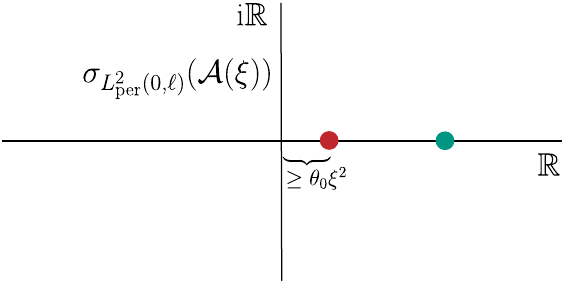}
\end{minipage}
\caption{\small Left: $L^2(\R)$-spectrum of an operator $\mathcal{A}$, where $-\mathcal{A}$ is diffusively spectrally stable in the sense of Definition~\ref{def:diffusive_spec}. Right: the $L^2_{\text{per}}(0,\ell)$-spectrum of the corresponding Bloch operator $\mathcal{A}(\xi)$ for some fixed $\xi \in [-\frac\pi\ell,\frac\pi\ell)$.} \label{fig1}
\end{figure}

Assuming that $-\A$ is diffusively spectrally stable, the simple eigenvalue $0$ of $\A(0)$ can be locally continued in $\xi$ using standard perturbation theory~\cite{kato_perturbation_2013}, resulting in an eigenvalue $\lambda_c(\xi)$ of $\A(\xi)$ of smallest real part, which is analytic in $\xi$ and obeys the expansion
\begin{align} \label{e:expansion}
\lambda_c(\xi) = d \xi^2 + \mathcal{O}(\xi^4), \qquad |\xi| \ll 1,
\end{align}
with diffusivity coefficient $d > 0$. Using the Floquet--Bloch decomposition~\eqref{e:spectral_rel}, the critical $L^2(\R)$-spectrum of $\A$ can be parameterized by $\lambda_c(\xi)$, which behaves as the Fourier symbol of the diffusion operator $-d\partial_x^2$ for $|\xi| \ll 1$. Although $-d\partial_x^2$ is only positive semi-definite, the identity
\begin{align} \label{e:coercivity_diffusion}
\langle -du'', u \rangle_{L^2(\R)} = d\|u'\|^2_{L^2(\R)}, \qquad u \in H^2(\R),
\end{align}
demonstrates that coercivity can be recovered at the cost of a derivative. The analysis in this paper shows that diffusive spectral stability of $-\A$ yields a similar \emph{coercivity estimate at the cost of a derivative}. This property ultimately allows us to establish nonlinear stability of $w$ against localized perturbations, despite the lack of positive definiteness of $\A$ on any finite-codimensional subspace.

\subsubsection{\texorpdfstring{\bf Modulational nonlinear stability}{Modulational nonlinear stability}} We illustrate how coercivity at the cost of a derivative shapes the resulting notion of nonlinear stability. To this end, we provide a heuristic description of the dynamics of localized perturbations under diffusive spectral stability of $-\A$. The nondegenerate spectral touching of $\A$ at the origin with associated Bloch eigenfunction $A_0w$ shows that the neutral directions of the conserved energy $E$ near its critical point $w$ are associated with the symmetry orbit $T(\phi)w$. Moreover, the presence of infinitely many modes $\lambda_c(\xi)$ of $\A$ accumulating at $0$ indicates that localized perturbations can excite the symmetry in a nonuniform manner across the spatial domain, giving rise to a \emph{spatiotemporal modulation} of the wave of the form
\begin{align*}
w_{\text{mod}}(x,t) = T(\varphi(x,t))w(x).
\end{align*}
To capture this behavior, we represent solutions to~\eqref{e:Hamilton} as
\begin{align} \label{e:decomp}
u(x,t) = T(\varphi(x,t))w(x) + z(x,t),
\end{align}
where $z(x,t)$ is a small remainder term and the modulation function $\varphi(x,t)$ tracks the leading-order neutral dynamics. We emphasize that spatiotemporal modulation of periodic waves subject to localized perturbations has been observed and analyzed in a broad class of dissipative systems, cf.~\cite{doelman_dynamics_2009,johnson2011nonlinear,sandstede2012diffusive,johnson_nonlinear_2010,johnson_2014_behavior}. The analysis in this paper shows that this phenomenon persists in the current conservative setting.

The key advantage of the decomposition~\eqref{e:decomp} is that the loss of positive definiteness can be confined to $\varphi$, allowing for a coercivity estimate on $z$. As noted above, diffusive spectral stability also yields a coercivity estimate for the \emph{derivative} $\varphi_x$. Consequently, both $\varphi_x(t)$ and $z(t)$ are expected to remain small in $H^s(\R)$ over time for some $s \in \N$, which implies orbital stability of $w$, albeit only \emph{locally uniform in space}. More precisely, for every $R,\varepsilon > 0$ there exists $\delta > 0$ such that, for any initial perturbation $v_0 \in H^s(\R)$ with $\|v_0\|_{H^s(\R)} \leq \delta$, there exists a smooth modulation function $\varphi \colon \R \times [0,\infty) \to \R$ such that the solution $u(t)$ to~\eqref{e:Hamilton} with $u(0) = w + v_0$ obeys:
\begin{itemize}
    \item[(i)] \emph{(Modulational stability).} For all $t \geq 0$ we have
    \begin{align} \label{e:mod_stab_gen}
    \|u(t) - T(\varphi(t))w\|_{H^s(\R)} + \|\varphi_x(t)\|_{H^s(\R)} \leq \varepsilon;
    \end{align}
    \item[(ii)] \emph{(Local orbital stability).} For all $x_* \in \R$ and $t \geq 0$ we have
    \begin{align*}
    \inf_{\phi \in \R} \|u(t) - T(\phi)w\|_{H^s(x_*-R,x_*+R)} \leq \varepsilon.
    \end{align*}  
\end{itemize}
The estimate~\eqref{e:mod_stab_gen} shows that the perturbed solution $u(t)$ remains close to the modulated wave $T(\varphi(t))w$, where the modulation function $\varphi(\cdot,t)$ varies only slowly, so that, locally, the resulting spatial deformation of the wave remains small. We emphasize that this modulational notion of nonlinear stability reflects the impact of derivative loss: while global-in-time smallness is obtained for $z(t)$ and $\varphi_x(t)$, no direct control is available for the $L^2(\R)$-norm of $\varphi(t)$ itself. In fact, $\|\varphi(t)\|_{L^2(\R)}$ can grow over time; see the discussion below Theorem~\ref{t:main_result} for details. 

\subsubsection{\texorpdfstring{\bf Diffusive spectral stability in the KdV equation}{Diffusive spectral stability in the KdV equation}} \label{sec:diffusive_KdV_intro} Let $u_{\mathrm{tw}}(x,t) = w(x-ct)$ be a traveling-wave solution to the KdV equation~\eqref{KdV} with wave speed $c \in \R$ and $\ell$-periodic profile $w \colon \R \to \R$. We identify a conserved quantity for which $w$ is a critical point and whose second variation corresponds to a diffusively spectrally stable operator. Setting $\mathcal{J} = \partial_x$ and defining $\mathcal{H} \colon H^1(\R) \to \R$ by 
\begin{align*}
\mathcal{H}(u) = \int_\R \frac12 u_x^2 - \frac16 u^3 + \frac12 c u^2 + c_1 u \dx,
\end{align*}
where $c_1 \in \R$ is a constant, the Hamiltonian system~\eqref{e:Hamilton} becomes
\begin{align} \label{cKdV}
    u_t - cu_x + uu_x + u_{xxx} = 0.
\end{align}
Equation~\eqref{cKdV} corresponds to the KdV equation~\eqref{KdV} written in the frame moving with speed $c$. Its symmetry group is the translational group $T(\phi)u = u(\cdot - \phi)$, $\phi \in \R$ with generator $\mathcal{J} \colon H^1(\R) \to L^2(\R)$. 

Since $w$ is a stationary solution of~\eqref{cKdV}, we may fix $c_1 \in \R$ such that $w$ solves the associated Euler--Lagrange equation 
\begin{align} \label{e:Euler_Lagrange_1}
  c w - \frac12 w^2 - w'' + c_1 = 0,
\end{align}
that is $\mathcal{H}'(w) = 0$. The second variation $\mathcal{H}''(w)$ corresponds to a second-order $\ell$-periodic differential operator $A \colon H^2(\R) \to L^2(\R)$, given by
\begin{align} \label{e:defL1}
A = -\partial_x^2+c-w,
\end{align}
whose $\smash{L^2_\per(0,\ell)}$-kernel is spanned by $\mathcal{J}w = w'$. Consequently, the linearization of~\eqref{cKdV} about $w$ is given by $\El = \mathcal{J}A$. Since $w'$ has at least two zeros on a periodicity interval, a standard Sturm--Liouville argument~\cite{zettl_sturm_2005} implies that $-A$ is indefinite as an operator on $L^2(\R)$ and therefore not diffusively spectrally stable.

However, a conserved quantity with positive semi-definite second variation has been constructed in~\cite{deconinck2010orbital}, based on ideas developed in~\cite{maddocks1993stability}. This energy is given by 
\begin{align} \label{e:def_energy}
E_\gamma \colon H^2(\R) \to \R, \qquad  E_\gamma(u) = \mathcal{H}_2(u) + \gamma \mathcal{H}(u)
\end{align}
with 
\begin{align*}
    \mathcal{H}_2(u) = \int_\R \frac12 u_{xx}^2 - \frac{5}{6}uu_x^2 + \frac{5}{72} u^4 - \frac16 c_1 u^2 - \frac12 c^2 u^2 + c_2 u \dx,
\end{align*}
constant $c_2 \in \R$, and free parameter $\gamma \in \R$. The Hamiltonian system $u_t = \mathcal{J} \mathcal{H}_2'(u)$ corresponds to the fifth-order PDE
\begin{align} \label{e:KdV2}
u_t - \frac13 c_1 u_x -  c^2 u_x + \frac5{6} u^2 u_x + \frac{10}3 u_x u_{xx} + \frac53 u u_{xxx} + u_{xxxxx} = 0.
\end{align}
Equation~\eqref{e:KdV2} is the second nontrivial member of the \emph{KdV hierarchy}, which is an infinite sequence of mutually commuting Hamiltonian flows that was introduced in~\cite{miura1968korteweg,Gardner_1974_Korteweg,ablowitz_solitons_1981,lax_integrals_1968}. The KdV equation~\eqref{KdV} itself constitutes the first nontrivial equation in the hierarchy. The hierarchy is generated through the \emph{Lenard recursion scheme} associated with the Hamiltonian structure. The first Lenard recursion relation takes the form
\begin{align} \label{e:Lenard}
\mathcal{J}\mathcal{H}_2'(u) = \mathcal{J}_2(u) \mathcal{H}'(u),
\end{align}
where $\mathcal{J}_2(y)$ is the skew-symmetric differential operator
\begin{align} \label{e:defJ2}
\mathcal{J}_2(y) = \partial_x^3 + \frac23 y \partial_x + c \partial_x + \frac13 y'.
\end{align}
Since $w$ satisfies $\mathcal{H}'(w) = 0$, it follows from~\eqref{e:Lenard} that $w$ is a stationary solution to~\eqref{e:KdV2}. Hence, we may fix $c_2 \in \R$ such that $w$ solves the associated Euler--Lagrange equation
\begin{align} \label{e:Euler_Lagrange_2}
- \frac13 c_1 w -  c^2 w + \frac5{18} w^3 + \frac{5}{6} (w')^2+\frac{5}{3}w w'' + w'''' + c_2 = 0,
\end{align}
that is $\mathcal{H}_2'(w) = 0$. Consequently, $w$ is a critical point of the energy $E_\gamma$ for each $\gamma \in \R$. We note that, using the explicit cnoidal-wave representation~\eqref{e:cnoidal_form} of the periodic traveling wave $u_{\mathrm{tw}}$, the integration constants $c_{1,2} \in \R$ in the Euler--Lagrange equations~\eqref{e:Euler_Lagrange_1} and~\eqref{e:Euler_Lagrange_2} may be expressed in terms of the elliptic modulus $\Eps$, the height $h$, and the rescaled wave number $\kappa$; see Appendix~\ref{appendixPhPA}.

The analysis in~\cite{deconinck2010orbital} demonstrates that there exists a nonempty open interval $I \subset \R$ such that the second variation $E_\gamma''(w)$, which corresponds to the differential operator $\A_\gamma \colon H^4(\R) \to L^2(\R)$ given by
\begin{align} \label{e:defLgamma}
\A_\gamma = \partial_x^4 + \frac{5}{3}w\partial_x^2+\frac{5}{3}w'\partial_x+\frac{5}{3}w''+\frac{5}{6}w^2-\frac13 c_1 - c^2 +\gamma(-\partial_x^2+c-w),
\end{align}
is positive semi-definite for all $\gamma \in I$. In fact, a refined spectral analysis of the Bloch operators $\A_\gamma(\xi)$ shows that $-\A_\gamma$ is diffusively spectrally stable for all $\gamma \in I$, which serves as the key ingredient for our nonlinear stability analysis.

\begin{theorem}[Diffusive spectral stability] \label{thm:diffusive_intro}
Let $u_{\text{tw}}(x,t) = w(x-ct)$ be a traveling-wave solution to the KdV equation~\eqref{KdV} with wave speed $c \in \R$ and profile $w\colon \R \to \R$ of fundamental period $\ell > 0$. Fix $c_{1,2} \in \R$ such that $w$ satisfies the Euler--Lagrange equations~\eqref{e:Euler_Lagrange_1} and~\eqref{e:Euler_Lagrange_2}. Then, there exists a nonempty open interval $I \subset \R$ such that for each $\gamma \in I$ the operator $-\A_\gamma$ is diffusively spectrally stable with $\ker(\A_\gamma(0)) = \textnormal{Sp}\{w'\}$. 
\end{theorem}

The proof of Theorem~\ref{thm:diffusive_intro}, which we delegate to Appendix~\ref{appendixSpectral}, exploits that the Bloch eigenfunctions of $\mathcal{J}\A_\gamma$ coincide with those of the linearization $\El = \mathcal{J} A$ of~\eqref{cKdV} about $w$ by the commuting-flow property. Consequently, the action of the Bloch operator $\A_\gamma(\xi)$ can be represented in terms of biorthogonal families of eigenfunctions of $\El(\xi)$ and its adjoint. These families of eigenfunctions form unconditional bases~\cite{rodrigues2018linear} and can be completely characterized via the \emph{squared-eigenfunction connection} arising from the Lax-pair formulation of the KdV equation~\cite{ablowitz_solitons_1981,bottman2009kdv}. This ultimately leads to the coercivity estimate~\eqref{e:coercivity_diffusive}.

\begin{remark} \label{rem:diff_energies_in_other_systems}
We stress that the existence of an appropriate conserved energy $E$, whose second variation corresponds to a diffusively spectrally stable operator, is not a specific feature of the KdV equation, nor of integrable systems in general. In particular, whenever a periodic wave $w$ is a \emph{ground state} of the Hamiltonian $\mathcal{H}$, diffusive spectral stability of its second variation is generic. Representative examples in non-integrable Hamiltonian systems with symmetry include periodic \emph{dnoidal} waves in the defocusing Gross--Pitaevskii equation with periodic potential~\cite{Haragus_2008_spectra}, as well as plane waves in the complex Klein--Gordon equation~\cite{bukieda2025orbital}. 
\end{remark}

\subsection{Main result} \label{sec:main_result}
We state our main result, which establishes nonlinear modulational stability of periodic traveling waves in the KdV equation against localized perturbations.

\begin{theorem}[Nonlinear modulational stability] \label{t:main_result}
Let $u_{\text{\upshape tw}}(x,t) = w(x-ct)$ be a periodic traveling-wave solution to the KdV equation~\eqref{KdV} with wave speed $c \in \R$ and profile $w\colon \R \to \R$ of fundamental period $\ell > 0$. Fix $k \in \N_0$. Then, there exist constants $M,\delta > 0$ such that, whenever $v_0 \in H^3(\R)$ satisfies
\begin{align*}E_0 \coloneqq\|v_0\|_{H^2(\R)} \leq \delta,\end{align*}
there exist a global classical solution
\begin{align} \label{e:regtooth}
u \in C\big([0,\infty),H^3(\R) \oplus H^3_\per(0,\ell)\big) \cap C^1\big([0,\infty),L^2(\R) \oplus L^2_\per(0,\ell)\big)
\end{align}
to~\eqref{KdV} with initial condition $u(0) = w+v_0$ and a smooth modulation function $\varphi \colon \R \times [0,\infty) \to \R$ such that
\begin{align} \label{e:modulational_KdV}
\|u(t) - w(\cdot - c t + \varphi(\cdot,t))\|_{H^2(\R)} + \|\nabla \varphi(t)\|_{H^k(\R)} + \frac{\|\varphi(t)\|_{L^2(\R)}}{1+t} \leq ME_0
\end{align}
for all $t \geq 0$, where $\nabla \varphi(t) = (\varphi_x(t), \varphi_t(t))^\top$ denotes the spacetime gradient. Moreover, we have
\begin{align} \label{e:orbital_KdV}
\inf_{\phi \in \R} \|u(t) - w(\cdot + \phi)\|_{H^2(x_* - R,x_* + R)} \leq M R E_0 
\end{align}
for each $t \geq 0$, $x_* \in \R$, and $R \geq 1$.
\end{theorem}

Theorem~\ref{t:main_result} shows that the perturbed solution $u(x,t)$ to~\eqref{KdV} remains close to the \emph{modulated} periodic traveling wave $w(x - ct + \varphi(x,t))$, where the modulation function $\varphi(x,t)$ depends on both space and time. This stands in contrast to previous stability results for co-periodic and subharmonic perturbations~\cite{angulo2007nonlinear,johnson_nonlinear_2010,mckean1977stability,deconinck2010orbital}, where, for each $t \geq 0$, the perturbed solution is close to a single translate of the periodic wave, corresponding to a modulation that depends only on time. 

We emphasize that the spatial dependence of $\varphi(x,t)$ is intrinsic and not merely an artifact of our analysis. In fact, this phenomenon already appears in exact solutions of the KdV equation describing $N$ solitary waves propagating on a periodic background. Such solutions were constructed in~\cite{kurznetsov_1974_stability}, see also~\cite{hoefer2023kdv,Hu_2012_explicit}, via the inverse scattering transform. As observed in~\cite{kurznetsov_1974_stability,hoefer2023kdv,Hu_2012_explicit}, each solitary wave travels with its own characteristic mean velocity along the periodic background wave and induces a nontrivial phase shift in the background as it passes. By choosing the soliton amplitudes sufficiently small and arranging that the total accumulated phase shift equals $0$, one obtains a periodic wave subject to a localized perturbation, corresponding precisely to a solution $u(t)$ to~\eqref{KdV} considered in Theorem~\ref{t:main_result}. Although the accumulated phase shift is zero, the solution resembles different spatial translates of the periodic wave between successive solitary dislocations; see Figure~\ref{fig2}. Capturing these \emph{phase defects} naturally requires a modulation function $\varphi(x,t)$ that depends on both space and time. In particular, since at any fixed time the perturbed solution must not be uniformly close to a single spatial translate of the periodic wave, orbital stability can only hold locally uniformly; that is, the space $H^2(x_*-R,x_*+R)$ in estimate~\eqref{e:orbital_KdV} cannot be replaced by $H^2(\R)$.

\begin{figure}[t!]
\centering
\includegraphics[width=0.86\textwidth]{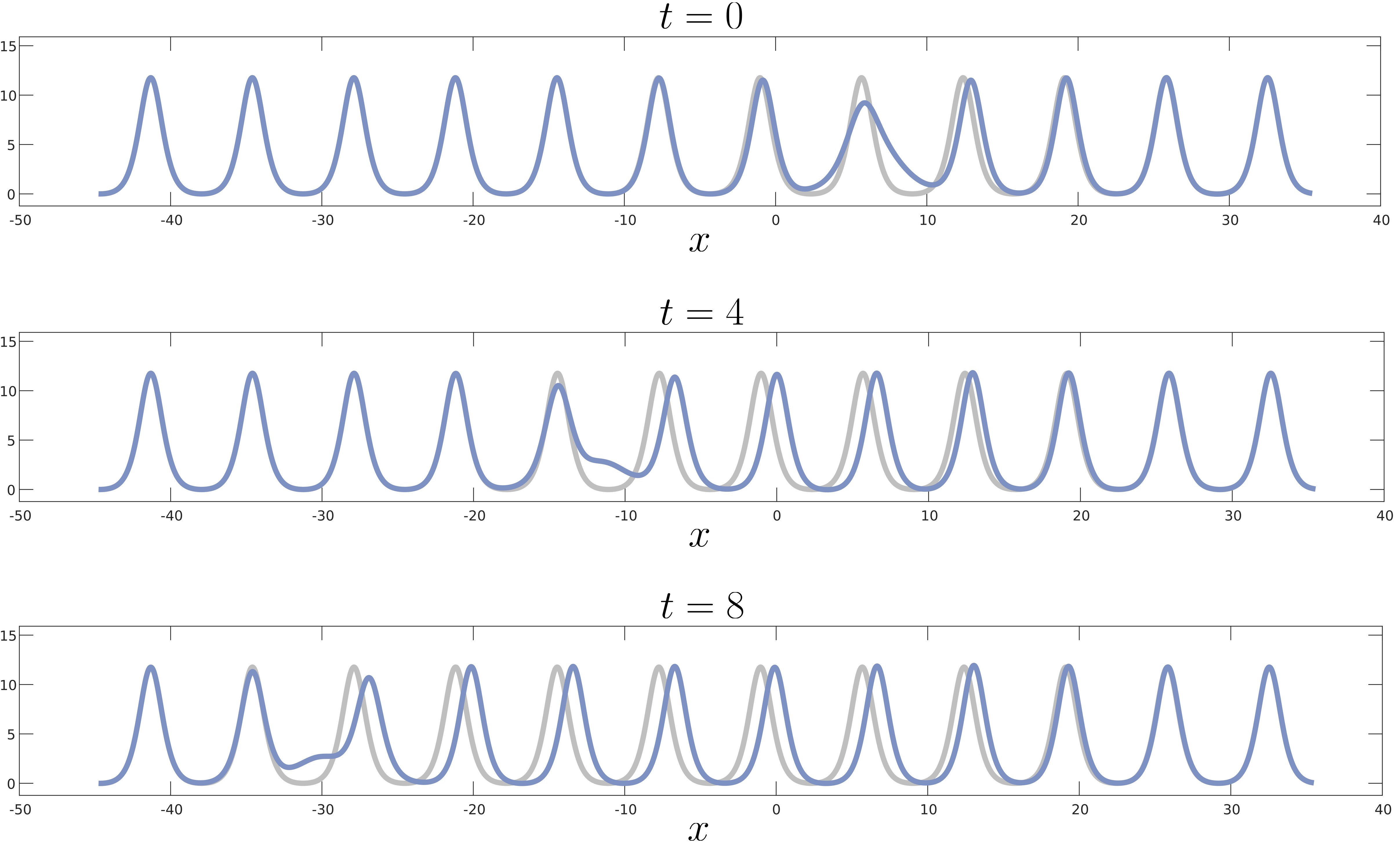}
\includegraphics[width=0.86\textwidth]{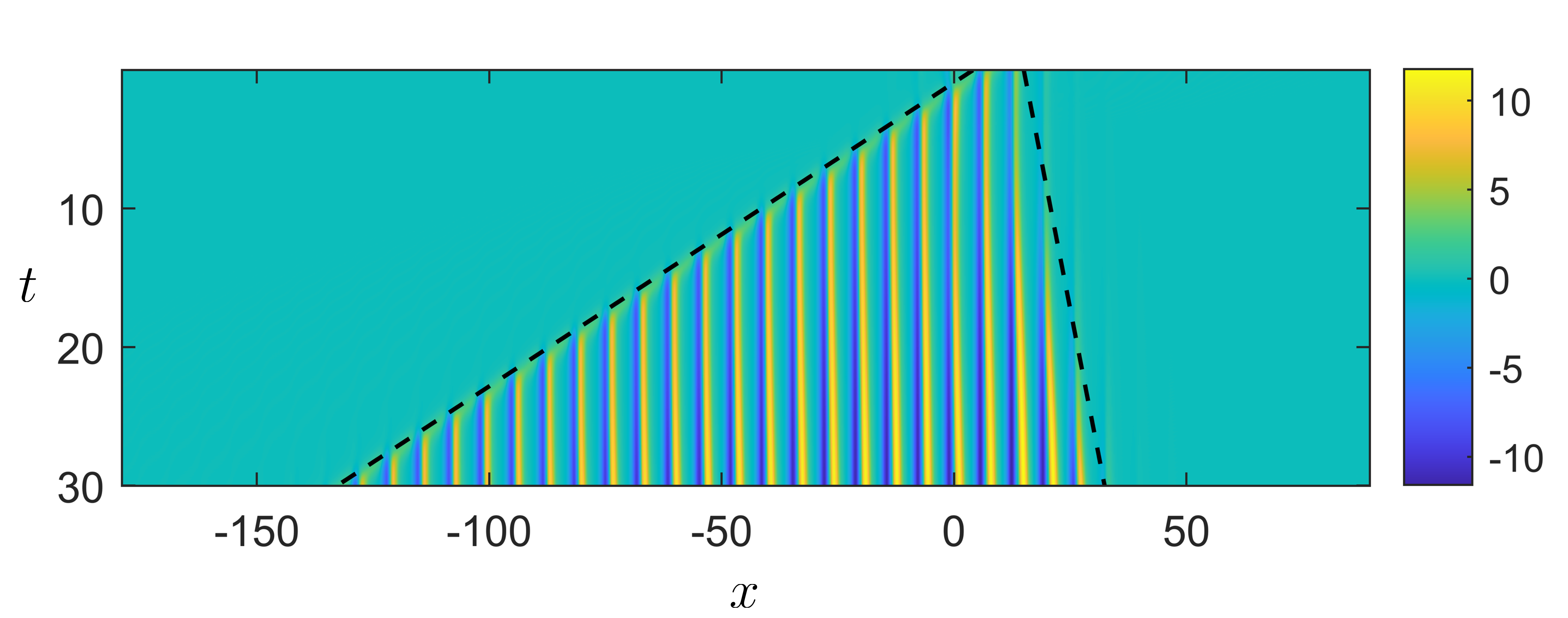}
\caption{\small Time integration of a periodic cnoidal wave subject to a localized perturbation. The initial condition is constructed by patching together two exact one-soliton solutions on a cnoidal background from~\cite{hoefer2023kdv}, chosen so that the net phase shift of the background is zero. Top: The solution $u(t)$  to~\eqref{cKdV} is shown in blue at $t=0$, $4$, and $8$; the unperturbed cnoidal wave $w(x) = 12 \Eps \, \mathrm{cn}^2(x-x_0,\Eps)$ with elliptic modulus $\Eps = 0.9801$ and phase off-set $x_0 = -1.029$ is shown in gray for reference. The wave speed is $c = 4(2\Eps - 1)$. The perturbation triggers a phase defect that travels outwards in both spatial directions with distinct velocities. Bottom: Space-time plot of the perturbation $v(t) = u(t) - w$, together with straight lines (dashed) fitted to the left and right interfaces of the phase defect.
} \label{fig2}
\end{figure}

Furthermore, as the solitary dislocations propagate with their own characteristic velocities, the distances between neighboring dislocations can grow linearly in time; see Figure~\ref{fig2}. As a consequence, the spatial profile of $\varphi(t)$ can develop linearly expanding plateau states, causing its $L^2(\R)$-norm to grow at rate $t^{1/2}$. Numerical simulations in~\cite{bukieda2025orbital} show similar dynamics for localized perturbations of plane waves in the complex Klein--Gordon equation: the phase exhibits expanding plateau states and its $L^2(\R)$-norm grows at rate $t^{1/2}$. Based on these analytical and numerical considerations, we conjecture that this behavior is generic across Hamiltonian models with symmetry. In particular, we believe that the linear growth bound on $\smash{\|\varphi(t)\|_{L^2(\R)}}$ in Theorem~\ref{t:main_result} is not sharp and can be improved to $\leq M E_0 (1+t)^{1/2}$. 

It was observed in~\cite{bona2004stability} that the higher-order conserved quantities associated with the KdV hierarchy may be used to control higher-order Sobolev norms of a perturbation in terms of lower-order ones. Adapting this idea to the present modulational framework, we extend Theorem~\ref{t:main_result} to obtain nonlinear stability in Sobolev spaces of arbitrarily high order. The precise statement is given in Corollary~\ref{c:main_result}.

\subsection{Challenges}

Standard stability arguments in Hamiltonian systems with symmetry, see~\cite{kapitula_spectral_2013,geyer2025stability} for an overview, characterize stable solutions as strict minimizers of a conserved energy subject to finitely many constraints. These approaches, building on the seminal works of Grillakis, Shatah, and Strauss~\cite{grillakis1987stability,grillakis1990stability}, have been successfully applied to establish orbital stability of periodic waves against co-periodic or subharmonic perturbations; see~\cite{geyer2025stability,benzonigavage2016coperiodic,Benzoni2013Stability} and references therein. However, they fail for localized perturbations of periodic waves, since the $L^2(\R)$-spectrum of the second variation of the energy about the wave is purely essential and touches the origin due to the system's symmetries. As a result, the second variation fails to be coercive on any finite-codimensional constraint space, a key requirement for the application of the framework of Grillakis, Shatah, and Strauss and related methods. This obstruction is reflected in subharmonic stability analyses~\cite{deconinck2010orbital,gallay2015orbitalI}: since the spectral gap around the zero eigenvalue of the second variation (as an operator on $L^2_\per(0,n\ell)$, where $n \in \N$ and $\ell > 0$ is the fundamental period of the underlying wave) vanishes as $n\to\infty$, the allowable size of perturbations shrinks to zero, precluding a direct passage to the limit $n \to \infty$ and hence any extension to localized perturbations.

While nonlinear stability of periodic waves under localized perturbations has remained largely open in Hamiltonian systems~\cite{Benzoni2013Stability,benzonigavage2016coperiodic}, significant progress has been made over the past decades in \emph{dissipative} systems such as reaction-diffusion models~\cite{johnson2011nonlinear,sandstede2012diffusive}, viscous conservation laws~\cite{johnson_nonlinear_2010}, and the Saint-Venant equations for viscous shallow water waves~\cite{johnson2011nonlinear}. Due to the translational symmetry of these systems, localized perturbations typically trigger a non-uniform phase mismatch along the periodic wave, which is only healed at a diffusive rate through the temporal dynamics. This is reflected by the spectrum of the linearization about the wave, which touches the origin. Thus, the linearization is diffusively spectrally stable at best. This obstruction has been overcome by the use of mode filters in Bloch-frequency domain~\cite{schneider_diffusive_1996,schneider1998nonlinear} combined with the modulational ansatz of~\cite{doelman_dynamics_2009,johnson_nonlinear_2010}, which enables the isolation of the critical diffusive phase dynamics at both the linear and nonlinear level. Because the remaining dynamics are damped by dissipation, the resulting decay rates are sufficient to close a nonlinear stability argument via iterative estimates on the Duhamel formulation.

A key obstacle to extending these techniques to the present Hamiltonian setting is the absence of damping. This is reflected by the purely imaginary spectrum of the linearization of~\eqref{KdV} about a cnoidal wave~\cite{bottman2009kdv}. While~\cite{rodrigues2018linear} shows that the associated semigroup can be decomposed into a slowly dispersing low-frequency component capturing the critical modulational dynamics and a high-frequency component exhibiting enhanced dispersive decay, these decay rates are insufficient to close a nonlinear stability argument following the aforementioned approach for dissipative systems.

The crucial observation is that the second variation of a conserved energy corresponds to a linear operator that is much more similar to the linearization of a dissipative system. In particular, the second variation of the conserved quantity $E_\gamma$, given by~\eqref{e:def_energy}, corresponds to a diffusively spectrally stable operator by Theorem~\ref{thm:diffusive_intro}. This suggests to combine variational arguments with the modulational ansatz of~\cite{doelman_dynamics_2009,johnson_nonlinear_2010} to close a nonlinear stability argument, which is the approach adopted in this work.

\subsection{Strategy of proof} \label{sec:strategy}

Suppose $u_{\text{tw}}(x,t) = w(x-ct)$ is a traveling-wave solution to the KdV equation~\eqref{KdV} with wave speed $c \in \R$ and periodic profile $w\colon \R \to \R$. Passing to the frame moving with speed $c$, the profile $w$ becomes a stationary solution to~\eqref{cKdV}. Take an initial perturbation $v_0 \in H^3(\R)$, and let $u(t)$ denote the solution to~\eqref{cKdV} with initial condition $u(0) = w + v_0$. To track the evolution of the perturbation, we define $v(t) = u(t) - w$.

We aim to close a variational nonlinear stability argument using the energy $E_\gamma$, given by~\eqref{e:def_energy}. Since $E_\gamma$ is conserved by the flow of~\eqref{cKdV} and $w$ is a critical point of $E_\gamma$, it is natural to consider the formal energy difference
\begin{align*}
\Lambda_\gamma(v(t)) = E_\gamma(w+v(t)) - E_\gamma(w)
\end{align*}
as a Lyapunov functional. Although both $E_\gamma(w+v(t))$ and $E_\gamma(w)$ are divergent improper integrals, the fact that $E_\gamma'(w) = 0$ ensures their \emph{relative energy} $\Lambda_\gamma(v(t))$ is nevertheless well defined and conserved after interchanging subtraction and limits; see Section~\ref{sec:energy} for details. However, a standard coercivity estimate using the expansion
\begin{align} \label{e:formal_energy_expansion_intro} 
\begin{split}
E_\gamma(w+v(t)) - E_\gamma(w) &= \frac12 \langle \A_\gamma v(t),v(t)\rangle_{L^2(\R)} + \mathcal{O}\big(\|v(t)\|_{H^2(\R)}^3\big)
\end{split}
\end{align}
fails, because the differential operator $\A_\gamma$ associated with the second variation $E_\gamma''(w)$ is not positive definite. 

Instead, we introduce the \emph{inverse-modulated perturbation}
\begin{align} \label{e:def_inv_mod}
\vt(x,t) = u(\cdot + \psi(\cdot,t),t) - w(x), 
\end{align}
where the modulation function $\psi(x,t)$ is chosen \emph{a-posteriori} so as to capture the neutral translational dynamics that obstruct positive-definiteness of $\A_\gamma$. Exploiting the diffusive spectral stability of $-\A_\gamma$, we recover coercivity, at the expense of a spatial derivative of $\psi$. To prove Theorem~\ref{t:main_result}, it then suffices to establish global-in-time smallness of $\vt(t)$ and $\nabla \psi(t)$. Indeed, performing the change of variables $y = x + \psi(x,t)$, which is invertible as long as $\|\psi_x(t)\|_{L^\infty(\R)} < 1$, in~\eqref{e:def_inv_mod} yields a decomposition of the form~\eqref{e:decomp}, where $z(t)$ and $\nabla \varphi(t)$ are controlled by $\vt(t)$ and $\nabla \psi(t)$. We refer to Remark~\ref{rem:forward} for an explanation of why we work with the inverse-modulated perturbation $\vt(t)$ rather than attempting to control the \emph{forward-modulated perturbation} $z(t)$ directly. 

The inverse-modulated perturbation $\vt(t)$ satisfies an equation of the form
\begin{align} \label{e:mod_pert_eq}
(\partial_t - \El)(\vt - w'\psi) = \mathcal{N}(\vt,\nabla \psi)
\end{align}
where $\El = \mathcal{J} A$ is the linearization of~\eqref{cKdV} about $w$ and the nonlinear residual $\mathcal{N}(\vt,\nabla \psi)$ depends only on spatial and temporal derivatives of $\psi$. Assuming that $\psi(t)$ vanishes identically at $t = 0$, the associated Duhamel formulation reads
\begin{align} \label{e:Duh_mod}
\vt(t) - w'\psi(t) = \eu^{t\El} v_0 + \int_0^t \eu^{(t-s)\El} \mathcal{N}(\vt(s),\nabla \psi(s)) \ds.
\end{align}
We rewrite the conserved quantity $\Lambda_\gamma(v(t)) = E_\gamma(u(t)) - E_\gamma(w)$ in terms of $\vt(t)$ and $\psi_x(t)$. This is achieved by making the change of variables $y = x + \psi(x,t)$ in the integral~\eqref{e:def_energy} defining $E_\gamma(u(t))$. Altogether, this yields the expansion
\begin{align} \label{e:expan_energy_mod}
\begin{split}
E_\gamma(u(t)) - E_\gamma(w) &= \frac12 \langle \A_\gamma (\vt(t) - w'\psi(t)), \vt(t) - w'\psi(t)\rangle_{L^2(\R)}\\
&\qquad + \mathcal{O}\left(\left(\|\vt(t)\|_{H^2(\R)} + \|\psi_x(t)\|_{H^2(\R)}\right)^3\right).
\end{split}
\end{align}
We note that the structure of the linear term in~\eqref{e:mod_pert_eq} and the bilinear term in~\eqref{e:expan_energy_mod} may be seen directly by inserting $u(t) = w+ v(t)$ into~\eqref{e:def_inv_mod}, which gives
\begin{align} \label{e:formal_expasion_mod_pert}
v(t) = \vt(t) - w' \psi(t) + \text{nonlinear terms in $\psi(t)$ and $v(t)$}.
\end{align}
 
Since $w'$ spans the kernel of the Bloch operator $\A_\gamma(0)$, we can use the Duhamel formulation~\eqref{e:Duh_mod} to define $\psi(t)$ so that the large-time linearized dynamics of $\vt(t)$ is confined to a subspace on which $\A_\gamma$ is coercive, while $\psi(t)$ vanishes identically at $t = 0$. More precisely, we decompose the right-hand side of~\eqref{e:Duh_mod} in Bloch space and choose $\psi(t)$ in such a way that $\vt(t) = \vt_1(t) + \vt_2(t)$, where $\vt_2(t)$ is a residual term that can be estimated directly from the Duhamel formula, while $\vt_1(t)$ and $\psi(t)$ satisfy the key estimate
\begin{align} \label{e:coercivity_key}
\langle \A_\gamma (\vt_1(t) - w'\psi(t)), \vt_1(t) - w'\psi(t)\rangle_{L^2(\R)} \geq \alpha \left(\|\vt_1(t)\|_{H^2(\R)}^2 + \|\psi_x(t)\|_{H^2(\R)}^2\right)
\end{align}
for some constant $\alpha > 0$.  

Estimate~\eqref{e:coercivity_key} is a direct consequence of the diffusive spectral stability of $-\A_\gamma$. It recovers coercivity at the cost of a derivative of the modulation function $\psi(t)$, cf.~identity~\eqref{e:coercivity_diffusion}. Nevertheless, since both the perturbation equation~\eqref{e:mod_pert_eq} and the energy expansion~\eqref{e:expan_energy_mod} depend only on derivatives of $\psi(t)$, this derivative loss does not prevent the nonlinear argument from closing. Indeed, combining the coercivity estimate~\eqref{e:coercivity_key} with Duhamel-based bounds on $\vt_2(t)$ and $\psi_t(t)$, we obtain global-in-time smallness of $\vt(t)$ and $\nabla \psi(t)$, which ultimately leads to the proof of Theorem~\ref{t:main_result}.

\begin{remark} \label{rem:forward}
The reason that we do not directly work with the forward-modulated perturbation $z(t)$, given by~\eqref{e:decomp}, in our nonlinear stability analysis is that the equation for $z(t)$ contains nonlinear terms depending on the difference $w(x+\varphi(x,t)) - w(x)$ between the modulated and unmodulated wave. These terms cannot be controlled through the $L^2(\R)$-norm of the space-time gradient $\nabla \varphi(t)$. In contrast, since $\A_\gamma w' = 0$, the equation~\eqref{e:def_inv_mod} for the inverse-modulated perturbation $\vt(t)$ only depends on spatial and temporal derivatives of the modulation $\psi(t)$. We refer to~\cite{zumbrun_2024_forward} for further details on forward and inverse modulation of periodic waves.
\end{remark}

\begin{remark}
The central ingredient in proving Theorem~\ref{t:main_result} is a conserved quantity $E_\gamma$, which possesses $w$ as a critical point and whose second variation corresponds to a diffusively spectrally stable operator. Once this key ingredient has been established, see Theorem~\ref{thm:diffusive_intro} and Remark~\ref{rem:diff_energies_in_other_systems}, neither the explicit form~\eqref{e:cnoidal_form} of the periodic traveling waves nor the specific structure or complete integrability of the KdV equation play any role in our nonlinear stability analysis; cf.~Section~\ref{sec:intro_principle}.
\end{remark}

\subsection{Outline of paper}

Section~\ref{sec:coercivity} is devoted to establishing a general coercivity estimate under suitable diffusive spectral assumptions. In~\S\ref{sec:nonlinear_scheme}, we develop the nonlinear tracking scheme, introduce the modulated perturbation variables, and modulate the energy. The nonlinear stability argument is carried out in~\S\ref{sec:nonlinear_stab_analysis}, where we prove our main result and derive its extension to Sobolev spaces of arbitrarily high order. We conclude in~\S\ref{sec:discussion} with a discussion of possible extensions of our approach and their relevance to several open problems.

Appendix~\ref{appendixFB} recalls several properties of the Floquet--Bloch transform used throughout the analysis. In Appendix~\ref{appendixPhPA}, we collect basic properties of the profile equation associated with the periodic traveling wave. Appendix~\ref{appendixSpectral} contains the spectral analysis of the second variation of the energy about the periodic wave. Finally, a number of technical lemmas are deferred to Appendix~\ref{appendixtech}.

\section{Coercivity estimate} \label{sec:coercivity}

In this section, we establish a coercivity estimate for general self-adjoint periodic differential operators $\A$ under the assumption that $-\A$ is diffusively spectrally stable. As outlined in~\S\ref{sec:strategy}, this coercivity estimate, which incurs at the cost of a derivative, serves as the key ingredient in our variational stability argument. In this section and throughout this paper, we denote by $B(X)$ the Banach space of bounded linear operators on a Banach space $X$.

Let $k,m \in \N$ and $\ell > 0$. Consider a self-adjoint differential operator $\A \colon H^{2k}(\R,\R^m) \to L^2(\R,\R^m)$ of the form
\begin{align*}
\A z = \sum_{j = 0}^{2k} a_j \, \partial_x^j z
\end{align*}
where the coefficients $a_j \colon \R \to \R^{m \times m}$ are $k$-times continuously differentiable and $\ell$-periodic. We assume that $\A$ is \emph{uniformly elliptic}, in the sense that $a_{2k}(x)$ is symmetric and there exists a constant $\theta_1 > 0$ such that
\begin{align} \label{e:ellipticity} 
(-1)^k \langle a_{2k}(x) z,z\rangle \geq \theta_1 |z|^2
\end{align}
for all $x \in \R$ and $z \in \R^m$. The associated Bloch operators $\A(\xi) \colon H^{2k}_\per(0,\ell) \to L^2_\per(0,\ell)$ for $\xi \in \C$ are defined by~\eqref{e:def_Bloch}.

Before presenting the coercivity estimate, we record two immediate consequences of the diffusive spectral stability and uniform ellipticity of $\A$, respectively. The first concerns the low-frequency spectrum of $\A$, which, by~\eqref{e:spectral_rel}, is parameterized by the eigenvalue $\lambda_c(\xi)$ of smallest real part of $\A(\xi)$ for $|\xi| \ll 1$. This eigenvalue is obtained by continuing the simple eigenvalue $0$ of $\A(0)$ in $\xi$ via standard analytic perturbation theory~\cite{kato_perturbation_2013}.

\begin{lemma}[Low-frequency spectrum] \label{lem:low_freq}
Suppose that $0$ is the smallest eigenvalue of $\A(0)$. If $0$ is simple, then there exist an open neighborhood $U \subset \C$, a constant $\theta_2 \in (0,1)$, and analytic functions $\lambda_c \colon U \to \C$ and $P \colon U \to B(L^2_\per(0,\ell))$ such that $\lambda_c(\xi)$ is the eigenvalue of smallest real part of $\A(\xi)$, $P(\xi)$ is the associated spectral projection of rank $1$, and it holds
\begin{align*}
\inf \textnormal{Re} \left(\sigma_{L^2_\per(0,\ell)}(\A(\xi)) \setminus \{\lambda_c(\xi)\}\right) \geq \theta_2
\end{align*}
for all $\xi \in U$.
\end{lemma}

Next, we state a standard result from elliptic regularity theory, which shows that $L^2$-coercivity of the uniformly elliptic operator $\A$ upgrades to $H^k$-coercivity. This result is crucial for controlling regularity in our nonlinear stability argument.

\begin{lemma}[Uniform elliptic regularity estimate] \label{lem:elliptic}
Let $K \subset \C$ be compact. Then, there exists a constant $\theta_3 > 0$ such that for all $\xi \in K$, each $\zeta_\xi \in (0,1)$, and each subspace $V_\xi \subset H^{2k}_\per(0,\ell)$ the implication
\begin{align*}
&\forall \, z \in V_\xi : \langle \A(\xi) z, z\rangle_{L^2(0,\ell)} \geq \zeta_\xi \|z\|_{L^2(0,\ell)}^2\\ 
&\qquad \Longrightarrow \qquad \forall \, z \in V_\xi : 
\langle \A(\xi) z, z\rangle_{L^2(0,\ell)} \geq \theta_3 \zeta_\xi \|z\|_{H^k(0,\ell)}^2
\end{align*}
holds.
\end{lemma}
\begin{proof}
Since $K$ is compact, there exists a constant $C_0 > 0$ such that $|\xi| \leq C_0$ for all $\xi \in K$. Moreover, by uniform ellipticity of $\A$, there exists $\theta_1 > 0$ such that~\eqref{e:ellipticity} holds for all $x \in \R$ and $z \in \R^m$. Hence, integrating by parts, using the Cauchy--Schwarz and Young's inequalities, and interpolating, we obtain a constant $C > 0$ such that
\begin{align} \label{e:elliptic2}
\langle \A(\xi) z, z\rangle_{L^2(0,\ell)} \geq \frac12 \theta_1 \|z\|_{H^k(0,\ell)}^2 - C \|z\|_{L^2(0,\ell)}^2
\end{align}
for all $\xi \in K$ and $z \in H^{2k}_\per(0,\ell)$. 

Take $\xi \in K$, $\zeta_\xi > 0$, and a subspace $V_\xi \subset H^{2k}_\per(0,\ell)$ such that
\begin{align*}
\langle \A(\xi) z, z\rangle_{L^2(0,\ell)} \geq \zeta_\xi \|z\|_{L^2(0,\ell)}^2
\end{align*}
holds for all $z \in V_\xi$. Set $\alpha_\xi \coloneqq \zeta_\xi/(C+\zeta_\xi) \in (0,1)$. Then,~\eqref{e:elliptic2} implies
\begin{align*}
\langle \A(\xi) z, z\rangle_{L^2(0,\ell)} &\geq \alpha_\xi \left(\frac12 \theta_1 \|z\|_{H^k(0,\ell)}^2 - C \|z\|_{L^2(0,\ell)}^2\right) + (1-\alpha_\xi) \zeta_\xi \|z\|_{L^2(0,\ell)}^2\\ 
&\geq \frac{\theta_1}{2(C + 1)}\, \zeta_\xi \|z\|_{H^k(0,\ell)}^2
\end{align*}
for all $z \in V_\xi$. Thus, setting $\theta_3 = \theta_1/(2(C+1))$ yields the result.    
\end{proof}

Finally, we establish the desired coercivity estimate. The proof relies on the Floquet--Bloch version of Parseval's identity, which decomposes the quadratic form $\langle \A z, z\rangle_{L^2(\R)}$ into contributions from low and high Bloch frequencies. We refer to Appendix~\ref{appendixFB} for relevant background on the Floquet--Bloch transform $\mathcal{B} u = \check{u}$ of a function $u \in L^2(\R)$. 

\begin{proposition}[Coercivity at the cost of a derivative] \label{prop:coercivity}
Suppose that $-\A$ is diffusively spectrally stable, i.e., $0$ is a simple eigenvalue of $\A(0)$ and there exists $\theta_0 > 0$ such that~\eqref{e:coercivity_diffusive} holds. Let $w_0 \in \ker(\A(0)) \setminus \{0\}$ be a corresponding eigenfunction. Let $P(\xi)$ be the spectral projection of the Bloch operator $\A(\xi)$, established in Lemma~\ref{lem:low_freq}. Fix $j \in \N$ with $j \geq 2k-1$.

Then, for all $\xi_1 \in (0,\frac{\pi}{\ell})$ sufficiently small, there exists a constant $\alpha > 0$ such that for all $v \in H^{2k}(\R,\R^m)$ and $\psi \in H^{j+1}(\R,\R)$ satisfying
\begin{itemize}
    \item[(i)] $\check{v}(\xi) \in \ker(P(\xi))$ for all $\xi \in (-\xi_1,\xi_1)$;
    \item[(ii)] $\check{\psi}(\xi)$ is a constant function for each $\xi \in [-\frac{\pi}{\ell},\frac{\pi}{\ell})$;
    \item[(iii)] $\check{\psi}(\xi) \equiv 0$ for all $\xi \in [-\frac{\pi}{\ell},\frac{\pi}{\ell}) \setminus (-\xi_1,\xi_1)$,
\end{itemize}
we have
\begin{align*}
\left\langle \A (v - w_0\psi), v-w_0\psi\right\rangle_{L^2(\R)} \geq \alpha \left(\|v\|_{H^k(\R)}^2 + \|\psi'\|_{H^j(\R)}^2\right).
\end{align*}
\end{proposition}
\begin{proof}
First, we recall from Lemma~\ref{lem:low_freq} that $P \colon U \to B(L^2_\per(0,\ell))$ is analytic. In particular, there exists a constant $C \geq 1$ such that
\begin{align} \label{e:coer2}
\|P(\xi) - P(0)\|_{B(L^2_\per(0,\ell))}, \|P(\xi) - P(0)\|_{B(H^k_\per(0,\ell))} \leq C|\xi| 
\end{align} 
for all $\xi \in U$, where we used that the spectral projection $P(\xi)$ commutes with $\A(\xi)$ and the $H^{k}_\per(0,\ell)$-norm is equivalent to the graph norm of $\A(\xi)$. Moreover, since $\A(\xi)$ is self-adjoint, $P(\xi)$ is an orthogonal projection for all $\xi \in U \cap \R$.

Let $\xi_1 > 0$ be so small that $[-\xi_1,\xi_1] \subset U$ and $\xi_1 \leq \min\{1/(4C),\theta_0^{-1/2}\}$. Combining Lemma~\ref{lem:elliptic} once with the spectral bound~\eqref{e:coercivity_diffusive} and once with Lemma~\ref{lem:low_freq}, we find a constant $\theta_3 > 0$ such that the inequality
\begin{align} \label{e:coer00}
\langle \A(\xi)z,z\rangle_{L^2(0,\ell)} \geq \theta_0 \theta_3 \xi_1^2 \|z\|_{H^k(0,\ell)}^2
\end{align}
holds for all $z \in H_\per^{2k}(0,\ell)$ and $\xi \in [-\frac{\pi}{\ell},\frac{\pi}{\ell}] \setminus (-\xi_1,\xi_1)$, as well as the estimate
\begin{align} \label{e:coer01}
\langle \A(\xi)z,z\rangle_{L^2(0,\ell)} \geq \theta_2 \theta_3 \|z\|_{H^k(0,\ell)}^2
\end{align}
for all $z \in \ker(P(\xi))$ and $\xi \in [-\xi_1,\xi_1]$.

Take $v \in H^{2k}(\R,\R^m)$ and $\psi \in H^{j+1}(\R,\R)$ satisfying conditions (i)-(iii) and set $z = v - w_0\psi$. Using Parseval's identity together with the other basic properties of the Floquet--Bloch transform (see Appendix~\ref{appendixFB}), we decompose
\begin{align*}
2\pi \ell \, \langle \A z,z\rangle_{L^2(\R)} &= \int_{-\frac{\pi}{\ell}}^{\frac{\pi}{\ell}} \langle \A(\xi)\check{z}(\xi),\check{z}(\xi)\rangle_{L^2(0,\ell)} \dxi = I_1 + I_2 + I_3
\end{align*}
into a critical low-frequency component
\begin{align*}
I_1 = \int_{-\xi_1}^{\xi_1} \lambda_c(\xi) \|P(\xi) w_0\|_{L^2(0,\ell)}^2 \, |\check{\psi}(\xi)|^2 \dxi, 
\end{align*}
a residual low-frequency component
\begin{align*}
I_2 = \int_{-\xi_1}^{\xi_1} \langle \A(\xi)(I-P(\xi))\check{z}(\xi),(I-P(\xi))\check{z}(\xi)\rangle_{L^2(0,\ell)} \dxi,
\end{align*}
and a high-frequency component
\begin{align*} 
I_3 = \int_{[-\frac{\pi}{\ell},\frac{\pi}{\ell}] \setminus (-\xi_1,\xi_1)} \langle \A(\xi)\check{v}(\xi),\check{v}(\xi)\rangle_{L^2(0,\ell)} \dxi.
\end{align*}
Setting $\beta = \min\{\theta_0,\theta_2\theta_3\} > 0$, applying~\eqref{e:coercivity_diffusive} and~\eqref{e:coer01}, and using $\|w_0\|_{H^k(0,\ell)} \geq \|w_0\|_{L^2(0,\ell)}$ and $C \geq 1$, we readily obtain
\begin{align} \label{e:coer3}
I_1 \geq \beta \int_{-\xi_1}^{\xi_1} \xi^2 \|P(\xi) w_0\|_{L^2(0,\ell)}^2 \, |\check{\psi}(\xi)|^2 \dxi
\end{align}
and
\begin{align} \label{e:coer33}
I_2 \geq \frac{\beta \|w_0\|_{L^2(0,\ell)}^2}{4C^2\, \|w_0\|_{H^k(0,\ell)}^2} \int_{-\xi_1}^{\xi_1} \|(I-P(\xi))\check{z}(\xi)\|_{H^k(0,\ell)}^2 \dxi.
\end{align}
Using~\eqref{e:coer2}, $\xi_1 \leq 1/(4C)$, and $w_0 \in  \ker(\A(0)) = \ran(P(0))$, we establish
\begin{align} \label{e:coer4}
\|P(\xi)w_0\|_{L^2(0,\ell)} &\geq \left(1 - \|(P(\xi) - P(0))\|_{B(L^2_\per(0,\ell))} \right)\|w_0\|_{L^2(0,\ell)} \geq \frac34 \|w_0\|_{L^2(0,\ell)}
\end{align}
for $\xi \in (-\xi_1,\xi_1)$. On the other hand, applying the Cauchy--Schwarz and Young inequalities and using~\eqref{e:coer2}, $w_0 \in \ran(P(0))$, $\check{v}(\xi) \in \ker(P(\xi))$, and the fact that $\check{\psi}(\xi)$ is a constant function, we infer
\begin{align} \label{e:coer5}
\begin{split}
\|(I-P(\xi))\check{z}(\xi)\|_{H^k(0,\ell)}^2 
&\geq \|\check{v}(\xi)\|_{H^k(0,\ell)}^2 - 2 \|\check{v}(\xi)\|_{H^k(0,\ell)} \|(I-P(\xi)) w_0\|_{H^k(0,\ell)} |\check{\psi}(\xi)|\\ 
&\qquad + \, \|(I-P(\xi)) w_0\|^2_{H^k(0,\ell)} |\check{\psi}(\xi)|^2\\
&\geq \frac12\|\check{v}(\xi)\|_{H^k(0,\ell)}^2 - 2 \|(P(0) - P(\xi)) w_0\|^2_{H^k(0,\ell)} |\check{\psi}(\xi)|^2\\
&\geq \frac12\|\check{v}(\xi)\|_{H^k(0,\ell)}^2 - 2 C^2 \xi^2 \|w_0\|^2_{H^k(0,\ell)} |\check{\psi}(\xi)|^2
\end{split}
\end{align}
for $\xi \in (-\xi_1,\xi_1)$. Since $\check{\psi}(\xi)$ is a constant function, we have $\check{\psi'}(\xi) = \iu \xi \check{\psi}(\xi)$. Hence, inserting~\eqref{e:coer4} and~\eqref{e:coer5} into~\eqref{e:coer3} and~\eqref{e:coer33}, respectively, employing~\eqref{e:coer00} to bound $I_3$, and applying Parseval's identity again, we arrive at
\begin{align*}
2\pi \ell \, \langle \A z,z\rangle_{L^2(\R)} &= I_1 + I_2 + I_3\\ 
&\geq \frac{\beta}{16} \|w_0\|^2_{L^2(0,\ell)} \int_{-\xi_1}^{\xi_1} |\iu \xi \check{\psi}(\xi)| \dxi + \frac{\beta \|w_0\|_{L^2(0,\ell)}^2}{8C^2\, \|w_0\|_{H^k(0,\ell)}^2} \int_{-\xi_1}^{\xi_1} \|\check{v}(\xi)\|^2_{H^k(0,\ell)} \dxi\\ 
&\qquad + \, \theta_0 \theta_3 \xi^2_1 \int_{[-\frac{\pi}{\ell},\frac{\pi}{\ell}) \setminus (-\xi_1,\xi_1)} \|\check{v}(\xi)\|^2_{H^k(0,\ell)} \dxi\\ 
&\geq 2\pi\ell \tilde\alpha\left(\|\psi'\|_{L^2(\R)}^2 + \|v\|_{H^k(\R)}^2\right)
\end{align*}
where $\tilde\alpha = \smash{\min\{\beta \|w_0\|_{L^2(0,\ell)}^2/(8(2+C^2\|w_0\|_{H^k(0,\ell)}^2)), \theta_0 \theta_3 \xi_1^2\}}$. Finally, we note that the fact that $\check{\psi}(\xi)$ is a constant function  implies that $\mathcal{B}(\partial^n_x \psi)(\xi) = (\iu \xi)^n \check{\psi}(\xi)$ for all $n \in \N$ and $\xi \in [-\frac\pi\ell,\frac\pi\ell)$. Therefore, Parseval's identity yields a constant $C_1 > 0$, only depending on $j$, such that $\|\psi'\|_{H^j(\R)} \leq C_1 \|\psi'\|_{L^2(\R)}$, which yields the result. 
\end{proof}

\section{Nonlinear tracking scheme} \label{sec:nonlinear_scheme}

Our nonlinear stability analysis combines variational arguments based on the conserved energy $E_\gamma$, given by~\eqref{e:def_energy}, with direct estimates on the Duhamel formulation of the perturbation. To account for the excitation of the neutral translational modes, we introduce a spatiotemporal modulation function that tracks the leading-order translational dynamics of the wave under perturbations. 

In this section, we develop the corresponding tracking scheme. After having introduced the perturbation and the associated relative energy, we first modulate both quantities using an \emph{arbitrary} smooth modulation function. Guided by the resulting expressions and estimates, we then make a judicious choice for the modulation function in a way that is compatible with the coercivity estimate of Proposition~\ref{prop:coercivity}. This choice yields a lower bound for the modulated relative energy, at the cost of a derivative of the modulation function. This allows us to close a nonlinear stability argument in the upcoming section, ultimately showing that the modulation function indeed captures the translational dynamics to leading order.

Throughout this section, we let $u_{\text{tw}}(x,t) = w(x-ct)$ be a periodic traveling-wave solution to the KdV equation~\eqref{KdV} with wave speed $c \in \R$ and profile $w\colon \R \to \R$ of fundamental period $\ell > 0$. Moreover, we let $c_{1,2} \in \R$ be as in~\S\ref{sec:diffusive_KdV_intro}, so that $w$ satisfies the Euler--Lagrange equations~\eqref{e:Euler_Lagrange_1} and~\eqref{e:Euler_Lagrange_2}.

\subsection{The unmodulated perturbation}

Take an initial perturbation $v_0 \in H^3(\R)$, and let $u(t)$ denote the solution to~\eqref{cKdV} with initial condition $u(0) = w + v_0$. Since $w$ is a stationary solution to~\eqref{cKdV}, the perturbation $v(t) = u(t) - w$ obeys the equation
\begin{align} \label{e:unmod_pert_eq}
(\partial_t - \El) v = N(v),
\end{align}
where the operator $\El \colon H^3(\R) \to L^2(\R)$, given by
\begin{align*}
\El = \mathcal{J}A = \partial_x(-\partial_x^2+c-w),
\end{align*}
is the linearization of~\eqref{cKdV} about $w$, and
\begin{align*}
N(v) = \mathcal{J}\left(\mathcal{H}'(w+v) - \mathcal{H}'(w) - A v\right) = -vv_x
\end{align*}
is the nonlinear residual. 

Since~\eqref{e:unmod_pert_eq} can be regarded as a KdV equation with smooth bounded potential, global existence and uniqueness of classical solutions follows from standard well-posedness theory, for instance via the energy method developed by Bona and Smith~\cite{Bona_1975_initial}. The following result was obtained in~\cite[Section 3.1]{erdogan2016dispersive} by combining this energy method with parabolic regularization.

\begin{proposition}[Global existence and uniqueness of the perturbation] \label{prop:globalwellpert}
    Fix $m \in \N_0$. Let $v_0\in H^{m+3}(\R)$. Then,~\eqref{e:unmod_pert_eq} admits a unique global classical solution 
    \begin{align*} v \in C\big(\R,H^{m+3}(\R)\big) \cap C^1\big(\R,H^m(\R)\big)\end{align*} 
    with initial condition $v(0)=v_0$.
\end{proposition}

To later pass to the Duhamel formulation and obtain associated estimates, we observe that the linear operator $\El$ in~\eqref{e:unmod_pert_eq} generates a $C_0$-group on $H^s(\R)$ for any $s \in \N_0$. This follows from basic semigroup theory using the perturbation theorems~\cite[Theorems~II.2.7 and~II.1.3]{engel_one-parameter_2000}, together with the fact that the principal part $-\partial_x^3$ is skew-adjoint on $H^s(\R)$ and hence generates a unitary $C_0$-group by Stone's theorem. The same arguments apply to the associated Bloch operators $\El(\xi)\colon H^{s+3}_\per(0,\ell) \to H^s_\per(0,\ell)$, defined by
\begin{align*}
\El(\xi) = (\partial_x+\iu \xi)(-(\partial_x+\iu \xi)^2+c-w).
\end{align*}

\begin{lemma}[Generation of $C_0$-groups] \label{lem:semigroup}
Fix $s \in \N_0$. There exists a constant $M > 0$ such that the operators $\El\colon H^{s+3}(\R) \rightarrow H^s(\R)$ and $\El(\xi)\colon H^{s+3}_\per(0,\ell) \to H^s_\per(0,\ell)$ generate $C_0$-groups on $H^s(\R)$ and $H^s_\per(0,\ell)$, respectively, satisfying
    \begin{align*}
       \left\| \eu^{t\El} \right\|_{B(H^s(\R))}, \left\| \eu^{t\El(\xi)} \right\|_{B(H^s_\per(0,\ell))} \leq \eu^{M |t|}
    \end{align*}
for all $t \in \R$ and $\xi \in [-\frac{\pi}{\ell},\frac{\pi}{\ell})$.
\end{lemma}

\subsection{The relative energy} \label{sec:energy}

We use the energy $E_\gamma$, defined in~\eqref{e:def_energy}, to construct a Lyapunov functional for the perturbation equation~\eqref{e:unmod_pert_eq}. Interchanging subtraction and integration in the formal difference $E_\gamma(w+v) - E_\gamma(w)$, we obtain the smooth nonlinear functional $\Lambda_\gamma \colon H^2(\R) \to \R$, given by
\begin{align} \label{e:def_Lyapunov}
\Lambda_\gamma(v) = \int_\R \mathcal{K}_\gamma(w+v) - \mathcal{K}_\gamma(w) \dx, \qquad \gamma \in \R,
\end{align}
for $v \in C_{\textnormal{c}}^\infty(\R)$, where
\begin{align} \label{e:def_Kgamma}
\begin{split}
\mathcal{K}_\gamma(u) &= \frac12 u_{xx}^2 - \frac{5}{6}uu_x^2 + \frac{5}{72} u^4 - \frac16 c_1 u^2 - \frac12 c^2 u^2 + c_2 u\\
&\qquad + \, \gamma\left(\frac12 u_x^2 - \frac16 u^3 + \frac12 c u^2 + c_1 u\right).
\end{split}
\end{align}
The \emph{relative energy} $\Lambda_\gamma(v)$ is well defined, because $w$ is a critical point of $E_\gamma$. Indeed, integrating by parts and using the Euler--Lagrange equations~\eqref{e:Euler_Lagrange_1} and~\eqref{e:Euler_Lagrange_2}, we arrive at the representation
\begin{align} \label{e:rep_Lyapunov}
\begin{split}
\Lambda_\gamma(v) &= \int_\R \frac12 v_{xx}^2 - \frac56 (w+v) v_x^2 - \frac53 w'vv_x + \frac5{72} v^4 + \frac5{18} w v^3 + \frac5{12} w^2 v^2 - \frac16 c_1 v^2 \\ 
&\qquad - \, \frac12 c^2 v^2 + \gamma\left(\frac12 v_x^2 - \frac16 v^3 - \frac12 w v^2 + \frac12 cv^2\right) \dx
\end{split}
\end{align}
for $v \in C_{\textnormal{c}}^\infty(\R)$, which shows that $\Lambda_\gamma$ defines a nonlinear functional on $H^2(\R)$ by density. 

It follows from the Lenard recursion relation~\eqref{e:Lenard} that the energy $E_\gamma$ is conserved along (sufficiently localized) solutions of~\eqref{cKdV}. If $v(t)$ solves the perturbation equation~\eqref{e:unmod_pert_eq}, then $u(t)=w+v(t)$ and $w$ are both solutions of~\eqref{cKdV}. Consequently, one expects the relative energy $\Lambda_\gamma(v(t))$, which arises from the formal energy difference $E_\gamma(w+v(t)) - E_\gamma(w)$, to be conserved as well. This is confirmed by the following proposition.

\begin{proposition}[Conservation of the relative energy] \label{prop:conservation}
Let $v_0 \in H^3(\R)$, and let $v \in \smash{C\big(\R,H^3(\R)\big) \cap C^1\big(\R,L^2(\R)\big)}$ be the associated solution to~\eqref{e:unmod_pert_eq}, established in Proposition~\ref{prop:globalwellpert}. Then, we have
\begin{align*}
\Lambda_\gamma(v(t)) = \Lambda_\gamma(v_0)
\end{align*}
for all $t \in \R$ and $\gamma \in \R$.
\end{proposition}
\begin{proof}
Let $t_0 > 0$ and $\gamma \in \R$. It was established in~\cite[Section~3.1]{erdogan2016dispersive} that there exists a sequence of regularized solutions $(v_k)_{k \in \N}$ to~\eqref{e:unmod_pert_eq} in $C([-t_0,t_0],H^5(\R)) \cap C^1([-t_0,t_0],H^2(\R))$
such that $v_k$ converges to $v$ in $C([-t_0,t_0],H^3(\R))$ as $k \to \infty$. In particular, this implies that $\Lambda_\gamma(v_k(t)) \to \Lambda_\gamma(v(t))$ as $k \to \infty$ for all $t \in [-t_0,t_0]$. Hence, it suffices to show that $\Lambda_\gamma(v_k(t)) = \Lambda_\gamma(v_k(0))$ for all $t \in [-t_0,t_0]$ and $k \in \N$. 

Thus, integrating by parts, using the Lenard recursion relation~\eqref{e:Lenard}, and recalling that $w$ satisfies the Euler--Lagrange equations~\eqref{e:Euler_Lagrange_1} and~\eqref{e:Euler_Lagrange_2}, we obtain
\begin{align*}
\partial_t \Lambda_\gamma(v_k(t)) &= \left\langle \mathcal{H}'_2(w+v_k(t)) - \mathcal{H}'_2(w),\partial_t v_k(t)\right\rangle_{L^2(\R)}\\ 
&\qquad + \, \gamma \left\langle \mathcal{H}'(w+v_k(t)) - \mathcal{H}'(w), \partial_t v_k(t)\right\rangle_{L^2(\R)} \\
&= -\left\langle \mathcal{J}_2(w+v_k(t))\left( \mathcal{H}'(w+v_k(t)) - \mathcal{H}'(w)\right),\mathcal{H}'(w+v_k(t)) - \mathcal{H}'(w)\right\rangle_{L^2(\R)}\\
&\qquad - \, \gamma \left\langle \mathcal{J} \left(\mathcal{H}'(w+v_k(t)) - \mathcal{H}'(w)\right),\mathcal{H}'(w+v_k(t)) - \mathcal{H}'(w)\right\rangle_{L^2(\R)}
\end{align*}
for all $t \in [-t_0,t_0]$ and $k \in \N$. Since both $\mathcal{J}$ and $\mathcal{J}_2(w+v_k(t))$ are skew-symmetric differential operators on $L^2(\R)$, we have $\partial_t \Lambda_\gamma(v_k(t)) = 0$ for all $t \in [-t_0,t_0]$ and $k \in \N$.
\end{proof}

\subsection{The inverse-modulated perturbation} \label{sec:inv_mod}

Let $T \in (0,\infty]$ and let $\psi \colon \R \times [0,T) \to \R$ be a smooth function satisfying
\begin{align} \label{e:prop_mod}
\|\psi_x(\cdot,t)\|_{L^\infty(\R)} < \frac12, \qquad \psi \in C^j\big([0,T),H^k(\R)\big), \qquad \psi(\cdot,0) \equiv 0
\end{align}
for all $t \in [0,T)$ and $j,k \in \N_0$. We define the inverse-modulated perturbation $\vt \colon \R \times [0,T) \to \R$ by
\begin{align} \label{e:def_inv_mod_2}
\begin{split}
\vt(x,t) &\coloneqq u(x+\psi(x,t),t) - w(x)\\
&= v(x + \psi(x,t),t) + w(x + \psi(x,t)) - w(x).
\end{split}
\end{align}
Using the smoothness of $w$ and $\psi$, the mean value theorem, and Lemma~\ref{lem:cont_mod}, it follows from Proposition~\ref{prop:globalwellpert} that $\vt$ possesses the same regularity properties as the unmodulated perturbation $v$.

\begin{corollary}[Regularity of the inverse-modulated perturbation] \label{cor:reg_vt}
Fix $m \in \N_0$. Let $v_0 \in H^{m+3}(\R)$, and let $v$ be as in Proposition~\ref{prop:globalwellpert}. Let $\psi \colon \R \times [0,T) \to \R$ be a smooth function satisfying~\eqref{e:prop_mod}. Then, the inverse-modulated perturbation, given by~\eqref{e:def_inv_mod_2}, satisfies
\begin{align*} \vt \in C\big([0,T),H^{m+3}(\R)\big) \cap C^1\big([0,T),H^m(\R)\big), \qquad \vt(0) = v_0. \end{align*} 
\end{corollary}

Using that both $u(t) = w+v(t)$ and $w$ solve to~\eqref{cKdV}, we find that the inverse-modulated perturbation $\vt(t)$ satisfies the quasilinear equation
\begin{align} \label{e:modperteq}
(\partial_t-\El)(\vt-w'\psi) = \mathcal{N}(\vt,\nabla \psi), \qquad \mathcal{N}(\vt,\nabla\psi) = \mathcal{Q}_1(\vt,\nabla \psi) + \partial_x\left(\mathcal{Q}_2(\vt,\psi_x)\right)\!,
\end{align}
where the nonlinear residuals are given by
\begin{align*}
\mathcal{Q}_1(\vt,\nabla \psi) &= \vt_x \mathcal{R}_1(\psi_x)\psi_t + w'(\mathcal{R}_1(\psi_x) - \mathcal{R}_1(0)) \psi_t + \mathcal{S}_2(\vt,\psi_x)\mathcal{R}_2(\psi_x)\psi_{xx}, \\ 
\mathcal{Q}_2(\vt,\psi_x) &= \mathcal{S}_1(\vt,\psi_x) + \mathcal{S}_2(\vt,\psi_x)\left(\mathcal{R}_1(\psi_x) - \mathcal{R}_1(0)\right),
\end{align*}
with
\begin{align} \label{e:defS}
\begin{split}
\mathcal{S}_1(\vt,\psi_x) &= -\frac{1}{2}\vt^2 - \vt_{xx}\left( \mathcal{R}_2(\psi_x)-\mathcal{R}_2(0)\right) - w''\left( \mathcal{R}_2(\psi_x)-\mathcal{R}_2(0) -\mathcal{R}_2'(0)\psi_x\right)\\ 
&\qquad + \vt_x \mathcal{R}_3(\psi_x)\psi_{xx}+w' \left( \mathcal{R}_3(\psi_x)-\mathcal{R}_3(0) \right) \psi_{xx},\\
\mathcal{S}_2(\vt,\psi_x) &= c \vt - \vt w - \frac12 \vt^2 - \vt_{xx} \mathcal{R}_2(\psi_x) - w'' (\mathcal{R}_2(\psi_x) - \mathcal{R}_2(0))\\ 
&\qquad + \left(\vt_x+w'\right)\mathcal{R}_3(\psi_x)\psi_{xx},
\end{split}
\end{align}
and $\mathcal{R}_k \colon (-1,1) \to \R$ given by
\begin{align} \label{e:defR_k}
\mathcal{R}_k(y) = \frac{1}{(1+y)^k}, \qquad k \in \N.
\end{align}
Details of the derivation of~\eqref{e:modperteq} are provided in Appendix~\ref{appendixPertEq}. 

With the aid of the continuous embedding $H^1(\R) \hookrightarrow L^\infty(\R)$, one readily obtains the following nonlinear bounds.

\begin{lemma}[Nonlinear bounds] \label{lem:nonlinear_bounds} Fix $R > 0$. Then, there exists a constant $C > 0$ such that
\begin{align*}
\|\mathcal{Q}_1(\vt,\psi_x,\psi_t) \|_{L^2(\R)} &\leq C \left(\|\vt\|_{H^{2}(\R)} + \|\psi_x\|_{H^2(\R)} + \|\psi_t\|_{L^2(\R)}\right)^2,\\
\|\mathcal{Q}_2(\vt,\psi_x) \|_{L^2(\R)} &\leq C\left( \|\vt\|_{H^2(\R)} + \|\psi_x\|_{H^1(\R)}\right)^2,
\end{align*}
for all $\vt \in H^2(\R)$ and $(\psi_x,\psi_t) \in H^2(\R)\times L^2(\R)$ with $\|\psi_x\|_{L^\infty(\R)} \leq \frac12$ and $\|\psi_x\|_{H^2(\R)} \leq R$.
\end{lemma}

Using Lemma~\ref{lem:semigroup}, Corollary~\ref{cor:reg_vt}, and identity~\eqref{e:prop_mod}, we may integrate~\eqref{e:modperteq} to obtain the Duhamel formula
\begin{align} \label{e:duhvt}
\vt(t)-w'\psi(t) = \eu^{t\El}v_0 + \int_0^t \eu^{(t-s)\El} \mathcal{N}(\vt(s),\nabla \psi(s)) \ds
\end{align}
for $t \in [0,T)$. In~\S\ref{sec:choice_phase_modulation}, we will use~\eqref{e:duhvt} to define a modulation function $\psi$ compatible with the coercivity estimate of Proposition~\ref{prop:coercivity}.

\subsection{Modulating the relative energy} \label{sec:mod_energy}

It follows from~\eqref{e:rep_Lyapunov} that the leading-order contribution of the relative energy $\Lambda_\gamma(v)$ is the quadratic form $\frac12 \langle \A_\gamma v,v\rangle_{L^2(\R)}$, where $\A_\gamma$ is the differential operator associated with the second variation of $E_\gamma$ about $w$, given by~\eqref{e:defLgamma}. In order to exploit the coercivity estimate from Proposition~\ref{prop:coercivity}, we rewrite $\Lambda_\gamma(v)$ in terms of the inverse-modulated perturbation $\vt$ and the modulation function $\psi$. In view of~\eqref{e:formal_expasion_mod_pert}, one expects the leading-order term of the modulated relative energy to be
$\frac12 \langle \A_\gamma(\vt-w'\psi), \vt-w'\psi\rangle_{L^2(\R)}$, which is a quadratic form amendable to the coercivity estimate of Proposition~\ref{prop:coercivity}.

Let $\psi \in H^2(\R)$ satisfy $\|\psi'\|_{L^\infty(\R)} \leq \frac12$. To express the relative energy $\Lambda_\gamma(v)$ in terms of $\psi$ and
\begin{align} \label{e:invmodpert}
\vt = v(\cdot + \psi(\cdot)) + w(\cdot + \psi(\cdot)) - w
\end{align}
for $v \in H^2(\R)$, we perform a change of variables to the first integral in 
\begin{align} \label{e:Lambda_improper}
\Lambda_\gamma(v) = \lim_{R \to \infty} \int_{-R}^R \mathcal{K}_\gamma(w+v) - \mathcal{K}_\gamma(w) \dx.
\end{align}
The bound $\|\psi'\|_{L^\infty(\R)} \leq \frac12$ implies that the map $h \colon \R \to \R$ given by $h(x) = x + \psi(x)$ is strictly increasing and invertible. Moreover, using the continuous embedding $H^1(\R) \hookrightarrow C_0(\R)$, we have that $R - h^{-1}(R) = \psi(h^{-1}(R)) \to 0$ as $R \to \pm\infty$. Thus, setting $\varkappa \coloneqq \mathcal{K}_\gamma(w + v)$ and making the change of variables $y = h(x)$, we infer
\begin{align} \label{e:subst_rule}
\int_{-R}^R \varkappa(y) \dy = \int_{h^{-1}(-R)}^{h^{-1}(R)} \varkappa(x+\psi(x)) \, (1+\psi'(x)) \dx.
\end{align}
Recalling~\eqref{e:def_Kgamma}, we find that $\varkappa(x+\psi(x))$ can be expressed in terms of
\begin{align} \label{e:diffexp}
\begin{split}
u(x+\psi(x)) &= \vt(x) + w,\\
u'(x+\psi(x)) &= \frac{\vt'(x)+w'(x)}{1+\psi'(x)}, \\
u''(x+\psi(x)) &= \frac{\vt'' (x)+w''(x)}{(1+\psi'(x))^2} - \frac{\vt'(x)+w'(x)}{(1+\psi'(x))^3}\psi''(x),
\end{split}
\end{align}
where $u = w+v$. Substituting~\eqref{e:diffexp} into~\eqref{e:subst_rule}, inserting the resulting expression into~\eqref{e:Lambda_improper}, and using that $R - h^{-1}(R) \to 0$ as $R \to \pm \infty$, we obtain the identity $\Lambda_\gamma(v) = \smash{\widetilde{\Lambda}_\gamma}(\vt,\psi)$, where $\smash{\widetilde{\Lambda}_\gamma} \colon H^2(\R) \times \smash{\{\psi \in H^2(\R) \colon \|\psi'\|_{L^\infty(\R)} \leq \tfrac{1}{2}\}} \to \R$ is the smooth nonlinear functional, given by
\begin{align} \label{e:defwidetildegamma}
\widetilde{\Lambda}_\gamma(\vt,\psi) = \int_{\R} \widetilde{\mathcal{K}}_\gamma(\vt,\psi') - \mathcal{K}_\gamma(w) \dx
\end{align}
for $\vt,\psi \in C_c^\infty(\R)$ with $\|\psi'\|_{L^\infty(\R)} \leq \frac12$. Here, we denote
\begin{align*}
\widetilde{\mathcal{K}}_\gamma(\vt,\phi) &=  \frac{1}{2(1+\phi)^3} \left(\vt''+w''-\frac{\vt'+w'}{1+\phi}\phi'\right)^2 -\frac{5(\vt+w)\left(\vt'+w'\right)^2}{6(1+\phi)}\\
&\qquad + \Bigg(\frac{5}{72} (\vt+w)^4 - \left(\frac16 c_1 + \frac12 c^2\right) (\vt+w)^2 + c_2(\vt+w)\\
&\qquad + \gamma\left(\frac{(\vt'+w')^2}{2(1+\phi)^2} - \frac16 (\vt+w)^3 + \frac12 c(\vt+w)^2 + c_1(\vt+w)\right)\!\Bigg)(1+\phi).
\end{align*}
Altogether, we have established the following result, expressing the relative energy $\Lambda_\gamma(v)$ in terms of $\vt$ and $\psi$.

\begin{proposition}[Modulation of the relative energy] \label{prop:mod_unmod_Lyapunov_eq}
Let $v \in H^2(\R)$ and let $\psi \in H^2(\R)$ satisfy $\|\psi'\|_{L^\infty(\R)} \leq \frac12$. Then, we have
\begin{align} \label{e:functionalid2}
\Lambda_\gamma(v) = \widetilde{\Lambda}_\gamma(\vt,\psi)
\end{align}
for all $\gamma \in \R$, where $\vt$ is given by~\eqref{e:invmodpert}.
\end{proposition}

Expanding the \emph{modulated relative energy} $\smash{\widetilde{\Lambda}}_\gamma(\vt,\psi)$ in terms of $\vt$ and $\psi$, we arrive at the desired nonlinear estimate.

\begin{lemma}[Expansion of the modulated relative energy] \label{lem:mod_energy}
Fix $R > 0$. There exists a constant $C > 0$ such that
\begin{align} \label{e:mod_energy_exp}
\left|\widetilde{\Lambda}_\gamma(\vt,\psi) - \frac12 \langle \A_\gamma(\vt-w'\psi), \vt-w'\psi\rangle_{L^2(\R)}\right| \leq C \left(\|\vt\|_{H^2(\R)} + \|\psi'\|_{H^1(\R)} \right)^3
\end{align}
for all $\gamma \in \R$ and $\vt, \psi \in H^4(\R)$ with $\|\psi'\|_{L^\infty(\R)} \leq \frac12$ and $\|\vt\|_{H^2(\R)}, \|\psi'\|_{H^1(\R)} \leq R$.
\end{lemma}
\begin{proof}
Set $X_j = \{\psi \in H^{j+1}(\R) \colon \|\psi'\|_{L^\infty(\R)} \leq \smash{\frac{1}{2}}\}$ for $j \in \N$. By definition, we have
\begin{align} \label{e:functionalid}
\widetilde{\Lambda}_\gamma(\vt,\psi) = \mathring{\Lambda}_\gamma(\vt,\psi')
\end{align}
for $(\vt,\psi) \in H^2(\R) \times X_1$, where the smooth nonlinear functional $\smash{\mathring{\Lambda}}_\gamma \colon H^2(\R) \times \{\phi \in H^1(\R) \colon \|\phi\|_{L^\infty(\R)} \leq \smash{\frac{1}{2}}\} \to \R$ is given by 
\begin{align*}
\mathring{\Lambda}_\gamma(\vt,\phi) = \int_{\R} \widetilde{\mathcal{K}}_\gamma(\vt,\phi) - \mathcal{K}_\gamma(w) \dx
\end{align*}
for $\vt,\phi \in C^\infty_c(\R)$ with $\|\phi\|_{L^\infty(\R)} \leq \frac12$.
Expanding $\mathring{\Lambda}_\gamma$, using that $\smash{\mathring{\Lambda}_\gamma(0,0)} = 0$, and employing~\eqref{e:functionalid}, we obtain a constant $C > 0$ such that
\begin{align} \label{e:exp}
\begin{split}
&\left|\widetilde{\Lambda}_\gamma(\vt,\psi) - \mathring{\Lambda}_\gamma'(0,0)\left[\begin{pmatrix} \vt \\ \psi'\end{pmatrix}\right] - \frac12 \mathring{\Lambda}_\gamma''(0,0)\left[\begin{pmatrix} \vt \\ \psi'\end{pmatrix}, \begin{pmatrix} \vt \\ \psi'\end{pmatrix} \right]\right|\\ 
&\qquad \leq C \left(\|\vt\|_{H^2(\R)} + \|\psi'\|_{H^1(\R)} \right)^3
\end{split}
\end{align}
for all $(\vt,\psi) \in H^2(\R) \times X_1$ with $\|\vt\|_{H^2(\R)},\|\psi'\|_{H^1(\R)} \leq R$. It follows from the identity~\eqref{e:functionalid} that
\begin{align} \label{e:der1}
\begin{split}
\mathring{\Lambda}_\gamma'(0,0)\left[\begin{pmatrix} \vt_1 \\ \psi_1'\end{pmatrix}\right] &= \widetilde{\Lambda}_\gamma'(0,0)\left[\begin{pmatrix} \vt_1 \\ \psi_1\end{pmatrix}\right],\\
\mathring{\Lambda}_\gamma''(0,0)\left[\begin{pmatrix} \vt_1 \\ \psi_1'\end{pmatrix},\begin{pmatrix} \vt_2 \\ \psi_2'\end{pmatrix} \right] &= \widetilde{\Lambda}_\gamma''(0,0)\left[\begin{pmatrix} \vt_1 \\ \psi_1\end{pmatrix},\begin{pmatrix} \vt_2 \\ \psi_2\end{pmatrix} \right]
\end{split}
\end{align}
for all $(\vt_1,\psi_2),(\vt_2,\psi_2) \in H^2(\R) \times X_1$. 

Thus, to establish~\eqref{e:mod_energy_exp}, it remains to compute the first and second variation of $\smash{\widetilde{\Lambda}}_\gamma$ at $(0,0)$. While these could in principle be obtained by collecting the linear and quadratic terms in $(\tilde v,\psi)$ in~\eqref{e:defwidetildegamma} and subsequently integrating by parts and using the Euler--Lagrange equations~\eqref{e:Euler_Lagrange_1} and~\eqref{e:Euler_Lagrange_2}, it is more convenient to instead differentiate the identity~\eqref{e:functionalid2} using the chain rule. To this end, we first infer from~\eqref{e:rep_Lyapunov} that
\begin{align} \label{e:Lambdagammader}
\Lambda_\gamma(0) = 0, \qquad \Lambda_\gamma'(0) = 0, \qquad \Lambda_\gamma''(0) \left[v_1,v_2 \right] = \langle \A_\gamma v_1,v_2\rangle_{L^2(\R)}
\end{align}
for $v_1,v_2 \in H^4(\R)$. Next, we introduce the map $V \colon H^4(\R) \times X_3 \to H^4(\R)$, given by
\begin{align*}
V(\vt,\psi) = \vt \circ (\mathrm{id} + \psi)^{-1} + w \circ (\mathrm{id} + \psi)^{-1} - w.
\end{align*}
As $\mathrm{id} + \psi \colon \R \to \R$ is invertible for each $\psi \in C^1(\R)$ with $\|\psi'\|_{L^\infty(\R)} \leq \frac12$, it follows from the mean value theorem and integration by substitution that $V$ is well-defined and twice Fr\'echet differentiable with first derivative
\begin{align} \label{e:Vder}
V'(0,0)\left[\begin{pmatrix} \vt \\ \psi\end{pmatrix}\right] = \vt - w'\psi
\end{align}
for $(\vt,\psi) \in H^4(\R) \times X_3$. By construction, $v = V(\vt,\psi)$ satisfies~\eqref{e:invmodpert}. Hence, Proposition~\ref{prop:mod_unmod_Lyapunov_eq} yields
\begin{align*}
\widetilde{\Lambda}_\gamma(\vt,\psi) = \Lambda_\gamma(V(\vt,\psi)),
\end{align*}
for all $(\vt,\psi) \in H^4(\R) \times X_3$.  Applying the chain rule to this identity, using that $V(0,0)=0$, and invoking~\eqref{e:Lambdagammader} and~\eqref{e:Vder}, we obtain
\begin{align*}
\widetilde{\Lambda}_\gamma'(0,0) = 0, \qquad \widetilde{\Lambda}_\gamma''(0,0) \left[\begin{pmatrix} \vt_1 \\ \psi_1\end{pmatrix},\begin{pmatrix} \vt_2 \\ \psi_2\end{pmatrix} \right] = \langle \A_\gamma (v_1-w'\psi_1),v_2-w'\psi_2\rangle_{L^2(\R)}
\end{align*}
for all $(\vt_1,\psi_1),(\vt_2,\psi_2) \in H^4(\R) \times X_3$. Combining this with~\eqref{e:exp} and~\eqref{e:der1} completes the proof.
\end{proof}

We may now use the coercivity estimate from Proposition~\ref{prop:coercivity} to bound the quadratic form in the expansion~\eqref{e:mod_energy_exp} from below.

\begin{corollary}[Lower bound on the modulated relative energy] \label{cor:lower_bound_energy} Let $I$ be the interval from Theorem~\ref{thm:diffusive_intro}. Fix $R > 0$, $k \in\N$, and $\gamma \in I$. 

Then, for each $\xi_1 \in (0,\frac{\pi}{\ell})$ sufficiently small, there exist constants $C,\alpha > 0$ and a smooth map $P_\gamma \colon (-\xi_1,\xi_1) \to B(L^2_\per(0,\ell))$ such that $P_\gamma(\xi)$ is the spectral projection associated with the smallest eigenvalue of the Bloch operator $\A_\gamma(\xi)$. It holds
\begin{align} \label{e:eigenfunction_Bloch}
\alpha \leq \|P_\gamma(\xi)w'\|_{L^2_\per(0,\ell)} \leq C
\end{align}
for all $\xi \in (-\xi_1,\xi_1)$. Moreover, for all $\vt_1,\vt_2 \in H^2(\R)$ and $\psi \in H^{k+1}(\R)$ obeying
\begin{itemize}
    \item[(i)] $\mathcal{B}(\vt_1)(\xi) \in \ker(P_\gamma(\xi))$ for all $\xi \in (-\xi_1,\xi_1)$;
    \item[(ii)] $\check{\psi}(\xi)$ is a constant function for each $\xi \in [-\frac{\pi}{\ell},\frac{\pi}{\ell})$;
    \item[(iii)] $\check{\psi}(\xi) \equiv 0$ for all $\xi \in [-\frac{\pi}{\ell},\frac{\pi}{\ell}) \setminus (-\xi_1,\xi_1)$;
    \item[(iv)] $\|\vt_1\|_{H^2(\R)},\|\vt_2\|_{H^2(\R)}, \|\psi'\|_{H^1(\R)} \leq R$ and $\|\psi'\|_{L^\infty(\R)} \leq \frac12$,
\end{itemize}
we have
\begin{align} \label{e:descoer}
\begin{split}
\widetilde{\Lambda}_\gamma(\vt_1 + \vt_2,\psi) &\geq \alpha \left(\|\vt_1\|_{H^2(\R)}^2 + \|\psi'\|_{H^k(\R)}^2\right)\\
&\qquad - \, C \left(\|\vt_2\|_{H^2(\R)}^2 + \left(\|\vt_1\|_{H^2(\R)} + \|\psi'\|_{H^1(\R)}\right)^3 \right).
\end{split}
\end{align}
\end{corollary}
\begin{proof}
The operator $-\A_\gamma$ is diffusively spectrally stable by Theorem~\ref{thm:diffusive_intro} with $\ker(\A_\gamma(0)) = \textnormal{Sp}\{w'\}$. Hence, Lemma~\ref{lem:low_freq} yields an open neighborhood $U_\gamma \subset \C$ of $0$ and an analytic map $P_\gamma \colon U_\gamma \to B(L^2_\per(0,\ell))$ such that $P_\gamma(\xi)$ is the spectral projection associated with the smallest eigenvalue of $\A_\gamma(\xi)$ for all $\xi \in U_\gamma$ and we have $P_\gamma(0) w' = w'$. Hence, for each $\xi_1 \in (0,\frac{\pi}{\ell})$ sufficiently small, there exist constants $C_1,\alpha_1 > 0$ such that $[-\xi_1,\xi_1] \subset U_\gamma$ and we have $\alpha_1 \leq \|P_\gamma(\xi)w'\| \leq C_1$ for $\xi \in (-\xi_1,\xi_1)$.

Using~\eqref{e:defLgamma} together with $\A_\gamma(0) w' = 0$, we obtain
\begin{align*}
\A_\gamma(w'\psi) &= w'\psi'''' + 4w''\psi''' + 6 w''' \psi'' + 4 w'''' \psi' + \left(\frac53 w- \gamma\right) \left(w' \psi'' + 2w'' \psi'\right) + \frac53 (w')^2 \psi'
\end{align*}
for $\psi \in H^4(\R)$. Consequently, integrating by parts and applying the Cauchy--Schwarz inequality, we obtain a constant $C_2 > 0$ such that
\begin{align} \label{e:CS1}
\begin{split}
\langle \A_\gamma \vt_2,z-w'\psi\rangle_{L^2(\R)} &= \langle \A_\gamma (z - w'\psi),\vt_2\rangle_{L^2(\R)}\\ 
&\leq C_2 \left(\|z\|_{H^2(\R)} + \|\psi'\|_{H^1(\R)} \right)\|\vt_2\|_{H^2(\R)}
\end{split}
\end{align}
for $z,\vt_2,\psi \in H^4(\R)$. On the other hand, by Proposition~\ref{prop:coercivity} there exists, for each $\xi_1 \in (0,\frac{\pi}{\ell})$ sufficiently small, a constant $\alpha > 0$ such that
\begin{align} \label{e:CS2}
\langle \A_\gamma (\vt_1 - w'\psi),\vt_1 - w'\psi\rangle_{L^2(\R)} \geq 2\alpha \left(\|\vt_1\|_{H^2(\R)}^2 + \|\psi'\|_{H^k(\R)}^2\right)
\end{align}
for all $\vt_1\in H^4(\R)$ and $\psi \in H^k(\R) \cap H^4(\R)$ satisfying conditions~(i)-(iii). Combining~\eqref{e:CS1} and~\eqref{e:CS2} with Lemma~\ref{lem:mod_energy} and Young's inequality yields a constant $C > 0$ such that~\eqref{e:descoer} holds for $\vt_1,\vt_2 \in H^4(\R)$ and $\psi \in H^k(\R)\cap H^4(\R)$ satisfying conditions~(i)-(iv). The general case with $\vt_1,\vt_2 \in H^2(\R)$ and $\psi \in H^k(\R)$ follows by density.
\end{proof}

\subsection{Choice of modulation function} \label{sec:choice_phase_modulation}

Let $v_0 \in H^3(\R)$. Denote by $v \in \smash{C\big(\R,H^3(\R)\big)} \cap\smash{C^1\big(\R,L^2(\R)\big)}$ the associated solution to~\eqref{e:unmod_pert_eq}, established in Proposition~\ref{prop:globalwellpert}. We now make a judicious choice of a smooth modulation function $\psi \colon \R \times [0,T) \to \R$ and decompose the inverse-modulated perturbation~\eqref{e:def_inv_mod_2} as 
\begin{align} \label{e:decompvt}
\vt(t) = \vt_1(t) + \vt_2(t)
\end{align} 
for $t \in [0,T)$. The modulation function and decomposition are chosen so that $\vt_1(t)$ and $\psi(t)$ satisfy both~\eqref{e:prop_mod} and the conditions (i)-(iii) in Corollary~\ref{cor:lower_bound_energy}, and $\vt_2(t)$ is a residual term. This construction allows us to apply the coercivity estimate from Corollary~\ref{cor:lower_bound_energy} to control $\vt_1(t)$ and $\psi_x(t)$ in the forthcoming nonlinear argument, while $\vt_2(t)$ and $\psi_t(t)$ can be estimated directly from their integral representations.

Our starting point for defining the modulation function $\psi(t)$ and the residual $\vt_2(t)$ is the Bloch transform of the Duhamel representation~\eqref{e:duhvt}, which reads
\begin{align} \label{e:duhvt_bloch}
\begin{split}
[\mathcal{B}\vt_1(t)](\xi) + [\mathcal{B}\vt_2(t)](\xi) - w'\check{\psi}(\xi,t) &= \eu^{t\El(\xi)}\check{v}_0(\xi)\\
&\qquad + \, \int_0^t \eu^{(t-s)\El(\xi)} \mathcal{B}[\mathcal{N}(\vt(s),\nabla \psi(s))](\xi) \ds
\end{split}
\end{align}
for $t \in [0,T)$ and $\xi \in [-\frac\pi\ell,\frac\pi\ell)$. Let $I$ be as in Theorem~\ref{thm:diffusive_intro} and take $\gamma \in I$. Moreover, let $\xi_1 \in \smash{(0,\frac{\pi}{\ell})}$ be so small that Corollary~\ref{cor:lower_bound_energy} applies. Using~\eqref{e:eigenfunction_Bloch} and the fact that the rank-one spectral projection $P_\gamma(\xi)$ is orthogonal, we note that the condition~(i) of Corollary~\ref{cor:lower_bound_energy} is equivalent to $\smash{\langle [\mathcal{B}\vt_1(t)](\xi), P_\gamma(\xi) w'\rangle_{L^2(0,\ell)}} = 0$ for all $\xi \in (-\xi_1,\xi_1)$. Hence, abbreviating the right-hand side of~\eqref{e:duhvt_bloch} by $F(\xi,t)$ and taking the inner product with $P_\gamma(\xi) w'$, we find that conditions~(i) and~(ii) of Corollary~\ref{cor:lower_bound_energy} can hold only if we have
\begin{align*}
\frac{\left\langle[\mathcal{B}\vt_2(t)](\xi),P_\gamma(\xi) w'\right\rangle_{L^2(0,\ell)}}{\|P_\gamma(\xi)w'\|_{L^2(0,\ell)}^2} 
- \check{\psi}(\xi,t) &= \frac{\left\langle F(\xi,t),P_\gamma(\xi) w'\right\rangle_{L^2(0,\ell)}}{\|P_\gamma(\xi)w'\|_{L^2(0,\ell)}^2}
\end{align*}
for $\xi \in (-\xi_1,\xi_1)$. Returning to physical space, now motivates the following definitions
\begin{align} \label{e:def_psi}
\psi(t) &= \chi(t) s_p(t) v_0 + \int_0^t \chi(t-s) s_p(t-s) \mathcal{N}(\vt(s),\nabla \psi(s)) \ds,\\
\vt_2(t) &= s_r(t) v_0 + \int_0^t s_r(t-s) \mathcal{N}(\vt(s),\nabla \psi(s)) \ds \label{e:def_vt2}
\end{align}
for $t \in [0,T)$, where the linear operators $s_p(t), s_r(t) \colon L^2(\R) \to L^2(\R)$ are defined via their Bloch transforms
\begin{align*}
\left[\mathcal{B} s_p(t) v\right](\xi) &= -\mathbbm{1}_{(-\xi_1,\xi_1)}(\xi) \frac{\left\langle \eu^{t \El(\xi)} \check{v}(\xi), P_\gamma(\xi) w'\right\rangle_{L^2(0,\ell)}}{\|P_\gamma(\xi)w'\|_{L^2(0,\ell)}^2},\\
\left[\mathcal{B} s_r(t) v\right](\xi) &= (1-\chi(t)) \mathbbm{1}_{(-\xi_1,\xi_1)}(\xi) P_\gamma(\xi) \eu^{t\El(\xi)} \check{v}(\xi)
\end{align*}
for $\xi \in [-\frac\pi\ell,\frac\pi\ell)$ and $t \geq 0$. Here, $\mathbbm{1}_{(-\xi_1,\xi_1)}$ denotes the indicator function of the interval $(-\xi_1,\xi_1)$, ensuring that also condition~(iii) of Corollary~\ref{cor:lower_bound_energy} is satisfied. Moreover, $\chi \colon [0,\infty) \to [0,1]$ is a smooth temporal cut-off function satisfying $\chi(t) = 0$ for $t \in [0,1]$ and $\chi(t) = 1$ for $t \in [2,\infty)$. Its presence ensures that formula~\eqref{e:def_psi} defines $\psi(t)$ iteratively, rather than implicitly as the solution of a fixed-point problem; see the paragraph proceeding Proposition~\ref{prop:psi_prop} for details. 

We show that $\vt_1(t), \psi(t)$, and $\vt_2(t)$ are well defined through~\eqref{e:decompvt},~\eqref{e:def_psi}, and~\eqref{e:def_vt2} and that they satisfy~\eqref{e:prop_mod} and conditions (i)-(iii) in Corollary~\ref{cor:lower_bound_energy}. To this end, we first establish bounds on the linear propagators $s_p(t)$ and $s_r(t)$.

\begin{lemma}[Linear estimates on low-frequency propagators] \label{lem:linear}
Fix $j,k,l \in \N_0$. There exist constants $C,M > 0$ such that
\begin{align*}
\big\|\partial_t^j \partial_x^k s_p(t) \partial_x^l z\big\|_{L^2(\R)} \leq C\eu^{M t} \|z\|_{L^2(\R)}, \qquad \big\|\partial_t^j \partial_x^k s_r(t) \partial_x^l z\big\|_{L^2(\R)} \leq C\|z\|_{L^2(\R)}
\end{align*}
for all $z \in H^l(\R)$ and $t \geq 0$.
\end{lemma}
\begin{proof}
It follows from the properties of the Floquet--Bloch transform (see Appendix~\ref{appendixFB}) that
\begin{align*}
\partial_t^j \mathcal{B} \partial_x^k s_p(t) \partial_x^l z = (\iu \xi)^k \mathbbm{1}_{(-\xi_1,\xi_1)}(\xi) \frac{(-1)^{l+1} \left\langle \check{z}(\xi), (\partial_x + \iu \xi)^l \left(\El(\xi)^j \eu^{t \El(\xi)}\right)^* P_\gamma(\xi) w'\right\rangle_{L^2(0,\ell)}}{\|P_\gamma(\xi)w'\|_{L^2(0,\ell)}^2}
\end{align*}
for $\xi \in [-\frac\pi\ell,\frac\pi\ell)$, $z \in H^l(\R)$, and $t \geq 0$. Noting that the eigenfunction $P_\gamma(\xi) w'$ of the Bloch operator $\A_\gamma(\xi)$ lies in $H^m_\per(0,\ell)$ for each $m \in \N$, using~\eqref{e:eigenfunction_Bloch}, and applying  Lemma~\ref{lem:semigroup}, we obtain constants $C_1,M > 0$ such that
\begin{align*}
\left\|(\partial_x + \iu \xi)^l \left(\El(\xi)^j \eu^{t \El(\xi)}\right)^* P_\gamma(\xi) w'\right\|_{L^2(0,\ell)} \leq C_1 \eu^{Mt}
\end{align*}
for all $\xi \in (-\xi_1,\xi_1)$ and $t \geq 0$. Combining the latter with~\eqref{e:eigenfunction_Bloch}, Parseval's identity, and the Cauchy--Schwarz inequality, we arrive at the desired bound on $s_p(t)$. Finally, observing that 
\begin{align*}
P_\gamma(\xi) y = \frac{\langle y, P_\gamma(\xi) w'\rangle_{L^2(0,\ell)}}{\|P_\gamma(\xi) w'\|_{L^2(0,\ell)}^2} P_\gamma(\xi)w'
\end{align*}
for $y \in L^2_\per(0,\ell)$ and $\xi \in (-\xi_1,\xi_1)$, we obtain
\begin{align*}
\left[\mathcal{B} s_r(t)z\right](\xi) = (\chi(t)-1) P_\gamma(\xi) w' \left[\mathcal{B} s_p(t)z\right](\xi)
\end{align*}
for all $z \in L^2(\R)$, $\xi \in [-\frac\pi\ell,\frac\pi\ell)$, and $t \geq 0$. Hence, the desired bound for $s_r(t)$ follows readily from the one for $s_p(t)$ upon using the estimate~\eqref{e:eigenfunction_Bloch} and the fact that $1-\chi$ vanishes on $[2,\infty)$.
\end{proof}

We argue that~\eqref{e:def_psi} provides an iterative definition of $\psi(t)$ as long as $\|\psi_x(t)\|_{L^\infty(\R)} \leq \frac12$. Since $\chi(t) = 0$ for all $t \in [0,1]$, the modulation function $\psi(t)$ vanishes identically on $[0,1]$. Now suppose that $\psi$ has been defined on $[0,n]$ for some $n \in \N$ and satisfies $\|\psi_x(s)\|_{L^\infty(\R)} < \frac12$ for all $s \in [0,n]$. Then, for $t \in [0,1]$, we define $\psi(n+t)$ through~\eqref{e:def_psi}. This is possible because, thanks to the fact that $\chi$ vanishes on $[0,1]$, the right-hand side depends only on $\vt|_{[0,n]}$ and $\psi|_{[0,n]}$.

The next result shows that the iterative definition of $\psi(t)$ yields the desired properties~\eqref{e:prop_mod} and provides an integral representation for the temporal derivative $\psi_t(t)$. Moreover, $\psi(t)$ and $\vt_1(t)$ satisfy conditions (i)-(iii) in Corollary~\ref{cor:lower_bound_energy}.

\begin{proposition}[Properties of the modulation function] \label{prop:psi_prop}
There exists $T \in (0,\infty]$ such that the following hold:
\begin{itemize}
\item[(i)] The modulation function $\psi \colon \R \times [0,T) \to \R$, given by~\eqref{e:def_psi}, is smooth and satisfies~\eqref{e:prop_mod}. Moreover, its Bloch transform $\smash{\check{\psi}}(\xi,t)$ is a constant function for each $\xi \in [-\frac{\pi}{\ell},\frac{\pi}{\ell})$ and $t \in [0,T)$, which vanishes identically for all $\xi \in [-\frac{\pi}{\ell},\frac{\pi}{\ell}) \setminus (-\xi_1,\xi_1)$. Finally, $T < \infty$ implies that
\begin{align} \limsup_{t \uparrow T} \|\psi_x(t)\|_{L^\infty(\R)} = \frac12. \label{e:blowup3}\end{align}
\item[(ii)] The inverse-modulated perturbation, given by~\eqref{e:def_inv_mod_2}, satisfies
\begin{align} \vt \in C\big([0,T),H^3(\R)\big) \cap C^1\big([0,T),L^2(\R)\big), \qquad \vt(0) = v_0. \label{e:regvt2}\end{align} 
and obeys the quasilinear equation~\eqref{e:modperteq} and the associated Duhamel formula~\eqref{e:duhvt} for all $t \in [0,T)$. 
\item[(iii)] The inverse-modulated perturbation decomposes as~\eqref{e:decompvt}, where
\begin{align}
\vt_1,\vt_2 \in C\big([0,T),H^3(\R)\big) \cap C^1\big([0,T),L^2(\R)\big) \label{e:reg_vt1}
\end{align}
and we have $[\mathcal{B}\vt_1(t)](\xi) \in \ker(P_\gamma(\xi))$ for all $t \in [0,T)$ and $\xi \in (-\xi_1,\xi_1)$.
\item[(iv)] It holds
\begin{align} \label{e:int_psit}
\begin{split}
\psi_t(t) &= s_p(0)\bigg( \El \vt(t) + \partial_x \left(w'\psi_{xx}(t) + 2 w''\psi_x(t)\right) - \partial_t s_r(t)v_0\\
&\qquad - \int_0^t \partial_t s_r(t-s) \mathcal{N}(\vt(s),\nabla \psi(s))\ds\bigg)
\end{split}
\end{align}
for $t \in [0,T)$. 
\end{itemize}
\end{proposition}
\begin{proof}
Since $\chi$ vanishes on $[0,1]$, we have $\psi(t)\equiv 0$ and $\vt(t)=v(t)$ for all $t \in [0,1]$. Using Proposition~\ref{prop:globalwellpert}, it is clear that $\psi$ and $\vt$ satisfy~\eqref{e:prop_mod} and~\eqref{e:regvt2} with $[0,T)$ replaced by $[0,1]$. Now suppose that $\psi$ and $\vt$ are defined on $[0,n]$ for some $n \in \N$ and satisfy~\eqref{e:prop_mod} and~\eqref{e:regvt2} with $[0,T)$ replaced by $[0,n]$. By the continuous embedding $H^1(\R) \hookrightarrow L^\infty(\R)$, the nonlinearity $s \mapsto \mathcal{N}(\vt(s),\nabla \psi(s))$ belongs to $C\big([0,n],L^2(\R)\big)$. Combining this observation with Lemma~\ref{lem:linear} and the fact that $\chi$ vanishes on $[0,1]$, we may use~\eqref{e:def_psi} to extend $\psi$ to $[0,n+1]$ so that $\psi \in \smash{C^j\big([0,n+1],H^k(\R)\big)}$ for all $j,k \in \N_0$. Moreover, the definition of the propagator $s_p(t)$ immediately implies that the Bloch transform $\smash{\check{\psi}}(\xi,t)$ is a constant function in $L^2_\per(0,\ell)$ for all $\xi \in [-\frac{\pi}{\ell},\frac{\pi}{\ell})$ and $t \in [0,n+1]$, and vanishes identically whenever $\xi \in [-\frac{\pi}{\ell},\frac{\pi}{\ell}) \setminus (-\xi_1,\xi_1)$. Finally, if there exists $s_0 \in [0,n+1]$ such that $\|\psi_x(s_0)\|_{L^\infty(\R)} \geq \frac12$, then continuity implies that
\begin{align*}
T \coloneqq \min\big\{s \in [0,n+1] : \|\psi_x(s)\|_{L^\infty(\R)} = \tfrac12\big\} < \infty
\end{align*}
exists. In this case, we set $I_n = [0,T)$; otherwise we let $I_n = [0,n+1]$. Defining $\vt$ subsequently through~\eqref{e:def_inv_mod_2} and invoking Corollary~\ref{cor:reg_vt}, we conclude that $\vt$ fulfills~\eqref{e:regvt2} with $[0,T)$ replaced by $I_n$. Hence, as argued in~\S\ref{sec:inv_mod}, $\vt$ satisfies equation~\eqref{e:modperteq}, as well as the associated Duhamel formula~\eqref{e:duhvt}. The assertions (i) and (ii) now follow by induction on $n$.

Lemma~\ref{lem:linear}, together with the fact that $s \mapsto \mathcal{N}(\vt(s),\nabla \psi(s))$ belongs to $C([0,T),L^2(\R))$, shows that $\vt_2$, given by~\eqref{e:def_vt2}, is well defined and satisfies $\vt_2 \in \smash{C\big([0,T),H^k(\R)\big)} \cap \smash{C^1\big([0,T),L^2(\R)\big)}$ for all $k \in \N_0$, with temporal derivative
\begin{align} \label{e:der_vt2}
\partial_t \vt_2(t) = \partial_t s_r(t) v_0 + \int_0^t \partial_t s_r(t-s) \mathcal{N}(\vt(s),\nabla \psi(s)) \ds + \mathcal{N}(\vt(t),\nabla \psi(t))
\end{align}
for $t \in [0,T)$. Combining this with~\eqref{e:regvt2} and the decomposition~\eqref{e:decompvt}, we immediately obtain~\eqref{e:reg_vt1}. Finally, taking the inner product of~\eqref{e:duhvt_bloch} with $P_\gamma(\xi)w'$ and using the orthogonality of $P_\gamma(\xi)$, we conclude that $[\mathcal{B}\vt_1(t)](\xi) \in \ker(P_\gamma(\xi))$ for all $\xi \in (-\xi_1,\xi_1)$ and $t \in [0,T)$, thereby completing the proof of the third assertion.  

To prove the last assertion, we note that $Aw' = 0$ implies
\begin{align} \label{e:psiwkernelid}
\El(w'\psi(t)) = -\partial_x \left(w'\psi_{xx}(t) + 2w''\psi_x(t) \right)
\end{align}
for $t \in [0,T)$. Applying $s_p(0)$ to equation~\eqref{e:modperteq}, noting that $[\partial_t \mathcal{B}\vt_1(t)](\xi) \in \ker(P_\gamma(\xi))$ for $\xi \in (-\xi_1,\xi_1)$, and rearranging terms, we arrive at
\begin{align*}
\partial_t \psi(t) = s_p(0) \big(\El \left(\vt(t) - w'\psi(t)\right) + \mathcal{N}(\vt(t),\nabla \psi(t)) - \partial_t \vt_2(t)\big)
\end{align*}
for $t \in [0,T)$. Combining the latter with~\eqref{e:der_vt2} and~\eqref{e:psiwkernelid} yields the fourth assertion.
\end{proof}

\section{Nonlinear stability analysis} \label{sec:nonlinear_stab_analysis}

In this section, we first prove our main result, Theorem~\ref{t:main_result}, establishing the nonlinear modulational stability of periodic traveling-wave solutions to the KdV equation under $H^3$-localized perturbations. We then use higher-order conserved quantities arising from the KdV hierarchy to extend the result to Sobolev spaces of arbitrarily high order.

\subsection{Proof of main result}

We close a nonlinear stability argument by tracking the norms of the inverse-modulated perturbation $\vt(t)=\vt_1(t)+\vt_2(t)$ and the space-time gradient $\nabla \psi(t)$ of the modulation function. More precisely, we control $\vt_1(t)$ and $\psi_x(t)$ via the energy estimate established in Corollary~\ref{cor:lower_bound_energy}, while the norms of $\vt_2(t)$ and $\psi_t(t)$ are estimated directly from their respective integral representations~\eqref{e:def_vt2} and~\eqref{e:int_psit}.

\begin{proof}[Proof of Theorem~\ref{t:main_result}]
Fix $k \in \N$ with $k \geq 2$. Let $v_0 \in H^3(\R)$ and set $E_0 = \|v_0\|_{H^2(\R)}$. Take $\gamma \in I$, where $I$ is the interval from Theorem~\ref{thm:diffusive_intro}. Let $\xi_1 \in \smash{(0,\frac{\pi}{\ell})}$ be so small that Corollary~\ref{cor:lower_bound_energy} applies. Let $v \in \smash{C\big(\R,H^3(\R)\big) \cap C^1\big(\R,L^2(\R)\big)}$ be the solution to~\eqref{e:unmod_pert_eq} with $v(0) = v_0$, established in Proposition~\ref{prop:globalwellpert}. 

Proposition~\ref{prop:psi_prop} yields $T \in (0,\infty]$ and a smooth modulation function $\psi \colon \R \times [0,T) \to \R$ satisfying~\eqref{e:prop_mod} and~\eqref{e:int_psit}, whose Bloch transform $\smash{\check{\psi}}(\xi,t)$ is a constant function for all $\xi \in [-\frac{\pi}{\ell},\frac{\pi}{\ell})$ and $t \in [0,T)$ such that $\smash{\check{\psi}}(\xi,t) = 0$ for all $\xi \in [-\frac{\pi}{\ell},\frac{\pi}{\ell}) \setminus (-\xi_1,\xi_1)$. Moreover, we have that $T < \infty$ implies~\eqref{e:blowup3}. Finally, the inverse-modulated perturbation, given by~\eqref{e:def_inv_mod_2}, decomposes as $\vt(t) = \vt_1(t) + \vt_2(t)$, where $\vt_1,\vt_2 \in \smash{C\big([0,T),H^3(\R)\big) \cap C^1\big([0,T),L^2(\R)\big)}$ satisfy~\eqref{e:def_vt2} and $[\mathcal{B}\vt_1(t)](\xi) \in \ker(P_\gamma(\xi))$ for all $t \in [0,T)$ and $\xi \in (-\xi_1,\xi_1)$.

The template function $\eta \colon [0,T) \to \R$, defined by
\begin{align*}
\eta(t) = \sup_{s \in [0,t]} &\big(\|\vt_1(s)\|_{H^2(\R)} + \|\vt_2(s)\|_{H^2(\R)} + \|\nabla \psi(s)\|_{H^k(\R)} + \|\psi_x(s)\|_{L^\infty(\R)}\big),
\end{align*}
is continuous and nondecreasing. Our first objective is to prove the existence of an $E_0$-independent constant $C_0\geq 1$ such that, for every $t \in [0,T)$ satisfying $\eta(t)\leq \tfrac12$, we have
\begin{align} \label{e:key}
\eta(t) \leq C_0\left(E_0 + \eta(t)^{\frac32}\right), \qquad \eta(0) \leq C_0E_0.
\end{align}
In the second step, we show that, provided $E_0<(2C_0)^{-3}$, the above estimates imply that $\eta(t)\leq 2C_0E_0$ for all $t \in [0,T)$, which in turn yields $T=\infty$. Finally, we construct a smooth modulation function $\varphi \colon \R \times [0,\infty) \to \R$ from $\psi$ such that~\eqref{e:modulational_KdV} is satisfied. The estimate~\eqref{e:orbital_KdV} then follows by an application of the mean value theorem.

In order to prove~\eqref{e:key}, we take $t \in [0,T)$ such that $\eta(t) \leq \tfrac12$ and let $s \in [0,t]$. We denote by $C \geq 1$ any constant, which is independent of $s$, $t$, and $E_0$. Using $\eta(t) \leq \tfrac12$, applying Lemmas~\ref{lem:nonlinear_bounds} and~\ref{lem:linear} to the Duhamel formula~\eqref{e:def_vt2}, and recalling that $1-\chi$ vanishes on $[2,\infty)$, we bound
\begin{align} \label{e:vt2_bound}
\|\vt_2(s)\|_{H^2(\R)} \leq C\left(E_0 + \int_{s-2}^s \eta(r)^2 \dr \right) \leq C\left(E_0 + \eta(t)^2\right)
\end{align}
and
\begin{align} \label{e:init_vt2}
\|\vt_2(0)\|_{H^2(\R)} \leq CE_0.
\end{align}
Next, we employ Propositions~\ref{prop:conservation} and~\ref{prop:mod_unmod_Lyapunov_eq} and identity~\eqref{e:rep_Lyapunov} to infer
\begin{align*}
\widetilde{\Lambda}_\gamma(\vt(s),\psi(s)) = \Lambda_\gamma(v(s)) = \Lambda_\gamma(v_0) \leq CE_0^2, 
\end{align*}
where we use that $\eta(0) \leq \tfrac12$. Hence,  Corollary~\ref{cor:lower_bound_energy}, the estimates~\eqref{e:vt2_bound} and $\eta(t) \leq \tfrac12$, and Young's inequality yield
\begin{align}
\begin{split}
\|\vt_1(s)\|_{H^2(\R)}^2 + \|\psi_x(s)\|_{H^k(\R)}^2 &\leq C\left(\widetilde{\Lambda}_\gamma(\vt(s),\psi(s)) + \left(E_0 +\eta(t)^2\right)^2 +\eta(t)^3\right)\\ &\leq C\left(E_0^2 + \eta(t)^3\right). 
\end{split} \label{e:vt1_bound}
\end{align}
Furthermore, Proposition~\ref{prop:psi_prop} implies
\begin{align} \label{e:init_psi_vt1}
\psi(0) \equiv 0, \qquad \vt_1(0) = v_0 - \vt_2(0).
\end{align}
Finally, applying Lemmas~\ref{lem:nonlinear_bounds} and~\ref{lem:linear} to~\eqref{e:int_psit}, employing~\eqref{e:vt2_bound} and~\eqref{e:vt1_bound},  recalling that $1-\chi$ vanishes on $[2,\infty)$, and using that $\eta(t) \leq \tfrac12$, we arrive at
\begin{align} \label{e:psit_bound}
\|\psi_t(s)\|_{H^k(\R)} \leq C\left(E_0 + \eta(t)^{\frac32} + \int_{s-2}^s \eta(r)^2 \dr\right)\leq C\left(E_0 + \eta(t)^{\frac32}\right).
\end{align}
Combining the estimates~\eqref{e:vt2_bound},~\eqref{e:init_vt2},~\eqref{e:vt1_bound},~\eqref{e:init_psi_vt1}, and~\eqref{e:psit_bound}, recalling that $\eta(t) \leq \tfrac12$, and using the continuous embedding $H^1(\R) \hookrightarrow L^\infty(\R)$, we establish a $t$- and $E_0$-independent constant $C_0 \geq 1$ such that the key estimates in~\eqref{e:key} hold.

Next, we use~\eqref{e:key} to close a nonlinear argument. Let $\delta = (2C_0)^{-3}$ and assume that $E_0 \in (0,\delta)$. We argue by contradiction. Suppose that there exists $t \in [0,T)$ such that $\eta(t) > 2C_0E_0$. Since $\eta(0)\leq C_0E_0$ and $\eta$ is continuous, there exists $t_0 \in (0,t)$ such that $\eta(t_0)=2C_0E_0<\frac12$. Applying~\eqref{e:key} at time $t_0$, we obtain
\begin{align*}
\eta(t_0) \leq C_0\left(E_0 + (2C_0E_0)^{\frac32}\right)
< 2C_0E_0,
\end{align*}
contradicting the definition of $t_0$. Therefore, we have 
\begin{align} \label{e:etaineq}
\eta(t)\leq 2C_0E_0 < \frac12
\end{align}
for all $t \in [0,T)$, which shows that~\eqref{e:blowup3} cannot occur, implying $T=\infty$.

We now construct a smooth modulation function $\varphi \colon \R \times [0,\infty) \to \R$ such that the estimate~\eqref{e:modulational_KdV} is satisfied. Since $\|\psi_x(t)\|_{L^\infty(\R)} \leq \frac12$, the map $\mathrm{id} + \psi(\cdot,t) \colon \R \to \R$ is invertible for all $t \geq 0$. Define $\smash{\tilde{\varphi}} \colon \R \times [0,\infty) \to \R$ by $\smash{\tilde{\varphi}}(\cdot,t) = (\mathrm{id} + \psi(\cdot,t))^{-1} - \mathrm{id}$. Since $\psi$ is smooth, $\smash{\tilde{\varphi}}$ is smooth by the implicit function theorem. Moreover, since $\psi(t)$ vanishes identically at $t = 0$, so does $\smash{\tilde{\varphi}}$. Differentiating the identity $\smash{\tilde{\varphi}}(x,t) = -\psi(x + \smash{\tilde{\varphi}}(x,t),t)$ with respect to time, we obtain
\begin{align*}
\smash{\tilde{\varphi}}_t(x,t) = -\frac{\psi_t(x + \tilde\varphi(x,t),t)}{1+\psi_x(x + \tilde\varphi(x,t),t)}
\end{align*}
for $x \in \R$ and $t \geq 0$. Thus, applying Lemma~\ref{lem:equivalence} with $f = -\psi(\cdot,t)$, $f = -\psi_t(\cdot,t)/(1+\psi_x(\cdot,t))$, and $f = v \circ (\mathrm{id} + \psi(\cdot,t)) + w \circ (\mathrm{id} + \psi(\cdot,t)) - w$, while using~\eqref{e:etaineq}, we obtain a $t$- and $E_0$-independent constant $C > 0$ such that
\begin{align} \label{e:nlest0}
\|\nabla \smash{\tilde{\varphi}}(t)\|_{H^k(\R)} &\leq C \|\nabla \psi(t)\|_{H^k(\R)} \leq 2CC_0E_0,
\end{align}
and
\begin{align}
\|v + w - w(\cdot + \smash{\tilde{\varphi}}(\cdot,t))\|_{H^2(\R)} &\leq C\|\vt(t)\|_{H^2(\R)} \leq 2CC_0E_0  \label{e:nlest1}
\end{align}
for all $t \geq 0$. Hence, using the fundamental theorem of calculus and $\smash{\tilde{\varphi}}(0) \equiv 0$, we infer
\begin{align} \label{e:nlest2}
\|\smash{\tilde{\varphi}}(t)\|_{L^2(\R)} \leq 2CC_0E_0 t, 
\end{align}
and
\begin{align} \label{e:nlest3}
\begin{split}
\|\smash{\tilde{\varphi}}(t) - \smash{\tilde{\varphi}}(x_*,t)\|_{L^2(I)} &\leq \left(\int_I \int_I \int_I \left|\smash{\tilde{\varphi}}_x(y,t)\right| \left|\smash{\tilde{\varphi}}_x(z,t)\right| \dy \dz \dx\right)^{\frac12}\\ 
&\leq R \|\smash{\tilde\varphi}_x(t)\|_{L^2(\R)} \leq 2 R C C_0 E_0
\end{split}
\end{align}
for all $t \geq 0$, each interval $I \subset \R$ of size $R > 0$, and all $x_* \in I$. Define $u,\varphi \colon \R \times [0,\infty) \to \R$ by $u(x,t) = v(x-ct,t) + w(x-ct)$ and $\varphi(x,t) = \smash{\tilde\varphi}(x-ct,t)$, respectively. Then, $\varphi$ is smooth and, by Proposition~\ref{prop:globalwellpert}, $u$ is a global classical solution to~\eqref{KdV} with initial condition $u(0) = w+v_0$, satisfying~\eqref{e:regtooth}. Finally, using the translational invariance of the norm, the bound~\eqref{e:modulational_KdV} follows directly from estimates~\eqref{e:nlest0},~\eqref{e:nlest1}, and~\eqref{e:nlest2}, whereas the bound~\eqref{e:orbital_KdV} is a consequence of~\eqref{e:nlest1} and~\eqref{e:nlest3} and the mean value theorem.
\end{proof}

\begin{remark}
We stress that there are alternative definitions of $\psi(t)$ that also allow to close the nonlinear argument. Although the associated coercivity estimate is slightly more delicate, the spectral projection $P_\gamma(\xi)$ used in~\eqref{e:def_psi} may be replaced by any projection $\smash{\tilde{P}(\xi)}$, depending continuously on $\xi$ and satisfying $\smash{\tilde{P}}(0)=P_\gamma(0)$. Consequently, the modulation function $\varphi(t)$ in Theorem~\ref{t:main_result} is by no means unique.
\end{remark}

\subsection{Higher-order Sobolev spaces}

Let $u_{\text{tw}}(x,t) = w(x-ct)$ be a periodic traveling-wave solution to the KdV equation~\eqref{KdV} with wave speed $c \in \R$. Then, $w$ is a stationary solution to~\eqref{cKdV}. The Hamiltonians corresponding to the equations in the KdV hierarchy give rise to an infinite sequence of conserved quantities for~\eqref{cKdV}, each of which possesses $w$ as a critical point, cf.~\cite{Lax_1975_periodic,miura1968korteweg}. As observed in~\cite{bona2004stability}, these quantities can be used to extend orbital stability results to Sobolev spaces of arbitrarily high order. The underlying idea is that, once orbital stability has been established in lower-order Sobolev spaces, relative energies associated with these conserved quantities provide control of higher-order Sobolev norms of the perturbation in terms of lower-order norms.

However, in contrast to~\cite{bona2004stability}, our analysis only yields control of  spatiotemporal modulations of the perturbation. We therefore modulate the higher-order relative energies in the same manner as the relative energy $\Lambda_\gamma$ in~\S\ref{sec:mod_energy}. This construction produces an infinite sequence of modulated relative energies that control higher-order derivatives of the inverse-modulated perturbation.

\begin{proposition}[Higher-order modulated relative energies] \label{prop:higher_order_conserved}
Fix $s \in \N$ with $s \geq 3$. Set $X_s = \{\psi \in H^{s+1}(\R) : \|\psi'\|_{L^\infty(\R)} \leq \frac12\}$. There exists a nonlinear functional $F_s \colon H^s(\R) \times X_s \to \R$ with the following properties:
\begin{itemize}
    \item[(i)] We have
    \begin{align} 
    \label{e:hoest0}
    F_s(\vt,\psi) = F_s(v,0)\end{align}
    for all $v \in H^s(\R)$ and $\psi \in X_s$, where $\vt$ is given by~\eqref{e:invmodpert}.
    \item[(ii)] For each $R > 0$ there exists a constant $C > 0$ such that
    \begin{align} \label{e:hoest1}
    \begin{split}
    \left\|\partial_x^s \vt\right\|^2_{L^2(\R)} &\leq C\left(F_s(\vt,\psi) + \|\vt\|^2_{H^{s-1}(\R)} + \|\psi'\|_{H^{s-1}(\R)}^2\right),\\
    F_s(v,0) &\leq C \|v\|_{H^s(\R)}^2 \end{split}\end{align}
    for all $v,\vt \in H^s(\R)$ and $\psi \in X_s$ with $\|\vt\|_{H^{s-1}(\R)}, \|v\|_{H^{s-1}(\R)}, \|\psi'\|_{H^s(\R)} \leq R$.
    \item[(iii)] For each $v_0 \in H^s(\R)$ and $t \in \R$, we have
    \begin{align} \label{e:hoest3}
        F_s(v(t),0) = F_s(v_0,0),
    \end{align}
    where $v \in \smash{C\big(\R,H^s(\R)\big) \cap C^1\big(\R,H^{s-3}(\R)\big)}$ is the global classical solution to~\eqref{e:unmod_pert_eq} with $v(0) = v_0$, established in Proposition~\ref{prop:globalwellpert}.
\end{itemize}
\end{proposition}
\begin{proof}
We make use of the sequence of conserved quantities introduced by Gardner, Kruskal, and Miura~\cite{miura1968korteweg}, which arise as Hamiltonians of the equations in the KdV hierarchy and possess $w$ as critical point. In particular, using the characterization in~\cite[\S2]{miura1968korteweg} and~\cite[Equation (3.2)' and Theorem~3.2]{Lax_1975_periodic}, such a quantity takes the form of a nonlinear functional $G_s \colon H^s(\R) \to \R$, given by
\begin{align*}
G_s(u) = \int_\R \mathcal{P}_s(u) \dx.
\end{align*} 
Here, the density has the structure
\begin{align}
\mathcal{P}_s(u)
= a_s (\partial_x^s u)^2 + b_s (\partial_x^{s-1} u)(\partial_x^s u) + c_s (\partial_x^{s-1} u)^2 + d_s \partial_x^s u + e_s \partial_x^{s-1} u + f_s, \label{e:defPS}
\end{align}
with $a_s>0$ and $b_s$, $c_s$, $d_s$, $e_s$, and $f_s$ polynomials in $u,\partial_x u,\ldots,\smash{\partial_x^{s-2}}u$, whose coefficients are real. Here, $\mathcal{P}_s(u)$ is quadratic in $u$ and its derivatives. Moreover, if $u \colon \R \times \R \to \R$ is a sufficiently regular pointwise solution of~\eqref{cKdV} such that its spatial derivatives up to order $s+3$ exist and its temporal derivative $\partial_t u$ is $s$ times differentiable with respect to space, then the local conservation law
\begin{align} \label{e:conslaw}
\partial_t \mathcal{P}_s(u(x,t)) = \partial_x \mathcal{X}_s(u(x,t))
\end{align}
holds for all $x,t \in \R$, where $\mathcal{X}_s(u)$ is a polynomial in $u,\partial_x u, \ldots,\partial_x^{s+2} u$ with real coefficients. 

Following the approach in~\S\ref{sec:energy}, we interchange subtraction and integration in the formal difference $G_s(w+v) - G_s(v)$ and obtain a smooth nonlinear functional $Q_s \colon H^s(\R) \to \R$, given by
\begin{align}
Q_s(v) = \int_\R \mathcal{P}_s(w+v) - \mathcal{P}_s(w) \dx = \lim_{R \to \infty} \int_{-R}^R \mathcal{P}_s(w+v) - \mathcal{P}_s(w) \dx \label{e:defHS}
\end{align}
for $v \in C^\infty_c(\R)$, which is well-defined since $w$ is a critical point of $G_s$. Moreover, if $v_0 \in H^{s+4}(\R)$ and $v \in \smash{C\big(\R,H^{s+4}(\R)\big) \cap C^1\big(\R,H^{s+1}(\R)\big)}$ is the associated global classical solution to~\eqref{e:unmod_pert_eq} with $v(0) = v_0$, established in Proposition~\ref{prop:globalwellpert}, then~\eqref{e:conslaw} yields
\begin{align*}
\partial_t Q_s(v(t)) = \int_\R \partial_x(\mathcal{X}_s(w+v) - \mathcal{X}_s(w)) \dx = 0,
\end{align*}
implying conservation of the relative energy:
\begin{align} \label{e:conslaw2}
Q_s(v(t)) = Q_s(v_0)
\end{align}
for all $t \in \R$, where we used that both $w$ and $w+v(t)$ are sufficiently regular pointwise solutions to~\eqref{cKdV}. Using approximation by regularized solutions as in the proof of Proposition~\ref{prop:conservation}, cf.~\cite[Section~3.1]{erdogan2016dispersive}, we find that~\eqref{e:conslaw2} holds for all $v_0 \in H^s(\R)$ and associated global solutions $v \in \smash{C\big(\R,H^s(\R)\big) \cap C^1\big(\R,H^{s-3}(\R)\big)}$ to~\eqref{e:unmod_pert_eq} with $v(0) = v_0$, established in Proposition~\ref{prop:globalwellpert}.

Let $v \in H^s(\R)$ and $\psi \in X_s$. Set $u = w+v$. We proceed as in~\S\ref{sec:mod_energy} and perform a change of variables  to the first integral in~\eqref{e:defHS} to obtain a modulated relative energy. Using that the map $h \colon \R \to \R$ given by $h(x) = x + \psi(x)$ is strictly increasing and hence invertible, we make the change of variables $y = h(x)$, which yields~\eqref{e:subst_rule} with $\varkappa \coloneqq \mathcal{P}_s(w + v)$. By the chain rule, we have
\begin{align} \label{e:subst_rule2}
(\partial_x^j u)(x+\psi(x)) = \frac{\partial_x^j(\vt(x) + w(x))}{(1+\psi'(x))^j} + \sum_{m = 1}^{j-1} \alpha_m \partial_x^m(\vt(x) + w(x)), \qquad j = 1,\ldots,s,
\end{align}
for $x \in \R$, where the coefficients $\alpha_m$ are polynomials in $\partial_x \psi,\ldots,\partial_x^j \psi$, and $1/(1+\psi')$, cf.~\eqref{e:diffexp}. Thus, substituting~\eqref{e:subst_rule2} into $\varkappa(x+\psi(x))$ and~\eqref{e:subst_rule} into~\eqref{e:defHS}, recalling that $w$ is a critical point of $G_s$, using~\eqref{e:defPS}, and noting that $R - h^{-1}(R) \to 0$ as $R \to \pm\infty$, we arrive at
\begin{align} \label{e:hoest4}
Q_s(v) = F_s(\vt,\psi)
\end{align}
where $\vt$ is given by~\eqref{e:invmodpert} and $F_s \colon H^s(\R) \times X_s \to \R$ is a nonlinear functional of the form
\begin{align*}
F_s(\vt,\psi) = \int_\R \mathcal{Q}_s(\vt,\psi) \dx
\end{align*}
for $\vt,\psi \in C^\infty_c(\R)$ with $\|\psi'\|_{L^\infty(\R)} \leq \frac12$, with density
\begin{align*}
\mathcal{Q}_s(\vt,\psi) = \frac{a_s}{(1+\psi')^{2s-1}} (\partial_x^s \vt)^2 + \tilde{b}_s (\partial_x^{s-1} \vt)(\partial_x^s \vt) + \tilde{c}_s (\partial_x^{s-1} \vt)^2 + 
\tilde{d}_s \partial_x^s \vt + \tilde{e}_s \partial_x^{s-1} \vt + \tilde{f}_s,
\end{align*}
where $\tilde{b}_s \ldots, \tilde{f}_s$ are polynomials in $w,\partial_x w,\ldots,\partial_x^s w, \vt,\partial_x \vt,\ldots,\partial_x^{s-2} \vt$, $\partial_x \psi,\ldots,\partial_x^s \psi$, and $1/(1+\psi')$ with real coefficients. Here, $\mathcal{Q}_s(\vt,\psi)$ is quadratic in $\vt$, $\psi_x$, and their derivatives. This immediately yields~\eqref{e:hoest1} upon invoking the continuous embedding $H^1(\R)\hookrightarrow L^\infty(\R)$, together with the Cauchy--Schwarz and Young inequalities, and the fact that $\|\psi'\|_{L^\infty(\R)} \leq \frac12$ for all $\psi\in X_s$. On the other hand, setting $\psi\equiv0$ in~\eqref{e:hoest4} gives $Q_s(v)=F_s(v,0)$ for every $v\in H^s(\R)$. Applying~\eqref{e:hoest4} once more and using~\eqref{e:conslaw2}, we obtain~\eqref{e:hoest0} and~\eqref{e:hoest3}, respectively, which completes the proof.
\end{proof}

Using the higher-order modulated relative energies, established in Proposition~\ref{prop:higher_order_conserved}, we obtain the following corollary of Theorem~\ref{t:main_result}, which yields nonlinear stability in Sobolev spaces of arbitrarily high order.

\begin{corollary}[Nonlinear modulational stability in higher-order Sobolev spaces] \label{c:main_result}
Let $u_{\text{tw}}(x,t) = w(x-ct)$ be a periodic traveling-wave solution to the KdV equation~\eqref{KdV} with wave speed $c \in \R$ and profile $w\colon \R \to \R$ of fundamental period $\ell > 0$. Fix $k,m \in \N_0$. Then, there exist constants $M,\delta > 0$ such that, whenever $v_0 \in H^{m+3}(\R)$ satisfies
\begin{align}E_0 \coloneqq\|v_0\|_{H^{m+3}(\R)} \leq \delta, \label{e:initho}\end{align}
there exist a global classical solution
\begin{align*}
u \in C\big([0,\infty),H^{m+3}(\R) \oplus H^{m+3}_\per(0,\ell)\big) \cap C^1\big([0,\infty),H^{m}(\R) \oplus H^{m}_\per(0,\ell)\big)
\end{align*}
to~\eqref{KdV} with initial condition $u(0) = w+v_0$ and a smooth modulation function $\varphi \colon \R \times [0,\infty) \to \R$ such that
\begin{align} \label{e:modulational_KdV2}
\|u(t) - w(\cdot - c t + \varphi(\cdot,t))\|_{H^{m+3}(\R)} + \|\nabla \varphi(t)\|_{H^k(\R)} + \frac{\|\varphi(t)\|_{L^2(\R)}}{1+t} \leq ME_0
\end{align}
for all $t \geq 0$. Moreover, we have
\begin{align} \label{e:orbital_KdV2}
\inf_{\phi \in \R} \|u(t) - w(\cdot + \phi)\|_{H^{m+3}(x_* - R,x_* + R)} \leq M R E_0 
\end{align}
for each $t \geq 0$, $x_* \in \R$, and $R > 0$.
\end{corollary}
\begin{proof}
Fix $k,m \in \N_0$ with $k \geq m+3$. Let $\delta > 0$ be as in the proof of Theorem~\ref{t:main_result}. Let $v_0 \in H^{m+3}(\R)$ satisfy~\eqref{e:initho}. Finally, let $v$, $\vt$, $\psi$, $\smash{\tilde\varphi}$, and $\varphi$ be as in the proof of Theorem~\ref{t:main_result}. 

By estimate~\eqref{e:etaineq}, there exists an $E_0$-, and $t$-independent constant $C > 0$ such that 
\begin{align} \label{e:highorder_init}
\|\vt(t)\|_{H^2(\R)} + \|\nabla \psi(t)\|_{H^k(\R)} \leq CE_0 < \frac12
\end{align}
for all $t \geq 0$. By Proposition~\ref{prop:globalwellpert} and Corollary~\ref{cor:reg_vt}, we have $v,\vt \in   \smash{C\big([0,\infty),H^{m+3}(\R)\big)} \cap \smash{C^1\big([0,\infty),H^{m}(\R)\big)}$. Thus, combining~\eqref{e:highorder_init} with Proposition~\ref{prop:higher_order_conserved} yields $E_0$- and $t$-independent constants $C_{1,2} > 0$ such that 
\begin{align} \label{e:induc_est}
\begin{split}
\|\partial_x^s \vt(t)\|_{L^2(\R)}^2 &\leq C_1\left(F_s(\vt(t),\psi(t)) + \|\vt(t)\|^2_{H^{s-1}(\R)} + \|\psi_x(t)\|_{H^{s-1}(\R)}^2\right) \\
&\leq C_2\left(E_0^2 + \|\vt(t)\|^2_{H^{s-1}(\R)}\right)
\end{split}
\end{align}
for $t \geq 0$ and $s \in \{3,\ldots,m+3\}$, where we used $F_s(\vt(t),\psi(t)) = F_s(v(t),0) = F_s(v_0,0)$. Applying estimate~\eqref{e:induc_est} inductively, while using~\eqref{e:highorder_init} and Lemma~\ref{lem:equivalence}, we obtain $E_0$- and $t$-independent constants $C_{3,4} > 0$ such that
\begin{align*}
\|v + w - w(\cdot + \smash{\tilde{\varphi}}(\cdot,t))\|_{H^{m+3}(\R)} \leq C_3\|\vt(t)\|_{H^{m+3}(\R)} \leq C_4E_0
\end{align*}
for $t \geq 0$. Combining the latter with~\eqref{e:nlest2},~\eqref{e:nlest3}, and~\eqref{e:highorder_init}, the mean value theorem, and the translational invariance of the norm, we arrive at~\eqref{e:modulational_KdV2} and~\eqref{e:orbital_KdV2}.
\end{proof}

\section{Discussion and outlook} \label{sec:discussion}

We outline several possible extensions of our approach and explore their application to Hamiltonian stability problems.

\subsection{Multiple symmetries and their interactions}

In~\S\ref{sec:intro_principle}, we described how our nonlinear stability theory applies to general Hamiltonian systems with a one-parameter symmetry group under diffusive spectral stability conditions. However, many canonical Hamiltonian models, including the nonlinear Schr\"odinger (NLS) and complex Klein--Gordon equations, are invariant under the action of symmetry groups of dimension greater than one. We expect that our approach can be extended to this setting by adapting the notion of diffusive spectral stability so that the dimension of the kernel of the second variation coincides with that of the symmetry group. Modulating the wave along each symmetry direction by a spatiotemporal modulation function then presents the additional challenge of controlling possible interactions between modulation functions in the nonlinear argument. 

Nevertheless, we expect that no interaction occurs at the level of the quadratic form associated with the modulated relative energy, since the $0$-eigenvalue of the second variation is semisimple. This should allow one to recover a coercivity estimate at the cost of a derivative in a similar way as in the present one-symmetry setting. We are currently investigating this extension in the context of 
periodic cnoidal waves in the defocusing NLS equation, which is invariant under a $2$-parameter symmetry group generated by translations and phase rotations.

\subsection{Asymptotic stability and modulational dynamics}

Our main result, Theorem~\ref{t:main_result}, shows that the perturbed solution $u(t)$ to~\eqref{KdV} remains close to a modulated periodic wave of the form $w_{\textnormal{mod}}(x,t) = w(x-ct+\varphi(x,t))$ and that the spacetime gradient $\nabla \varphi(t)$ stays small. This naturally raises the question of whether this can be strengthened to an \emph{asymptotic stability} result. That is, does $u(t)$ converge to $w_{\textnormal{mod}}(t)$ as $t \to \infty$ and can the leading-order asymptotics of $\varphi(t)$ be determined?

Although the KdV equation, when posed on the whole real line, exhibits dispersion, one cannot expect that all small initial data decay to $0$ in any $L^p(\R)$-norm as $t \to \infty$. Indeed, due to Galilean invariance, the KdV equation admits traveling solitons of arbitrarily small amplitude. In fact, like several dispersive Hamiltonian systems, the KdV equation exhibits \emph{soliton resolution}~\cite{schuur1986asymptotic,martel2001asymptotic}, meaning that, generically, sufficiently localized solutions decompose into a finite number of solitons plus a dispersively decaying radiation term as $t \to \infty$; see~\cite{jendrej2025recent} and references therein. However, for sufficiently small, polynomially localized initial data, it was shown in~\cite{Ifrim_2023_dispersive} that solutions to~\eqref{KdV} initially satisfy linear-like dispersive decay bounds. This regime persists on a quartic time scale with respect to the size of initial data, after which nonlinear effects, in particular the emergence of solitons, become relevant.

Soliton resolution also holds in the present setting of a periodic background solution. Using the inverse scattering transform, the long-time dynamics of solutions arising from polynomially localized initial data on top of a (quasi-)periodic background wave were characterized in~\cite{Mikikits_Leitner_2012_Asymptotics,Egorova_2011_Cauchy}. Although the analysis in~\cite{Mikikits_Leitner_2012_Asymptotics,Egorova_2011_Cauchy} does not yield a stability result, it shows that the solution asymptotically consists of a finite number of solitons, each propagating with its own characteristic velocity on top of the periodic background. Between each two consecutive solitons, the asymptotics of the solution is described by a spatial translate of the underlying periodic wave together with a dispersively decaying radiation term. As a consequence, the modulation function $\varphi(t)$ may develop linearly growing plateau states, consistent with the discussion in~\S\ref{sec:main_result}. In particular, asymptotic stability in the sense that $u(t)$ converges to $w_{\textnormal{mod}}(t)$ as $t \to \infty$ cannot be expected.

A description of the long-time dynamics that is more naturally aligned with the present modulational framework is provided by Whitham modulation theory. Up to translations, the periodic waves in~\eqref{KdV} form a three-parameter family, which can be parameterized by mass, momentum (i.e.~$L^2_\per(0,\ell)$-norm), and wave number; see~\eqref{e:cnoidal_form}. A multiscale expansion then formally yields a system of equations governing the slow evolution of these parameters, known as the \emph{Whitham modulation system}~\cite{Whitham_1965_non-linear}. Higher-order corrections of this Whitham system are expected to effectively describe the dynamics of the derivative $\varphi_x(t)$, which corresponds to the local wave number, on long time scales; see~\cite{rodrigues2018linear} for a linear validation. After diagonalization of its linear part, such a higher-order Whitham system consists of three nonlinearly coupled KdV equations with distinct characteristic speeds. Since components propagate with different velocities, one expects nonlinear interactions between components to be asymptotically negligible for sufficiently localized data, so that the large-time behavior is essentially governed by decoupled KdV dynamics. This heuristic suggests that the higher-order Whitham system should itself exhibit soliton resolution. In particular, one expects the local wave number $\varphi_x(t)$ to asymptotically decompose into a finite collection of solitons traveling with distinct characteristic speeds together with a dispersively decaying remainder term as $t\to\infty$. Although this picture suggests that $\varphi_x(t)$ does not exhibit dispersive decay as $t\to\infty$, one may nevertheless expect linear-like dispersive decay on quartic time scales in the size of the initial perturbation as in~\cite{Ifrim_2023_dispersive}.

Establishing the validity of the higher-order Whitham modulation system, its soliton resolution, and the dispersive decay of $u(t)$ towards $w_{\textnormal{mod}}(t)$ on quartic time scales remain open problems. We expect that the stability results in this paper, which provide global-in-time control of $u(t) - w_{\textnormal{mod}}(t)$ and $\nabla \varphi(t)$ in any Sobolev norm, see Corollary~\ref{c:main_result}, can serve as key a-priori energy estimates in potential future works addressing these questions.

\subsection{Modulational data}

Since the modulated perturbation equation~\eqref{e:modperteq} and the modulated energy~\eqref{e:defwidetildegamma} depend only on spatial or temporal \emph{derivatives} of the modulation function $\psi(t)$, the nonlinear stability argument can be closed by only controlling the spacetime gradient $\nabla \psi(t)$ and the inverse-modulated perturbation $\vt(t)$. The $L^2$-norm of $\psi(t)$ itself is estimated a-posteriori. This is reflected in the choice of the template function $\eta(t)$ in the proof of Theorem~\ref{t:main_result}. 

Based on these observations, we expect that our approach can be employed to prove that solutions $u(t)$ of~\eqref{KdV} with \emph{modulational data} of the form
\begin{align*}
u(x,0) = w(x + \varphi_0(x)) + v_0(x),
\end{align*}
where $v_0,\varphi_0 \colon \R \to \R$ are such that $\|\nabla \varphi_0\|_{H^1(\R)}$ and $\|v_0\|_{H^2(\R)}$ are small, stay close to a modulated periodic wave $w_{\textnormal{mod}}(x,t) = w(x - ct + \varphi(x,t))$. We emphasize that $\varphi_0$ need not be localized, and may even be unbounded. 

Nonlinear stability results for modulational data have been obtained for periodic waves in dissipative systems~\cite{sandstede2012diffusive,johnson_2014_behavior,Iyer_2019_mixing,zumbrun_2024_forward,Alexopoulos2025Modulation}, as well as for plane waves in the complex Klein--Gordon and NLS equations~\cite{bukieda2025orbital,zhidkov2001korteweg}. To handle modulational data, the initial condition for the modulation function $\varphi(t)$ must satisfy $\varphi(0) = \varphi_0$. In particular, modulation functions are no longer $L^2$-localized, and their Floquet--Bloch transforms are only defined in the sense of tempered distributions. This creates a technical challenge in obtaining a coercivity estimate in the present setting. This issue was absent in previous works~\cite{bukieda2025orbital,zhidkov2001korteweg} on plane waves in Hamiltonian systems with gauge symmetry, where a passage to polar coordinates reduces the plane wave to a constant state and bypasses the Floquet--Bloch transform.

\subsection{Indefinite energies}

Although various Hamiltonian systems supporting periodic waves admit conserved energies whose second variation about the wave yields a diffusively spectrally stable operator (see Remark~\ref{rem:diff_energies_in_other_systems}), there are important exceptions, particularly among non-integrable Hamiltonian systems, for which no such conserved energy exists. An example is the \emph{generalized KdV equation}
\begin{align} \label{gKdV}
u_t + u_{xxx} + \partial_x(u^{p+1}) = 0
\end{align}
with parameter $p > 0$. Since~\eqref{gKdV} is non-integrable for $p \neq 1,2$, it does not possess an infinite hierarchy of conserved quantities, and variational stability arguments must instead rely on the Hamiltonian itself. Using a Sturm--Liouville argument similar to that in~\S\ref{sec:diffusive_KdV_intro}, one finds that the second variation of the Hamiltonian about a periodic wave has, when posed on $L^2(\R)$, negative bands of spectrum. Consequently, it cannot correspond to a diffusively spectrally stable operator, obstructing the modulational approach developed in this paper. This obstruction appears to be a limitation of the method rather than an indication of instability: indeed, the analysis in~\cite{Haragus_2008_spectra} shows that periodic waves in~\eqref{gKdV} are spectrally stable with respect to $L^2(\R)$-perturbations for any $p \in (0,2)$.

Orbital stability of periodic waves in~\eqref{gKdV} subject to co-periodic perturbations has been established by characterizing the wave as a strict minimizer of the Hamiltonian subject to a momentum constraint; see~\cite{johnson2009nonlinear,Angulo_2008_positivity}. We are currently investigating whether such constrained variational arguments can be lifted to Floquet--Bloch space, with the aim of extending our framework to indefinite energies whose second variation has finitely many negative spectral bands. Such an extension would pave the way toward a nonlinear stability result analogous to Theorem~\ref{t:main_result} for periodic traveling-wave solutions to~\eqref{gKdV} with $p\in(0,2)$.

\subsection{Applicability to other stability problems with critical essential spectrum} We expect that the modulational approach developed in this paper has broad applicability to stability problems in Hamiltonian systems with critical essential spectrum. Relevant examples include the stability of dark solitons and kinks in NLS-type equations, as well as the stability of planar waves in multidimensional Hamiltonian systems.

\emph{Planar waves} arise by trivially extending a one-dimensional traveling wave in the transverse spatial directions. Using constrained variational arguments, orbital stability results for planar solitary waves have been obtained for perturbations that are periodic in the transverse directions with a fixed period~\cite{Mizumachi_2012_stability,Rousset_2012_stability,Yamazaki_2017_stability}. For spatially localized perturbations, however, these methods break down because translational invariance in the transverse directions renders the spectrum of the second variation purely essential, so that no coercivity on a finite-codimensional constraint space can be expected. By adapting the present approach to account for spatiotemporal modulation in the transverse directions, we believe that orbital stability with respect to $L^2$-localized perturbations can be obtained. Related modulational techniques have been proven successful in the stability analysis of planar waves in dissipative systems~\cite{Kapitula_1997_multidimensional,deRijk_2026_stability}. We note that, under additional assumptions such as zero-mean conditions or polynomial/exponential localization of perturbations, complementary asymptotic stability results for planar solitary waves have been established by exploiting dispersive decay~\cite{Mizmachi_2015_stability,Mizumachi_2018_stability,Cuccagna_2008_asymptotic,Mizumachi_2020_stability}.

The defocusing NLS equation and its variants admit front-type solutions connecting nonzero asymptotic states. These end states may be constant, as in the case of kinks or black solitons, or periodic plane waves, as in the case of dark solitons. Owing to the rotational symmetry of the equation, such nontrivial background states typically generate a neutral branch of essential spectrum for the second variation of any conserved energy that touches the origin. The absence of a spectral gap at $0$ then precludes the direct application of standard variational stability arguments. A well-established approach for overcoming this obstruction is to work in the \emph{energy space}
\begin{align*}
\big\{v \in H^1_{\mathrm{loc}}(\R) : v' \in L^2(\R),\ |v|^2 - 1 \in L^2(\R)\big\},
\end{align*}
endowed with the distance
\begin{align*}
d(v_1,v_2) = |v_1(0) - v_2(0)| + \big\|v_1' - v_2'\big\|_{L^2(\R)} + \big\||v_1|^2 - |v_2|^2\big\|_{L^2(\R)},
\end{align*}
or variants thereof, which effectively factor out phase rotations at spatial infinity, thereby restoring coercivity, and enabling proofs of orbital stability; see, for example,~\cite{holmer_2025_orbital,Bethuel_2008_orbital,Lin_2002_stability,Gerard_2009_orbital,Alejo_2024_orbital,Alama_2015_domain} and the references therein. We expect that the modulational approach developed in this paper offers an alternative mechanism for addressing the same obstruction. Rather than encoding asymptotic phase rotations through the choice of function space, the method incorporates them through a spatiotemporal phase modulation of the wave. This yields a coercivity estimate at the expense of a derivative of the phase modulation. An advantage of the modulational framework is that it remains within standard $L^2$-based Hilbert spaces and, through the Floquet--Bloch transform, extends naturally to settings such as dispersive shocks, where the asymptotic state is a general periodic wave rather than a constant state or a plane wave.

\appendix

\section{Floquet--Bloch transform} \label{appendixFB}

We collect some properties of the Floquet--Bloch transform that are relevant for our analysis. A more detailed account can be found in~\cite{Reed1978Methods,Scarpellini_1999_Stability,Kuchment_1993_operator,Lewin_2024_spectral}.

Fix $\ell > 0$. The Floquet--Bloch transform
\begin{align*}
\mathcal{B}\colon L^2(\R) \rightarrow L^2\left(\left(-\tfrac{\pi}{\ell},\tfrac{\pi}{\ell}\right),L_\textnormal{per}^2(0,\ell)\right)
\end{align*}
is given by
\begin{align*}
\mathcal{B}f(\xi,x) = \sum_{n \in \Z} \eu^{\frac{2\pi\iu}{\ell}n x}\hat{f} \left(\xi + \frac{2\pi}{\ell} n\right),
\end{align*}
for $f \in \mathcal{S}(\R)$ and can be extended to $f \in L^2(\R)$ by density of the Schwartz space $\mathcal{S}(\R) \subset L^2(\R)$. Here,
\begin{align*}
\hat{f}(k) = \int_\R \eu^{- \iu k x} f(x) \dx
\end{align*}
denotes the Fourier transform. We use the shorthand notation $\check{f} = \mathcal{B}f$ throughout this paper. The Floquet--Bloch transform is an isomorphism with inverse given by
\begin{align*}
    f(x) = \frac{1}{2\pi} \int_{-\frac{\pi}{\ell}}^{\frac{\pi}{\ell}} \eu^{\iu \xi x} \check{f}(\xi,x) \dxi
\end{align*}
for $f \in \mathcal{S}(\R)$. It satisfies Parseval's identity
\begin{align*}
    \langle f, g \rangle_{L^2(\R)} = \frac{1}{2\pi \ell}\int_{-\frac{\pi}{\ell}}^{\frac{\pi}{\ell}} \langle \check{f}(\xi), \check{g}(\xi) \rangle_{L^2(0,\ell)} \dxi
\end{align*}
for $f, g \in L^2(\R)$ and admits a physical-space representation, which reads
\begin{align*}
\check{f}(\xi,x) = \ell \sum_{n \in \Z} \eu^{-\iu \xi (x+n \ell)} f(x+n\ell)
\end{align*}
for $f \in \mathcal{S}(\R)$. The latter readily yields the identity
\begin{align*} 
    \widecheck{fg} = \check{f}g.
\end{align*}
for $f \in L^2(\R)$ and $\ell$-periodic $g \in C(\R)$. Finally, it follows from the properties of the Fourier transform that 
\begin{align*}
\mathcal{B}\left(\partial_x f \right)(\xi) = (\partial_x + \iu \xi) \check{f}(\xi)
\end{align*}
for $f \in H^1(\R)$. 

\section{Analysis of profile equations} \label{appendixPhPA}

Let $u_{\text{tw}}(x,t) = w(x-ct)$ be a periodic traveling-wave solution to the KdV equation~\eqref{KdV} with wave speed $c \in \R$ and profile $w\colon \R \to \R$ of fundamental period $\ell > 0$. We briefly describe how a standard analysis of the associated Hamiltonian profile equations yields the cnoidal-wave representation~\eqref{e:cnoidal_form} of $u_{\mathrm{tw}}$, as well as explicit expressions for the wave speed $c$ and the integration constants $c_{1,2} \in \R$, appearing in the energy~\eqref{e:def_energy}, in terms of the wave's parameters.

First, we observe that, if $u(x,t)$ is a solution to~\eqref{KdV}, then, by Galilean, translational, and scaling invariance,
\begin{align} \label{e:Galilean}
\tilde{u}(x,t) = h + \kappa^2 u(\kappa(x-x_0-ht),\kappa^3(t-t_0))
\end{align}
is also a solution for arbitrary parameters $h, \kappa, x_0,t_0 \in \R$. Using the Galilean and translational symmetries, we may therefore normalize the profile so that the minimum value of $w$ is $0$, which is attained at $x = \pm \ell/2$.

Since $w$ is a stationary solution to~\eqref{cKdV} and~\eqref{e:KdV2}, it satisfies the Euler--Lagrange equations~\eqref{e:Euler_Lagrange_1} and~\eqref{e:Euler_Lagrange_2}. Evaluating these at $x = \frac12 \ell$ yields the integration constants $c_1 = w''(\frac12 \ell)$ and $c_2 = -w''''(\frac12 \ell)$. Both Euler--Lagrange equations admit a Hamiltonian formulation, and conservation of the associated Hamiltonians yields
\begin{align} \label{e:HamiltonianODE}
\begin{split}
\frac12 c w^2 - \frac16 w^3 - \frac12 (w')^2 + c_1 w &= 0,\\ 
-\frac16 c_1 w^2 - \frac12 c^2 w^2 + \frac5{72} w^4 + \frac56 w (w')^2 - \frac12 (w'')^2 + w''' w' + c_2 w &= -\frac12 c_1^2.
\end{split}
\end{align}
The first equation in~\eqref{e:HamiltonianODE} is a separable first-order ODE. Its solution with initial condition $w(\frac\ell2) = 0$ is $\ell$-periodic for $c \in (-4,4)$ and given by the cnoidal wave
\begin{align} \label{e:wave_with_kappa} 
w(x) = 12 \kappa^2 \Eps \, \textnormal{cn}^2(\kappa x;\Eps) = 12 \kappa^2 \Eps \, \cos^2(\textnormal{am}(\kappa x;\Eps))
\end{align}
with elliptic modulus
\begin{align} \label{e:elliptic_modulus_with_l}
\Eps=\frac{c+4\kappa^2}{8\kappa^2} \in (0,1)
\end{align}
and rescaled wave number
\begin{align} \label{e:def_l}
\kappa = \frac{2 K(\Eps)}{\ell} = \int_0^{\frac\pi2} \frac{2}{\ell \sqrt{1 - \Eps \sin^2(\theta)}} \, \textnormal{d}\theta,
\end{align}
where the invertible function $K \colon (0,1) \to (\frac{\pi}{2},\infty)$ is the complete elliptic integral of the first kind and $\phi = \textnormal{am}(x;\Eps)$ is the Jacobi amplitude, defined through
\begin{align*}
x = \int_0^{\phi} \frac{1}{\sqrt{1 - \Eps \sin^2(\theta)}} \, \textnormal{d}\theta.
\end{align*}
Using the scaling invariance of~\eqref{KdV}, we may fix $\kappa = 1$. Then,~\eqref{e:elliptic_modulus_with_l},~\eqref{e:wave_with_kappa}, and~\eqref{e:def_l} yield the wave speed
\begin{align}
c = 4(2\Eps - 1) \label{e:normwavespeed}
\end{align}
and the profile
\begin{align}\label{e:wexplicit}
w(x) = 12 \Eps \, \textnormal{cn}^2(x;\Eps) = 12 \Eps \, \cos^2(\textnormal{am}(x;\Eps))
\end{align}
with elliptic modulus $\Eps \in (0,1)$ and period $\ell = 2 K(\Eps)$. 

Inserting~\eqref{e:wexplicit} into~\eqref{e:Euler_Lagrange_1},~\eqref{e:Euler_Lagrange_2}, and~\eqref{e:HamiltonianODE} and evaluating at $x = 0$, while using $\textnormal{am}(0;\Eps) = 0$, $\partial_x \textnormal{am}(0;\Eps) = 1$, $\partial_x^2 \textnormal{am}(0;\Eps) = 0$, and $\partial_x^3 \textnormal{am}(0;\Eps) = -\Eps$, we obtain expressions for the integration constants
\begin{align} \label{e:normintegrationconstants}
  c_1 = 24 \Eps (1 - \Eps), \qquad c_2 = 96 (1 - 2 \Eps) (1 - \Eps) \Eps,
\end{align}
in terms of the elliptic modulus $\Eps \in (0,1)$.

Finally, applying the symmetry transformations~\eqref{e:Galilean} to~\eqref{e:wexplicit}, we arrive at the four-parameter family~\eqref{e:cnoidal_form} of periodic traveling-wave solutions to~\eqref{KdV} with wave speed~\eqref{e:wavespeed}. Adapting the integration constants $c_{1,2}$ accordingly, we find that the associated profile $w$ satisfies the Euler--Lagrange equations~\eqref{e:Euler_Lagrange_1} and~\eqref{e:Euler_Lagrange_2} with
\begin{align*}
c_1 &= 24\Eps(1 - \Eps)\kappa^4 - h(8\Eps - 4)\kappa^2 - \frac12 h^2,\\
c_2 &= 96(1 - 2\Eps)(1 - \Eps)\Eps \kappa^6 + 16(1 - 6(1 - \Eps)\Eps)h\kappa^4 - \frac{20}{3}(1 - 2\Eps)h^2\kappa^2 + \frac59 h^3.
\end{align*}
That is, $w$ is for any $\gamma \in \R$ a critical point of the conserved energy $E_\gamma$, given by~\eqref{e:def_energy}, 

\section{Spectral analysis} \label{appendixSpectral}

Let $u_{\text{tw}}(x,t) = w(x-ct)$ be a periodic traveling-wave solution to the KdV equation~\eqref{KdV} with wave speed $c \in \R$ and profile $w\colon \R \to \R$ of fundamental period $\ell > 0$. Let $c_{1,2} \in \R$ be such that $w$ satisfies the Euler--Lagrange equations~\eqref{e:Euler_Lagrange_1} and~\eqref{e:Euler_Lagrange_2}. The second variation of the energy $E_\gamma$, defined in~\eqref{e:def_energy}, corresponds to the operator $\A_\gamma\colon H^4(\R)\to L^2(\R)$ given by~\eqref{e:defLgamma}. In this appendix, we show that there exists an open nonempty interval $I \subset \R$ such that $-\A_\gamma$ is diffusively spectrally stable for each $\gamma \in I$, thereby proving Theorem~\ref{thm:diffusive_intro}.

Exploiting the Galilean, translational, and scaling invariances of~\eqref{KdV}, we assume without loss of generality that $w$ attains its minimum value $0$ at $x=\pm \ell/2$. Consequently, as shown in Appendix~\ref{appendixPhPA}, there exists $\Eps \in (0,1)$ such that the wave speed $c$, the profile $w$, and the integration constants $c_{1,2}$ are given by~\eqref{e:normwavespeed},~\eqref{e:wexplicit}, and~\eqref{e:normintegrationconstants}, respectively.

In order to establish diffusive spectral stability of $-\A_\gamma$, we analyze the family of Bloch operators $\A_\gamma(\xi) \colon H^4_\per(0,\ell)\to L^2_\per(0,\ell)$, given by
\begin{align*}
\A_\gamma(\xi) &= (\partial_x + \iu \xi)^4 + \frac{5}{3}w(\partial_x + \iu \xi)^2+\frac{5}{3}w'(\partial_x + \iu \xi)+\frac{5}{3}w''+\frac{5}{6}w^2-\frac13c_1-c^2\\
&\qquad +\,\gamma(-(\partial_x + \iu\xi)^2+c-w)
\end{align*}
for $\xi \in [-\frac{\pi}{\ell},\frac{\pi}{\ell})$. Our strategy is to represent the action of $\A_\gamma(\xi)$ with respect to the biorthogonal families of eigenfunctions of $\El(\xi)$ and its adjoint $\El(\xi)^*$. Here, the operator $\El \colon H^3(\R) \to L^2(\R)$, given by
\begin{align*}
\El = \mathcal{J}A = \partial_x(-\partial_x^2+c-w),
\end{align*}
is the linearization of~\eqref{cKdV} about its stationary solution $w$, the operator $A \colon H^2(\R) \to L^2(\R)$ is given by~\eqref{e:defL1}, and $\El(\xi) \colon H^3_\per(0,\ell) \to L^2_\per(0,\ell)$ and $A(\xi) \colon H^2_\per(0,\ell) \to L^2_\per(0,\ell)$, given by
\begin{align*}
\El(\xi) = (\partial_x+\iu \xi)A(\xi), \qquad A(\xi) = -(\partial_x+\iu \xi)^2+c-w
\end{align*}
are the associated Bloch operators.

Using the so-called \emph{squared-eigenfunction connection} arising from the Lax-pair formulation of the KdV equation, the eigenfunctions of $\El(\xi)$ and $\El(\xi)^*$ were determined explicitly in~\cite{bottman2009kdv} and it was shown in~\cite{rodrigues2018linear} that they form biorthogonal families of Riesz bases. The following proposition collects the relevant details obtained in~\cite{bottman2009kdv,rodrigues2018linear}.

\begin{proposition}[Properties of Bloch eigenfunctions] \label{prop:eigenfunctions}
There exists a constant $C > 0$ such that for all $\xi \in \smash{[-\frac{\pi}{\ell},\frac{\pi}{\ell})} \setminus \{0\}$ the space $\smash{L^2_\per(0,\ell)}$ admits Riesz bases of eigenfunctions $\{\phi_j(\xi) : j \in \Z\}$ of $\El(\xi)$ and of eigenfunctions $\{\smash{\tilde{\phi}}_j(\xi) : j \in \Z\}$ of $\El(\xi)^*$ with the following properties:
\begin{itemize}
    \item[(i)] \emph{(Biorthogonality)}. For all $j,k \in \Z$ it holds 
    \begin{align}\langle \phi_j(\xi),\tilde{\phi}_k(\xi)\rangle_{L^2(0,\ell)} = \delta_{jk},
    \label{e:biorthogonal}
    \end{align} 
    where $\delta_{jk}$ is the Kronecker delta.
    \item[(ii)] \emph{(Low- and high-frequency behavior)}. We have
    \begin{align} \label{e:eigenfunctionslowhigh} 
    \|\xi \phi_j(\xi)\|_{L^2(0,\ell)}, \|\tilde{\phi}_j(\xi)\|_{L^2(0,\ell)} \leq C, \qquad 
    \|\phi_k(\xi)-\tilde{\phi}_k(\xi)\|_{L^2(0,\ell)} \leq \frac{C}{k^2}
    \end{align}
    for $j \in \{0,\pm1\}$ and $k \in \Z \setminus \{0,\pm 1\}$. 
    \item[(iii)] \emph{(Representation formulas)}. For each $j \in \Z$ there exists $\eta_j(\xi) \in (-\infty,4(\Eps -1)) \cup (c,4\Eps)$ such that 
    \begin{align} \label{e:eigenfunctions_explicit}
    \eu^{\iu \xi \cdot} \phi_j(\xi) = \partial_x\alpha(\eta_j(\xi)), \qquad \eu^{\iu \xi \cdot} \smash{\tilde{\phi}}_j(\xi) = \frac{\overline{\lambda(\eta_j(\xi))}}{\nu(\eta_j(\xi))} \, \alpha(\eta_j(\xi))
    \end{align}
    where
    \begin{align*}
    \alpha(x;\eta) = \frac{\eta-c +\frac13 w(x)}{\lambda(\eta)} \exp\left(-\int_0^x \frac{\lambda(\eta)}{\eta - c + \frac13 w(y)}\dy\right)
    \end{align*}
    and
    \begin{align*}
\qquad \nu(\eta) &= \int_0^\ell \eta - c + \frac13 w(y) \dy,\qquad \lambda(\eta)^2 = (\eta - 4(\Eps - 1))(\eta - c)(\eta - 4\Eps).
    \end{align*}
    In particular, it holds 
    \begin{align} \label{e:eigenfunctionid}
     \El(\xi) \phi_j(\xi) = \lambda(\eta_j(\xi)) \,\phi_j(\xi), \qquad \El(\xi)^* \tilde{\phi}_j(\xi) = \overline{\lambda(\eta_j(\xi))}\, \tilde{\phi}_j(\xi),
    \end{align}
    and
    \begin{align} \label{e:eigenfunctionid2}
    A(\xi) \phi_j(\xi) = -\nu(\eta_j(\xi)) \, \tilde{\phi}_j(\xi)
    \end{align}
    Finally, we have $\nu(\eta) < 0$ for $\eta \in (-\infty,4(\Eps-1)]$ and $\nu(\eta) > 0$ for $\eta \in [c,4\Eps]$. 
\end{itemize}
\end{proposition}
\begin{proof}
The formulas~\eqref{e:eigenfunctions_explicit} were derived in~\cite[Theorem~7.1]{bottman2009kdv}. Using identities~\eqref{e:normwavespeed} and~\eqref{e:normintegrationconstants} and employing the Euler--Lagrange equation~\eqref{e:Euler_Lagrange_1} and the first-order profile equation in~\eqref{e:HamiltonianODE}, one directly verifies identity~\eqref{e:eigenfunctionid2}. Using that $\lambda(\eta_j(\xi))$ is purely imaginary, this immediately implies~\eqref{e:eigenfunctionid}. In particular, $\lambda(\eta_j(\xi))$ is an eigenvalue of $\El(\xi)$ with associated eigenfunction $\phi_j(\xi)$, while $\smash{\overline{\lambda(\eta_j(\xi))}}$ is an eigenvalue of $\El(\xi)^*$ with associated eigenfunction $\smash{\tilde\phi}_j(\xi)$. Moreover, since $0 = w(\pm \ell/2) \leq w(x) \leq w(0) = 12\Eps$ for all $x \in \R$ by~\eqref{e:wexplicit}, it follows from~\eqref{e:normwavespeed} that $\nu(\eta) < 0$ for $\eta \in (-\infty,4(\Eps-1)]$ and $\nu(\eta) > 0$ for $\eta \in [c,4\Eps]$.

It was established in~\cite[Section~2.2 and Proposition~3.1]{rodrigues2018linear} that $\{\phi_j(\xi) : j \in \Z\}$ and $\{\smash{\tilde{\phi}}_j(\xi) : j \in \Z\}$ are biorthogonal families which form Riesz bases of $L^2_\per(0,\ell)$. Finally, the estimates~\eqref{e:eigenfunctionslowhigh} were obtained in~\cite[Proposition~2.1 and Lemma~2.3]{rodrigues2018linear}.
\end{proof}

A Riesz basis $\{f_j : j \in \Z\}$ of a Hilbert space $X$ is characterized by the existence of constants $C_1, C_2 > 0$ such that
\begin{align} \label{e:Riesz0}
C_1 \|z\|_X^2 \leq \sum_{j \in \Z} \left|\langle z, f_j \rangle_X\right|^2 \leq C_2 \|z\|_X^2
\end{align}
for all $z \in X$; see~\cite[Theorem~3.4.5]{Davies_2007_linear}. The following lemma describes how the lower bound in~\eqref{e:Riesz0} depends on $\xi$ for the Riesz bases of eigenfunctions of $\El(\xi)$ and $\El(\xi)^*$, established in Proposition~\ref{prop:eigenfunctions}.

\begin{lemma}[Lower Riesz bounds] \label{lem:Riesz}
Let $\{\phi_j(\xi) : j \in \Z\}$ and $\{\tilde{\phi}_j(\xi) : j \in \mathbb{Z}\}$ denote the Riesz bases of $L^2_\per(0,\ell)$ formed by eigenfunctions of $\El(\xi)$ and of its adjoint $\El(\xi)^*$, respectively, as obtained in Proposition~\ref{prop:eigenfunctions}. There exists a constant $\varsigma > 0$ such that
\begin{align} \label{e:Rieszlower}
\varsigma \xi^2 \|z\|_{L^2(0,\ell)}^2 \leq \sum_{j \in \Z} \left|\langle z, \tilde\phi_j(\xi) \rangle_{L^2(0,\ell)}\right|^2, \qquad \varsigma \|z\|_{L^2(0,\ell)}^2 \leq \sum_{j \in \Z} \left|\langle z, \phi_j(\xi) \rangle_{L^2(0,\ell)}\right|^2
\end{align}
for all $z \in L_\per^2(0,\ell)$ and $\xi \in \smash{[-\frac{\pi}{\ell}, \frac{\pi}{\ell})} \setminus \{0\}$.
\end{lemma}
\begin{proof}
Let $\xi \in [-\frac\pi\ell,\frac\pi\ell) \setminus \{0\}$ and $z \in L^2_\per(0,\ell)$. Using~\cite[Lemma~3.3.3]{Davies_2007_linear} to write
\begin{align} \label{e:basis_rep}
z = \sum_{j \in \Z} \langle z,\tilde{\phi}_j(\xi)\rangle_{L^2(0,\ell)} \phi_j(\xi),
\end{align}
we decompose 
\begin{align*}
\xi^2 \|z\|_{L^2(0,\ell)}^2 = \xi^2 \sum_{j \in \Z} \langle z,\tilde{\phi}_j(\xi)\rangle_{L^2(0,\ell)}\langle \phi_j(\xi),z\rangle_{L^2(0,\ell)} = I_1 + I_2
\end{align*}
into a high-frequency component
\begin{align*}
I_1 = \xi^2 \sum_{j \in \Z \setminus \{0,\pm 1\}} \left|\langle z,\tilde{\phi}_j(\xi)\rangle_{L^2(0,\ell)}\right|^2 + \xi^2 \sum_{j \in \Z \setminus \{0,\pm 1\}} \langle\phi_j(\xi) - \tilde{\phi}_j(\xi),z\rangle_{L^2(0,\ell)} \langle z,\tilde{\phi}_j(\xi)\rangle_{L^2(0,\ell)}
\end{align*}
and a low-frequency component
\begin{align*}
I_2 = \xi^2 \sum_{j \in \{0,\pm 1\}} \langle z,\tilde{\phi}_j(\xi)\rangle_{L^2(0,\ell)}\langle \phi_j(\xi),z\rangle_{L^2(0,\ell)}.
\end{align*}
The Cauchy--Schwarz and H\"older inequalities and the high-frequency bound from Proposition~\ref{prop:eigenfunctions} yield
\begin{align*}
|I_1| \leq \frac{\pi^2}{\ell^2} \sum_{j \in \Z} \left|\langle z,\tilde{\phi}_j(\xi)\rangle_{L^2(0,\ell)}\right|^2 + \frac{C M \pi}{\ell} \left(\sum_{j \in \Z} \left|\langle z,\tilde{\phi}_j(\xi)\rangle_{L^2(0,\ell)}\right|^2\right)^{\frac12} |\xi| \|z\|_{L^2(0,\ell)},
\end{align*}
where we denote $M = \sum_{j \in \Z \setminus \{0\}} j^{-4}$. On the other hand, using the low-frequency bound $\|\xi \phi_j(\xi)\|_{L^2(0,\ell)} \leq C$ for $j = 0,\pm 1$ from Proposition~\ref{prop:eigenfunctions}, we obtain
\begin{align*}
|I_2| \leq \sqrt{3} \, C \left(\sum_{j \in \Z} \left|\langle z,\tilde{\phi}_j(\xi)\rangle_{L^2(0,\ell)}\right|^2\right)^{\frac12} |\xi| \|z\|_{L^2(0,\ell)}.
\end{align*}
Combining the estimates on $I_1$ and $I_2$ and applying Young's inequality, we obtain the first estimate in~\eqref{e:Rieszlower}. The second estimate in~\eqref{e:Rieszlower} follows analogously by interchanging the roles of $\phi_j(\xi)$ and $\tilde{\phi}_j(\xi)$ and using the bound $\|\tilde{\phi}_j(\xi)\|_{L^2(0,\ell)} \leq C$ for $j \in  0,\pm 1$ instead.
\end{proof}

Finally, we prove that $-\A_\gamma$ is diffusively spectrally stable by representing the action of the Bloch operator $\A_\gamma(\xi)$ with respect to the biorthogonal families of eigenfunctions of $\El(\xi)$ and $\El(\xi)^*$.

\begin{proof}[Proof of Theorem~\ref{thm:diffusive_intro}]
Proceeding as in~\cite{deconinck2010orbital}, we compute $\A_\gamma(\xi)\phi_j(\xi)$ for $j \in \Z$ and $\xi \in [-\frac\pi\ell,\frac\pi\ell) \setminus \{0\}$, where $\phi_j(\xi)$ is the eigenfunction of $\El(\xi)$ obtained in Proposition~\ref{prop:eigenfunctions}. Thus, using the representation formula~\eqref{e:eigenfunctions_explicit} for $\phi_j(\xi)$, applying the identities~\eqref{e:normwavespeed} and~\eqref{e:normintegrationconstants}, and employing the first-order profile equation in~\eqref{e:HamiltonianODE} and the Euler--Lagrange equations~\eqref{e:Euler_Lagrange_1} and~\eqref{e:Euler_Lagrange_2} to rewrite first, second, and fourth derivatives of $w$, we arrive at
\begin{align} \label{e:action_Elgamma}
    \A_\gamma(\xi) \phi_j(\xi) = \beta(\eta_j(\xi),\gamma) \nu(\eta_j(\xi))\tilde{\phi}_j(\xi), \qquad \beta(\eta,\gamma) \coloneqq \eta - \gamma + c 
\end{align}
for all $j \in \Z$ and $\xi \in [-\frac\pi\ell,\frac\pi\ell) \setminus \{0\}$. Recall from Proposition~\ref{prop:eigenfunctions} that $\nu(\eta)$ is an affine linear function of $\eta$ satisfying $\nu(\eta) < 0$ for $\eta \in (-\infty,4(\Eps-1)]$ and $\nu(\eta) > 0$ for $\eta \in [c,4\Eps]$. Hence, using~\eqref{e:normwavespeed}, we find that, for each $\gamma \in (4(3\Eps - 2),4(4\Eps - 2))$, there exists a constant $\theta_*(\gamma) > 0$ such that
\begin{align} \label{e:betajlower}
\beta(\eta,\gamma)\nu(\eta), \frac{\beta(\eta,\gamma)}{\nu(\eta)} \geq \theta_*(\gamma)
\end{align}
for all $\eta \in (-\infty,4(\Eps-1)] \cup [c,4\Eps]$.

Let $\gamma \in (4(3\Eps - 2),4(4\Eps - 2))$. Using the representation~\eqref{e:basis_rep} and applying Lemma~\ref{lem:Riesz}, identity~\eqref{e:action_Elgamma}, and estimate~\eqref{e:betajlower}, we establish
\begin{align} \label{e:coerc_proof}
\langle \A_\gamma(\xi) z, z\rangle_{L^2(0,\ell)} = \sum_{j \in \Z} \beta(\eta_j(\xi),\gamma)\nu(\eta_j(\xi)) \left|\langle z,\tilde{\phi}_j(\xi)\rangle_{L^2(0,\ell)}\right|^2 
\geq \varsigma \theta_*(\gamma) \xi^2 \|z\|_{L^2(0,\ell)}^2 
\end{align}
for all $z \in H^4_\per(0,\ell)$ and $\xi \in [-\frac\pi\ell,\frac\pi\ell) \setminus \{0\}$. Similarly, Proposition~\ref{prop:eigenfunctions}, Lemma~\ref{lem:Riesz}, identity~\eqref{e:action_Elgamma}, and estimate~\eqref{e:betajlower} yield
\begin{align*}
\langle \A_\gamma(\xi) z, z\rangle_{L^2(0,\ell)} = \sum_{j \in \Z} \frac{\beta(\eta_j(\xi),\gamma)}{\nu(\eta_j(\xi))} \left|\langle A(\xi) z, \phi_j(\xi)\rangle_{L^2(0,\ell)}\right|^2 \geq \varsigma \theta_*(\gamma) \, \|A(\xi) z\|_{L^2(0,\ell)}^2 
\end{align*}
for $z \in H^4_\per(0,\ell)$ and $\xi \in [-\frac\pi\ell,\frac\pi\ell) \setminus \{0\}$. Taking the limit $\xi \to 0$ at both sides of the latter inequality, it follows that $\ker(\A_\gamma(0)) \subset \ker(A(0))$. On the one hand, we have $\A_\gamma(0) w' = 0$, since $w$ is a stationary solution to~\eqref{e:KdV2}. On the other hand, it has been shown in~\cite[Theorem~5.1]{angulo2006stability} that the kernel of $A(0)$ is spanned by $w'$. We conclude that $0$ is a simple eigenvalue of the self-adjoint operator $\A_\gamma(0)$ with eigenfunction $\mathcal{J}w = w'$. Combining this with~\eqref{e:coerc_proof} shows that $-\A_\gamma$ is diffusively spectrally stable.
\end{proof}

\section{Derivation of modulated perturbation equation} \label{appendixPertEq}

We provide the details of the derivation of the equation~\eqref{e:modperteq} for the inverse-modulated perturbation~\eqref{e:def_inv_mod_2}.

Let $t \in [0,T)$. Taking spatial derivatives  of~\eqref{e:def_inv_mod_2}, we obtain
\begin{align} \label{e:uxx}
\begin{split}
        u_x(\cdot+\psi(\cdot,t),t) &= \left( \vt_x(t)+w' \right) \mathcal{R}_1(\psi_x(t)),\\ 
        u_{xx}(\cdot+\psi(\cdot,t),t) &= \left(\vt_{xx}(t)+w''\right)\mathcal{R}_2(\psi_x(t))-\left(\vt_x(t)+w'\right)\mathcal{R}_3(\psi_x(t))\psi_{xx}(t),
\end{split}
\end{align}
where we recall that $\mathcal{R}_k$ is given by~\eqref{e:defR_k}. On the other hand, taking the temporal derivative of~\eqref{e:def_inv_mod_2} and using that $u(t)$ solves~\eqref{cKdV}, we establish
\begin{align*}
\vt_t(t) &= u_t(\cdot + \psi(\cdot,t),t) + u_x(\cdot + \psi(\cdot,t),t)\psi_t(t)\\
&= \partial_x \bigg(cu(\cdot + \psi(\cdot,t),t)-\frac{1}{2}u(\cdot + \psi(\cdot,t),t)^2 - u_{xx}(\cdot + \psi(\cdot,t),t) \bigg)\mathcal{R}_1(\psi_x(t))\\ 
&\qquad + \, u_x(\cdot + \psi(\cdot,t),t)\psi_t(t).
\end{align*}
Substituting~\eqref{e:uxx} into the latter and suppressing dependency on $t$ yields
\begin{align*}
       \vt_t &= \partial_x \bigg( c(\vt+w)-\frac{1}{2}(\vt+w)^2 - \left(\vt_{xx}+w''\right)\mathcal{R}_2(\psi_x)+ \left(\vt_x+w'\right)\mathcal{R}_3(\psi_x)\psi_{xx} \bigg)\mathcal{R}_1(\psi_x)\\
       &\qquad + \left( \vt_x+w' \right) \mathcal{R}_1(\psi_x) \psi_t. 
 \end{align*}
Subsequently, we use that $w$ solves~\eqref{cKdV}, resulting in
\begin{align*}
       \vt_t &= \partial_x \bigg( c \vt - \vt w - \frac12 \vt^2 - \vt_{xx} \mathcal{R}_2(\psi_x) - w'' (\mathcal{R}_2(\psi_x) - \mathcal{R}_2(0))\\ 
       &\qquad + \left(\vt_x+w'\right)\mathcal{R}_3(\psi_x)\psi_{xx} \bigg)\mathcal{R}_1(\psi_x) + \left( \vt_x+w' \right) \mathcal{R}_1(\psi_x) \psi_t. 
 \end{align*}
Finally, expanding $\mathcal{R}_1$, $\mathcal{R}_2$, and $\mathcal{R}_3$, we collect all linear terms in $\vt$ and $\psi$ on the left-hand side and use that $\El w' = 0$. This leads to
\begin{align*}
(\partial_t - \El)(\vt - w' \psi) &= (\partial_t - \El)\vt - w' \psi_t - \partial_x(w'\psi_{xx} + 2 w'' \psi_x)\\ 
&= \partial_x \mathcal{S}_1(\vt,\psi_x) + 
\partial_x \mathcal{S}_2(\vt,\psi_x) \left(\mathcal{R}_1(\psi_x) - \mathcal{R}_1(0)\right) + \vt_x \mathcal{R}_1(\psi_x) \psi_t\\
&\qquad + w' (\mathcal{R}_1(\psi_x) - \mathcal{R}_1(0))\psi_t,
\end{align*}
where $\mathcal{S}_1(\vt,\psi_x)$ and $\mathcal{S}_2(\vt,\psi_x)$ are given by~\eqref{e:defS}. Writing $\partial_x \mathcal{S}_2(\vt,\psi_x) \left(\mathcal{R}_1(\psi_x) - \mathcal{R}_1(0)\right) = \partial_x \left(\mathcal{S}_2(\vt,\psi_x) \left(\mathcal{R}_1(\psi_x) - \mathcal{R}_1(0)\right)\right) + \mathcal{S}_2(\vt,\psi_x) \mathcal{R}_2(\psi_x)\psi_{xx}$, we finally arrive at~\eqref{e:modperteq}. 

\section{Technical lemmas} \label{appendixtech}

In this appendix, we establish two technical lemmas that are used in our nonlinear stability analysis. The first result shows that modulation by a function with sufficiently small derivative is continuous.

\begin{lemma}[Continuity of modulation map] \label{lem:cont_mod}
Set $X_1 = \{\psi \in H^2(\R) \colon \|\psi'\|_{L^\infty(\R)} \leq \smash{\frac{1}{2}}\}$. The map $\Phi \colon L^2(\R) \times X_1 \to L^2(\R)$, given by
\begin{align*}
\Phi(v,\psi) = v(\cdot +  \psi(\cdot)),   
\end{align*}
is continuous.
\end{lemma}
\begin{proof}
Let $(v_0,\psi_0) \in L^2(\R) \times X_1$ and $\eps > 0$. Take $v_* \in C^{\infty}_{\textnormal{c}}(\R)$ with $\|v_0-v_*\|_{L^2(\R)} \leq \eps$. Let $(v,\psi) \in L^2(\R) \times X_1$ with $\|v-v_0\|_{L^2(\R)} + \|\psi-\psi_0\|_{H^1(\R)} \leq \eps$. Since $\|\psi'\|_{L^\infty(\R)} \leq \frac12$, the map $\R \to \R, x \mapsto x + \psi(x)$ is invertible. Using integration by substitution, we establish
\begin{align*}
\|v(\cdot + \psi(\cdot)) - v_*(\cdot + \psi(\cdot))\|_{L^2(\R)}^2 \leq \frac{\|v-v_*\|_{L^2(\R)}^2}{1 - \|\psi'\|_{L^\infty(\R)}} \leq 8\eps^2,
\end{align*}
and, similarly,
\begin{align*}
\|v_*(\cdot + \psi_0(\cdot)) - v_0(\cdot + \psi_0(\cdot))\|_{L^2(\R)} \leq 2\eps.
\end{align*}
On the other hand, the mean-value theorem implies
\begin{align*}
\|v_*(\cdot + \psi_0(\cdot)) - v_*(\cdot+\psi(\cdot))\|_{L^2(\R)} \leq \|v_*'\|_{L^\infty(\R)} \|\psi - \psi_0\|_{L^2(\R)} \leq \|v_*'\|_{L^\infty(\R)} \eps.
\end{align*}
Combining the last three estimates yields the result.
\end{proof}

The following result shows that forward and inverse modulation induce equivalent norms. Related bounds appear in~\cite{zumbrun_2024_forward,johnson_2014_behavior}.

\begin{lemma}[Equivalence of forward and inverse modulation] \label{lem:equivalence}
Fix $k \in \N$ and $R > 0$. There exists a constant $C > 0$ such that for each $\psi \in H^{k+1}(\R)$ satisfying
\begin{align} \label{e:psi_cond}
\|\psi'\|_{H^k(\R)} \leq R, \qquad \|\psi'\|_{L^\infty(\R)} \leq \frac12,
\end{align}
the map $\mathrm{id} + \psi \colon \R \to \R$ is invertible with
\begin{align} \label{e:norm_equivalence}
\left\|f \circ (\mathrm{id} + \psi)^{-1} \right\|_{L^2(\R)} &\leq C \|f\|_{L^2(\R)}
\end{align}
for $f \in L^2(\R)$, and
\begin{align}\label{e:norm_equivalence2}
\left\|\partial_x^k \left(f \circ (\mathrm{id} + \psi)^{-1}\right) \right\|_{L^2(\R)} &\leq C \|f'\|_{H^{k-1}(\R)}
\end{align}
for $f \in H^k(\R)$.
\end{lemma}
\begin{proof}
Let $\psi \in H^{k+1}(\R)$  satisfy~\eqref{e:psi_cond}. Then, $\mathrm{id} + \psi \colon \R \to \R$ is strictly increasing and therefore invertible. Consequently, using integration by substitution and the bound $\|\psi'\|_{L^\infty(\R)} \leq \frac12$, we obtain~\eqref{e:norm_equivalence}. Similarly, combining integration by substitution, the chain rule, and~\eqref{e:psi_cond}, together with the continuous embedding $H^1(\R) \hookrightarrow L^\infty(\R)$, yields a constant $C>0$, depending only on $k$ and $R$, such that~\eqref{e:norm_equivalence2} holds for all $f \in H^k(\R)$.  
\end{proof}

\bigskip

\subsection*{Data availability statement} The numerical simulations displayed in Figure~\ref{fig2} have been obtained using Matlab (Version 24.1.0 R2024a), where the initial condition has been constructed with the aid of Mathematica (Version 14.3.0.0). The code is available through the
link \url{https://www.waves.kit.edu/downloads/CRC1173_Preprint_2026-28_Codes.zip}. 

\printbibliography
\end{document}